\documentclass{amsart}

\usepackage{amssymb,amsmath,amscd,graphicx,amsthm,color,bm,tikz-cd,mathtools,arydshln}
\usepackage[center]{caption}
\newtheorem{theorem}{Theorem}
\newtheorem{lemma}[theorem]{Lemma}
\newtheorem{proposition}[theorem]{Proposition}

\theoremstyle{definition}

\theoremstyle{remark}
\newtheorem{remark}[theorem]{Remark}

\numberwithin{equation}{section}
\numberwithin{theorem}{section}

\def\B{{\mathcal B}}
\def\BD{G}

\def\CC{{\mathcal C}}

\def\FF{{\mathcal F}}

\def\G{{\mathcal G}}

\def\H{\mathcal H}
\def\J{\mathbf J}

\def\KK{{\mathcal K}}
\def\L{{\mathcal L}}

\def\N{\mathcal N}

\def\P{{\mathcal P}}

\def\W{{\mathcal W}}

\def\ZZ{{\mathcal Z}}

\def\bfG{\mathbf \Gamma}
\def\bfGr{{\bfG^{{\rm r}}}}
\def\bfGc{{\bfG^{{\rm c}}}}
\def\tbfGr{\tilde\bfG^{{\rm r}}}

\def\tGamma{\tilde\Gamma}

\def\wog{W}
\def\pu{\varnothing}

\def\b{\mathfrak b}

\def\bgamma{\pmb\gamma}

\def\bgammar{{\bgamma^\er}}

\def\bgammac{{{\bgamma^\ec}}}

\def\ec{{\rm c}}

\def\er{{\rm r}}

\def\ttf{{\tt f}}

\def\g{\mathfrak g}
\def\gammar{{\gamma^\er}}
\def\gammac{{\gamma^\ec}}
\def\tgammar{{\tilde\gamma^\er}}

\def\h{\mathfrak h}
\def\i{\mathbf i}
\def\ii{{\hat\imath}}
\def\jj{\hat\jmath}

\def\n{\mathfrak n}
\def\one{\mathbf 1}

\def\pp{{\mathfrak p}}

\def\zero{\mathbf 0}

\def\Ad{\operatorname{Ad}}

\def\End{\operatorname{End}}

\def\Kil{\langle \cdot,\cdot\rangle}

\def\Poi{{\{\cdot,\cdot\}}}

\def\ad{\operatorname{ad}}
\def\adj{\dag}

\def\deg{{\operatorname{deg}}}
\def\diag{\operatorname{diag}}

\def\id{\operatorname{id}}

\def\sign{{\operatorname{sign}}}

\def\:{{:\ }}

\def\ze{\texttt{=}}

\begin{document}

\title[unified approach to exotic cluster structures on simple Lie groups]
{A unified approach to exotic cluster structures on simple Lie groups}

\author{Misha Gekhtman}

\address{Department of Mathematics, University of Notre Dame, Notre Dame,
IN 46556}
\email{mgekhtma@nd.edu}

\author{Michael Shapiro}
\address{Department of Mathematics, Michigan State University, East Lansing,
MI 48823}
\email{mshapiro@math.msu.edu}

\author{Alek Vainshtein}
\address{Department of Mathematics \& Department of Computer Science, University of Haifa, Haifa,
Mount Carmel 31905, Israel}
\email{alek@cs.haifa.ac.il}

\begin{abstract}
We propose a new approach to building log-canonical coordinate charts for any simply-connected simple
Lie group $\G$  and arbitrary Poisson-homogeneous bracket on $\G$ associated with Belavin--Drinfeld data. Given a pair of representatives $r, r'$ from two arbitrary Belavin--Drinfeld classes, 
we build a rational map from  $\G$ with the Poisson structure defined by two appropriately selected representatives from the standard class  
to  $\G$ equipped with the Poisson structure defined by the pair $r, r'$. In the $A_n$ case, we prove that this map is invertible whenever the pair $r, r'$ is
drawn from aperiodic Belavin--Drinfeld data, as defined in~\cite{GSVple}. We further
apply this construction to recover the existence of a regular complete cluster structure
compatible with the Poisson structure associated with the pair $r, r'$ in the aperiodic case.
\end{abstract}


\maketitle

\section{Introduction}
Shortly after cluster algebras were discovered by Fomin and Zelevinsky, important ties emerged between the new theory and Poisson geometry. As was first observed in~\cite{GSV1} and then expounded upon in~\cite{GSVb}, cluster algebras carry natural Poisson structures compatible with cluster transformations. This, in turn, helps in uncovering cluster structures in rings of regular functions on Poisson varieties of interest in Lie theory. In particular, the cluster structure constructed 
in~\cite{BFZ} for (double Bruhat cells of) a simply-connected simple Lie group $\G$ was shown 
in~\cite[Ch.~4.3]{GSVb} to be compatible with the standard Poisson--Lie structure on $\G$. This led to a question, posed in~\cite{GSVM}, of existence of what we called {\em exotic\/} cluster structures on $\G$, i.e. cluster structures non-isomorphic to the standard one and compatible with other 
Poisson--Lie brackets. Although the answer to this question is negative in general---an example to that effect was constructed in~\cite{GSVM} in the case of $SL_2$---we conjectured that the answer is affirmative in the case of Poisson–-Lie structures corresponding to quasi-triangular solutions of  the classical Yang–-Baxter equation classified by Belavin and Drinfeld in~\cite{BD}. Up to an automorphism, each such solution, called an $r$-matrix, is parametrized by discrete data consisting of an isometry between two subsets of positive roots in the root system of the Lie algebra of $\G$ and a continuous parameter that can be described as an element of the tensor square of the Cartan subalgebra that satisfies a system of linear equations governed by the discrete data. The discrete data determines a {\em Belavin--Drinfeld class\/} of $r$-matrices and corresponding Poisson--Lie brackets, and continuous data specifies 
a particular $r$-matrix and bracket within this class. Given two such brackets on $\G$ associated with representatives of two Belavin--Drinfeld classes, one can define a Poisson--Lie group 
$\G\times \G$ equipped with the direct product Poisson structure and then construct a Poisson-homogeneous structure on $\G$ with respect to the action of  $\G\times \G$ by right and left multiplication.

The conjecture of~\cite{GSVM} was modified in subsequent publications. Most importantly, 
in~\cite{GSVple} we restated it to include not just Poisson--Lie brackets but also Poisson-homogeneous brackets of the kind described above. It now claims that for any such bracket associated with an arbitrary pair of Belavin--Drinfeld data there exists a {\em compatible regular complete\/}, possibly generalized, cluster structure in the ring of regular functions on $\G$. In~\cite{GSVple}, 
we proved this conjecture for a large class of Belavin--Drinfeld data in $SL_n$ called 
{\em aperiodic\/} and {\em oriented}. Generalized cluster structures are not needed in this case. They arise when the aperiodicity condition is not satisfied, and a conjectural but supported by examples construction for this situation was outlined in~\cite{GSVpest}.

The most crucial and, as a rule, the most difficult step in constructing a cluster structure compatible with a Poisson bracket is finding an initial coordinate chart consisting of regular functions with particularly simple Poisson brackets between them, so-called {\em log-canonical\/} coordinates. In~\cite{GSVple}, this goal was accomplished in an {\em ad hoc\/} way, with the choice of functions in the chart motivated by their invariance properties with respect to the action of certain subgroups specified by the data, and with Poisson relations between these functions established via lengthy and cumbersome computations. It was also not immediately clear how to adapt these computations to verify our conjecture for other Lie types. In contrast, the standard cluster structure of Berenstein--Fomin--Zelevinsky~\cite{BFZ} is described in purely Lie-theoretic terms using generalized minors and combinatorics of the Weyl group, and its compatibility with the standard Poisson--Lie structure was also verified in Lie type independent way in~\cite{GSVb}.

In this paper, we propose a new approach to building log-canonical coordinate charts for any  
simply-connected simple Lie group $\G$  and arbitrary Belavin--Drinfeld data. 
The main ingredient is a rational Poisson map $h^{r,r'}$ between two copies of $\G$ endowed with two different Poisson-homogeneous structures. One is $\Poi_{r,r'}$ determined by a pair of $r$-matrices from two arbitry Belavin--Drinfeld classes. The other, which we denote here by
$\Poi_{r,r'}^{\rm st}$, corresponds to two $r$-matrices from the standard Belavin--Drinfeld class whose Cartan parts match those of $r, r'$. The rational map $h^{r,r'}$ maps 
$(\G, \Poi_{r,r'}^{\rm st})$ to $(\G, \Poi_{r,r'})$.
Existence of such a map is a new result interesting in its own right, and we will explore its applications in the Poisson--Lie theory, in particular, to integrable systems on Poisson--Lie groups, in our future work. In the context of construction of an initial seed for a cluster structure compatible with $\Poi_{r,r'}$, the map's utility is that by inverting $h^{r,r'}$ and using the inverse to pull back any of the clusters in the standard cluster structure on $\G$, one obtains a log-canonical parametrization for $(\G, \Poi_{r,r'})$.
In particular, when $h^{r,r'}$ has a rational inverse, one can build a regular log-canonical coordinate chart this way and then use it as an initial seed for a cluster structure. We illustrate this point in Section 4, where we use the current approach not only to recover all the results 
of~\cite{GSVple} in a much more conceptual way, but also to drop the orientability condition which was imposed in~\cite{GSVple} and which does not appear to be natural in a general Lie-theoretic framework. The aperiodicity condition is retained, however, since it is precisely the one that guarantees that the map $h^{r,r'}$ has a rational inverse. If this condition is not satisfied, finding the inverse involves considering certain polynomials in one variable whose roots allow to restore frozen variables for the standard cluster structure in terms of elements of 
$(\G, \Poi_{r,r'})$ and whose coefficients serve as coefficients for {\em generalized exchange relations\/} in a compatible generalized cluster structure on $(\G, \Poi_{r,r'})$. We do not discuss these results in the current paper and reserve them for future publications.

The paper is organized as follows. Section~2 contains a brief overview of the necessary background, including Berenstein--Fomin--Zelevinsky factorization parameters in simple Lie groups and 
Poisson--Lie groups and Poisson-homogeneous structures on simple Lie groups arising from the 
Belavin--Drinfeld classification of quasi-triangular $r$-matrices. The main result of the 
paper---construction of the rational Poisson map described above---is  presented in Section~3 
(Theorem~\ref{twosidedpoisson}). Section~4 deals with the case $\G=SL_n$. Here, we show that in the case of  {\em aperiodic\/} Belavin--Drinfeld data (the notion we introduced in~\cite{GSVple}), the Poisson map of Section~3 has a rational inverse. Explicit formulas for the inverse are obtained in terms of minors forming an initial cluster for the {\em standard\/} cluster structure on $\G=SL_n$ 
(Theorem~\ref{bigthroughsmall}).  These formulas allow us to construct a regular complete cluster structure compatible with the Poisson structure associated with a pair of representatives $r, r'$
 from two arbitrary Belavin--Drinfeld classes satisfying the aperiodicity condition 
(Theorems~\ref{quiver},~\ref{regneighbour}, and~\ref{laurent}).  

Our research was supported in part by the NSF research grants DMS \#1702054  and DMS \#2100785 and by the 2022, 2023 Mercator Research Fellowship, Heidelberg University (M.~G.), NSF research grants DMS \#1702115 and  DMS \#2100791 (M.~S.), and ISF grant \#876/20 (A.~V.).  

The authors would like to thank the Isaac Newton Institute for Mathematical Sciences, Cambridge, for support and hospitality during the programme ``Cluster algebras and representation theory"  (Fall 2021) where  the work on this paper was conceived. This programme was supported by EPSRC 
grant no EP/R014604/1. In addition, M.G. thanks the members of the CCBC for their support during the stay in Cambridge. While working further on this project, we benefited from support from several institutions and programs: Mathematical Institute of the University of Heidelberg (M.~S., A.~V., Summer 2022), University of Haifa (M.~G., M.~S., Summer 2022), Michigan State University (A.~V., Fall 2022), University of Notre Dame (A.~V., Spring 2023), Max Planck Institute for Mathematics, Bonn (A.~V., Spring  2023), and Research in Pairs Program at the Mathematisches Forschungsinstitut Oberwolfach 
(M.~G., M.~S., A.~V., Summer 2023), where the project was completed. 
  We are grateful to all these institutions for their hospitality and outstanding working conditions they provided. Special thanks are due to Vladimir Hinich and Anna Melnikov for valuable 
	discussions.

\section{Preliminaries}\label{sec:prelim}
\subsection{Factorizations in Lie groups}\label{sec:factor} 
Let $\G$ be a semsimple complex Lie group of rank $r$, $\g$ be its Lie algebra with the Cartan decomposition $\g=\n_+\oplus\h\oplus\n_-$, $e_i, h_i, f_i$, $i\in[1,r]$ be the standard generators of $\g$. We denote by $\b_+=\n_+\oplus\h$ the Borel subalgebra of $\g$ and by $\b_-=\n_-\oplus\h$ the opposite Borel subalgebra. The corresponding subgroups in $\G$ are denoted $\N_+$, $\H$, $\N_-$, $\B_+$, and $\B_-$. 

Let $\W$ be the Weyl group of $\G$; it is generated by simple reflections $s_1,\dots,s_r$. A reduced word for $w\in\W$ is a sequence of indices $\i=(i_1,\dots,i_m)$ of the shortest possible length such that $w=s_{i_1}\cdots s_{i_m}$. Following~\cite{BZ}, for any reduced word $\i$ for the longest element $w_0\in\W$ we can write a generic element $N\in\N_+$ in a unique way as a product $N=x_{i_1}(t_1)\cdots x_{i_m}(t_m)$ where $t_i$ are nonzero complex numbers and $x_i(t)=\exp(te_i)$. A similar factorization with $x_i(t)$ replaced by $\exp(tf_i)$ holds for a generic element in $\N_-$.

Let $J\subset[1,r]$, $\W_J$ be the subgroup of $\W$ generated by reflections $s_j$, $j\in J$, 
$\W^J=\{w\in\W\: l(ws_j)>l(w)\ \text{for any $j\in J$}\}$ be the quotient. 
By~\cite[Prop.~2.4.4]{BB}, every $w\in\W$ has a unique factorization $w=w^J\cdot w_J$ such that $w^J\in\W^J$, $w_J\in \W_J$, and
$l(w)=l(w^J)+l(w_J)$. We apply this result to $w_0$ and rewrite the reduced word $\i$ as the concatenation of the reduced words $\i^J$ and $\i_J$. Consequently, this yields a factorization of an arbitrary element $N\in\N+$ as
$N= N' N''$ where $N''$ belongs to the unipotent subgroup $\N_+^J$ that corresponds to $J$. Similarly, 
an element $\tilde N\in \N_-$ can be factored as
$\tilde N=\tilde N''\tilde N'$ with $\tilde N''\in \N^J_-$.

\subsection{Poisson--Lie groups}
A reductive complex Lie group $\G$ equipped with a Poisson bracket $\Poi$ is called a {\em Poisson--Lie group\/}
if the multiplication map $\G\times \G \ni (X,Y) \mapsto XY \in \G$
is Poisson. Perhaps, the most important class of Poisson--Lie groups
is the one associated with quasitriangular Lie bialgebras defined in terms of  {\em classical R-matrices\/} 
(see, e.~g., \cite[Ch.~1]{CP}, \cite{r-sts} and \cite{Ya} for a detailed exposition of these structures).

Let $\g$ be the Lie algebra corresponding to $\G$, $\Kil$ be an invariant nondegenerate form on $\g$,
 and let $\mathfrak{t}\in \g\otimes\g$ be the corresponding Casimir element.
For an arbitrary element $r=\sum_i a_i\otimes b_i\in\g\otimes\g$ denote
\[
[[r,r]]=\sum_{i,j} [a_i,a_j]\otimes b_i\otimes b_j+\sum_{i,j} a_i\otimes [b_i,a_j]\otimes b_j+
\sum_{i,j} a_i\otimes a_j\otimes [ b_i,b_j]
\]
and $r^{21}=\sum_i b_i\otimes a_i$.
A {\em classical R-matrix} is an element $r\in \g\otimes\g$ that satisfies
{\em the classical Yang-Baxter equation (CYBE)\/} $[[r, r]] =0$
together with the condition $r + r^{21} = \mathfrak{t}$.
The Poisson--Lie bracket on $\G$ that corresponds to $r$ can be written as
\begin{equation}\label{sklyabra}
\begin{aligned}
\{f_1,f_2\}_r &= \langle R_+(\nabla^L f_1), \nabla^L f_2 \rangle - \langle R_+(\nabla^R f_1), \nabla^R f_2 \rangle\\
&= \langle R_-(\nabla^L f_1), \nabla^L f_2 \rangle - \langle R_-(\nabla^R f_1), \nabla^R f_2 \rangle,
\end{aligned}
\end{equation} 
where $R_+,R_- \in \End \g$ are given by $\langle R_+ \eta, \zeta\rangle = \langle r, \eta\otimes\zeta \rangle$, 
$-\langle R_- \zeta, \eta\rangle = \langle r, \eta\otimes\zeta \rangle$ for any $\eta,\zeta\in \g$ and  
$\nabla^L$, $\nabla^R$ are the right and the left gradients of functions on $\G$ with respect to $\Kil$ 
defined by
\begin{equation*}
\left\langle \nabla^R f(X),\xi\right\rangle=\left.\frac d{dt}\right|_{t=0}f(e^{t\xi}X),  \quad
\left\langle \nabla^L f(X),\xi\right\rangle=\left.\frac d{dt}\right|_{t=0}f(Xe^{t\xi})
\end{equation*}
for any $\xi\in\g$, $X\in\G$.

The classification of classical R-matrices for simple complex Lie groups was given by Belavin and Drinfeld in \cite{BD}.
Let $\G$ be a simple complex Lie group, $\Phi$ be the root system associated with its Lie algebra $\g$, $\Phi^+$ be the set of positive roots, and $\Pi\subset \Phi^+$ be the set of positive simple roots. 
A {\em Belavin--Drinfeld triple} $\bfG=(\Gamma_1,\Gamma_2, \gamma)$ (in what follows, a {\em BD triple\/})
consists of two subsets $\Gamma_1,\Gamma_2$ of $\Pi$ and an isometry $\gamma\:\Gamma_1\to\Gamma_2$ nilpotent in the 
following sense: for every $\alpha \in \Gamma_1$ there exists $m\in\mathbb{N}$ such that $\gamma^j(\alpha)\in \Gamma_1$ 
for $j\in [0,m-1]$, but $\gamma^m(\alpha)\notin \Gamma_1$.

 The isometry $\gamma$ yields an isomorphism, also denoted by $\gamma$, between the Lie subalgebras $\g^{\Gamma_1}$ 
and $\g^{\Gamma_2}$ that correspond to $\Gamma_1$ and $\Gamma_2$. It is uniquely defined by the property 
$\gamma e_\alpha = e_{\gamma(\alpha)}$ for $\alpha\in \Gamma_1$, where $e_\alpha$ is the Chevalley generator corresponding to 
the root $\alpha$. The isomorphism $\gamma^*\: \g^{\Gamma_2} \to \g^{\Gamma_1}$ is defined as the adjoint to $\gamma$ with respect to the form $\Kil$. 
It is given by $\gamma^* e_{\gamma(\alpha)}=e_{\alpha}$ for $\gamma(\alpha)\in \Gamma_2$.
 Both $\gamma$ and $\gamma^*$ can be extended to maps of $\g$ to itself by applying first the orthogonal
projection on $\g^{\Gamma_1}$ (respectively, on $\g^{\Gamma_2}$) with respect to $\Kil$; clearly, the extended
maps remain adjoint to each other. Note that the restrictions of $\gamma$ and $\gamma^*$ to the positive and the negative nilpotent subalgebras $\n_+$ and $\n_-$ of $\g$ are Lie algebra homomorphisms 
of $\n_+$ and $\n_-$  to themselves, and $\gamma(e_{\pm\alpha})=0$
for all $\alpha\in\Pi\setminus\Gamma_1$. Further, if $\g$ is simply connected $\gamma$ can be lifted to $\bgamma=\exp\gamma$;  note that $\bgamma$ is  defined only on $\N_+$ and $\N_-$ and is a group homomorphism. 

 By the classification theorem, each classical R-matrix is equivalent to an R-matrix from a {\it Belavin--Drinfeld class\/} defined by a BD triple $\bfG$. The operator $R^\bfG_+$ corresponding
to a member of this class is given by
\begin{equation*}
\label{Rplusgamma}
R^\bfG_+=R_0^\bfG+\frac1{1-\gamma}\pi_{>}-\frac{\gamma^*}{1-\gamma^*}\pi_{<},
\end{equation*}
where $\pi_{>}$, $\pi_{<}$ are projections of  
$\g$ onto $\n_+$ and $\n_-$ and $R_0^\bfG$ acts on $\h$ (see~\cite{GSVple} for more details).

In what follows we will use a Poisson bracket on $\G$ that is a generalization of the bracket \eqref{sklyabra}.
Let $r, r'$ be two classical R-matrices, and $R_+, R'_+$ be the corresponding operators, then we write
\begin{equation}
\{f_1,f_2\}_{r,r'} = \langle R'_+(\nabla^L f_1), \nabla^L f_2 \rangle -\langle R_+(\nabla^R f_1), \nabla^R f_2 \rangle.
\label{sklyabragen}
\end{equation} 
By \cite[Proposition 12.11]{r-sts}, the above expression defines a Poisson bracket, which is not Poisson--Lie unless $r=r'$,
in which case $\{f_1,f_2\}_{r,r}$ evidently coincides with $\{f_1,f_2\}_{r}$. 
The bracket \eqref{sklyabragen} defines a Poisson homogeneous structure on $\G$ with respect to the left and right multiplication by Poisson--Lie groups $(\G,\Poi_{r'})$ and
$(\G,\Poi_{r})$, respectively.

\section{Poisson map}

\subsection{The main construction}\label{mainconstruction} 
We will write $\G_{r,r'}$ for the Poisson manifold 
$(\G,\allowbreak\Poi_{r,r'})$. Fix a pair of R-matrices 
$r^\bfGr$, $r^\bfGc$ from the BD classes defined by $\bfGr$ and $\bfGc$, respectively. Additionally, fix two R-matrices $r_\bfGr^\pu$, $r_\bfGc^\pu$ from the standard BD class 
(corresponding to the empty triple $\Gamma$) so that $R_0$ for $r^\bfGr$ and $r_\bfGr^\pu$ coincide, and $R_0$ for $r^\bfGc$ and $r_\bfGc^\pu$ coincide. Our aim is to build a rational Poisson map 
$h: \G_{r_\bfGr^\pu,r_\bfGc^\pu}\to \G_{r^\bfGr,r^\bfGc}$.

Take $U\in\G$ and consider its Gauss decomposition $U=U_-U_0U_+$ (so, in fact, $U$ lies in an open dense subset in $\G$). We further factor $U_-=V^\er\tilde U_-$ with 
$V^\er\in\N_-^{\Gamma_1^\er}$ 
and $U_+=\tilde U_+V^\ec$ with $V^\ec\in\N_+^{\Gamma_2^\ec}$, as explained in 
Section~\ref{sec:factor}.
Next, choose $\wog^\er\in\G^{\Gamma_1^\er}$  and $\wog^\ec\in\G^{\Gamma_2^\ec}$ such that 
$\wog^\er \B_-^{\Gamma_1^\er}(\wog^\er)^{-1}=\B_+^{\Gamma_1^\er}$ and 
$\wog^\ec \B_+^{\Gamma_2^\ec}(\wog^\ec)^{-1}=\B_-^{\Gamma_2^\ec}$. The elements $\wog^\er$ and 
$\wog^\ec$ may be chosen, for example, as representatives of the longest elements of the Weyl groups of $\G^{\Gamma_1^\er}$ and $\G^{\Gamma_2^\ec}$, respectively, via the procedure described in~\cite[Sect.~1.4]{FZ}. Write
\[
V^\er\wog^\er=(V^\er\wog^\er)_+(V^\er\wog^\er)_{0,-}, \qquad \wog^\ec V^\ec=(\wog^\ec V^\ec)_{+,0}(\wog^\ec V^\ec)_-
\]
and set $\bar V^\er=(V^\er\wog^\er)_+\in\N_+^{\Gamma_1^\er}$, $\bar V^\ec=(\wog^\ec V^\ec)_-\in\N_-^{\Gamma_2^\ec}$. 

Define 
\[
\begin{aligned}
H^\er&=H^\er(U)=\cdots (\bgammar)^3(\bar V^\er)(\bgammar)^2(\bar V^\er)\bgammar
(\bar V^\er)\in\N_+^{\Gamma_2^\er},\\
H^\ec&=H^\ec(U)=\cdots ((\bgammac)^*)^3(\bar V^\ec)((\bgammac)^*)^2(\bar V^\ec)
(\bgammac)^*(\bar V^\ec)\in\N_-^{\Gamma_1^\ec}
\end{aligned}
\]
(the products above are finite due to the nilpotency of $\gammar$ and $\gammac$). 

\begin{theorem}
\label{twosidedpoisson}
The map $h: \G_{r_\bfGr^\pu,r_\bfGc^\pu}\to \G_{r^\bfGr,r^\bfGc}$ defined by 
$h(U)=H^\er(U) U H^\ec(U)$ is a rational Poisson map.
\end{theorem}

\begin{proof} 
Fix an arbitrary $r^\pu$ from the standard BD class and define a map
$h^\er: \G_{r^\pu_\bfGr,r^\pu}\to \G_{r^\bfGr,r^\pu}$ via $h^\er(U)=H^\er(U) U$. 
For the same $r^\pu$ as above define a map
$h^\ec: \G_{r^\pu,r^\pu_\bfGc}\to \G_{r^\pu,r^\bfGc}$ via $h^\ec(U)=UH^\ec(U)$. 

\begin{theorem}
\label{onesidedpoisson}
The maps $h^\er$ and $h^\ec$ are rational Poisson maps.
\end{theorem}

The proof of Theorem~\ref{onesidedpoisson} is given in the next subsection.

Since the Borel subgroup $\B_+\subset\G$ and the opposite Borel subgroup $\B_-\subset\G$ are Poisson submanifolds, the restrictions of $h^\ec$ to $\B_+$ and of $h^\er$ to $\B_-$ are Poisson maps as well; their images $h^\ec(\B_+)$ and $h^\er(\B_-)$ are called  {\it twisted Borels\/}.

Note that the following diagram is commutative:
\begin{equation}
\label{maincd}
\begin{tikzcd}
(\B_-)_{r^\pu_\bfGr,r^\pu}\times(\B_+)_{r^\pu,r^\pu_\bfGc}\arrow[r,"{(h^\er,h^\ec)}"]\arrow[d,"g"]
& h^\er(\B_-)_{r^\bfGr,r^\pu}\times h^\ec(\B_+)_{r^\pu,r^\bfGc} \arrow[d,"g"]\\
\G_{r_\bfGr^\pu,r_\bfGc^\pu}\arrow[r,"h"] & \G_{r^\bfGr,r^\bfGc}
\end{tikzcd}
\end{equation}
where $g:(U^-,U^+)\mapsto U^-U^+=U$ is the multiplication map. Indeed, since $h^\er(U^-)=H^\er(U^-) U^-$ and 
$h^\ec(U^+)=U^+H^\ec(U^+)$, we get 
\[
g\circ(h^\er,h^\ec)(U^-,U^+)=H^\er(U^-) U^-U^+H^\ec(U^+)=H^\er(U^-) UH^\ec(U^+).
\]
Recall that $H^\er(U)$ depends only on the first term of the Gauss decomposition, and $H^\ec(U)$ only on its last term,
hence $H^\er(U^-)=H^\er(U)$, $H^\ec(U^+)=H^\ec(U)$, and 
\begin{equation}
\label{XthruU}
h\circ g(U^-,U^+)=h(U)=H^\er(U) UH^\ec(U).
\end{equation}

\begin{proposition}
\label{poissonmultiplication}
For any three R-matrices $r$, $r'$, $r''$, the multiplication map 
$g: \G_{r',r}\times \G_{r,r''}\to \G_{r',r''}$ is Poisson. 
\end{proposition} 

\begin{proof} Let $\lambda_X$ denote the left translation by $X$ and $\rho_Y$ denote the right translation by $Y$. We have to check the identity
\begin{equation}\label{multicheck}
\{\rho_Yf^1,\rho_Yf^2\}_{r',r}(X)+\{\lambda_Xf^1,\lambda_Xf^2\}_{r,r''}(Y)=
\{f^1,f^2\}_{r',r''}(Z)
\end{equation}
for $Z=XY$. Note that $\nabla_X(\rho_Yf)(X)=Y\nabla_Zf(Z)$ and 
$\nabla_Y(\lambda_Xf)(Y)=\nabla_Zf(Z)X$. Consequently, 
\[
\begin{aligned}
\nabla^R(\rho_Yf)(X)&=X\nabla_X(\rho_Yf)(X)=Z\nabla_Zf(Z)=\nabla^Rf(Z),\\
\nabla^L(\rho_Yf)(X)&=\nabla_X(\rho_Yf)(X)X=\Ad_Y(\nabla_Zf(Z)Z)=\Ad_Y\nabla^Lf(Z),\\
\nabla^R(\lambda_Xf)(Y)&=Y\nabla_Y(\lambda_Xf)(Y)=\Ad_Y(\nabla_Zf(Z)Z)=\Ad_Y\nabla^Lf(Z),\\
\nabla^L(\lambda_Xf)(Y)&=\nabla_Y(\lambda_Xf)(Y)Y=\nabla_Zf(Z)Z=\nabla^Lf(Z).
\end{aligned}
\]
We thus have
\begin{multline*}
\{\rho_Yf^1,\rho_Yf^2\}_{r',r}(X)=\langle R_+\nabla^L(\rho_Yf^1),\nabla^L(\rho_Yf^2)\rangle-
\langle R'_+\nabla^R(\rho_Yf^1),\nabla^R(\rho_Yf^2)\rangle\\
=\langle R_+\Ad_Y\nabla^Lf^1,\Ad_Y\nabla^Lf^2\rangle-\langle R'_+\nabla^Rf^1, \nabla^Rf^2\rangle
\end{multline*}
and
\begin{multline*}
\{\lambda_Xf^1,\lambda_Xf^2\}_{r,r''}(Y)=\langle R''_+\nabla^L(\lambda_Xf^1),
\nabla^L(\lambda_Xf^2)\rangle-
\langle R_+\nabla^R(\lambda_Xf^1),\nabla^R(\lambda_Xf^2)\rangle\\
=\langle R''_+\nabla^Lf^1, \nabla^Lf^2\rangle-\langle R_+\Ad_Y\nabla^Lf^1,\Ad_Y\nabla^Lf^2\rangle,
\end{multline*}
which proves~\eqref{multicheck}.
\end{proof}

A particular case of this claim for $r=r'$ or $r=r''$ is given in
Proposition 5.2.18  of~\cite{r-sts}. Note: in their notation, our $\Poi_{r,r'}$ is $\Poi_{r',-r}$.

Since $h^\er$ and $h^\ec$ are Poisson and $g$ is Poisson and surjective we get that $h$ is Poisson.
\end{proof}

\subsection{Proof of Theorem \ref{onesidedpoisson}}\label{sectiononesided} 
We only present the proof for $h^\er$, since the proof for 
$h^\ec$ is similar.  To make the formulas more readable, in this Section we use the following  notation:  
$\bfG=\bfGr$, $\Gamma_i=\Gamma^\er_i$, $\gamma=\gammar$, $\G_\pm=\G^\er_\pm$, $V=V^\er$, 
$\wog=\wog^\er$, $\bar V=\bar V^\er$, $H=H^\er$.

Our first goal is to  invert $h^\er$. 
We start with finding $\bar V$ via $H$. Since $H\in\N_+^{\Gamma_2}$ and $\bgamma$ is a homomorphism we have
\[
\bgamma(H)=\cdots \bgamma^4(\bar V)\bgamma^3(\bar V)\bgamma^2(\bar V)=H\bgamma(\bar V)^{-1},
\]
and so $\bgamma(\bar V)=\bgamma(H^{-1})H$, which gives
\[
\bgamma^*\bgamma(\bar V)=\bgamma^*\bgamma(H^{-1})\bgamma^*(H).
\]
Recall that $\bgamma^*\bgamma$ acts on $\N_+$ as the projection to $\N_+^{\Gamma_1}$, so 
$\bgamma^*\bgamma(\bar V)=\bar V$ and hence
\begin{equation}\label{barVtH}
\bar V=\bgamma^*\bgamma(H^{-1})\bgamma^*(H). 
\end{equation}

Next, we find $V$ via $H$. Recall that $\bar V=(V\wog)_+$, hence 
$\bar V(V\wog)_{0,-}=V\wog$, hence
\[
\bar V\wog^{-1}\left(\wog(V\wog)_{0,-}\wog^{-1}\right)=V.
\]
Applying the Gauss decomposition once again we get
\begin{equation}\label{trick}
(\bar V\wog^{-1})_-(\bar V\wog^{-1})_{0,+}\left(\wog(V\wog)_{0,-}\wog^{-1}\right)=V.
\end{equation}
The last bracket on the left belongs to $\B_+$, while $V\in\N_-$, so the second and the third bracket cancel each other and we get $V=(\bar V\wog^{-1})_-$, which together with~\eqref{barVtH} gives
\[
V=(\bgamma^*\bgamma(H^{-1})\bgamma^*(H)\wog^{-1})_-.
\]

Consider two parabolic subalgebras of $\g$ determined by $\bfG$: $\pp_+^\bfG$ contains $\b_+$ and all the negative root spaces in $\g^{\Gamma_1}$, while $\pp_-^\bfG$ contains $\b_-$ and all the positive root spaces in $\g^{\Gamma_2}$. Denote by $\P_\pm^\bfG$ the corresponding parabolic subgroups of $\G$, and let $\ZZ=\P^\bfG_+\cap\P^\bfG_-$.
Note that the corresponding subalgebra $\pp_+^\bfG\cap \pp_-^\bfG$ is a seeweed subalgebra introduced for type A in \cite{DK}. There is a commutative diagram
\[
\begin{tikzcd}
\ZZ\times\tilde\G^{\sigma_1,\sigma_2}\arrow[r,"{(h^\er,\id)}"] \arrow[d,"g"] & 
\ZZ\times\tilde\G^{\sigma_1,\sigma_2} \arrow[d,"g"]\\
\G \arrow[r,"{h^\er}"]&\G
\end{tikzcd}
\]
where $\tilde\G^{\sigma_1,\sigma_2}$ is the reduced double Bruhat cell corresponding to 
$\sigma_1=w_0^{\Gamma_1}w_0$, $\sigma_2=w_0^{\Gamma_2}w_0$ with $w_0$, $w_0^{\Gamma_1}$
and $w_0^{\Gamma_2}$ being the longest elements of the corresponding Weyl groups, 
and $g$ is the product similarly to~\eqref{maincd}. So, to invert $h^\er$ on $\G$ it is enough to invert it on $\ZZ$ and to invert the vertical arrow on the 
right. Note that reduced double Bruhat cells are not Poisson submanifolds.
For this reason, to prove that $(h^\er)^{-1}$ is Poisson on the
whole $\G$ provided it is Poisson on $\ZZ$ we use, similarly to~\eqref{maincd}, the commutative diagram
\[
\begin{tikzcd}[column sep=large]
\ZZ_{r^\bfG,r^\pu}\times\G^{\sigma_1,\sigma_2}_{r^\pu,r^\pu}
\arrow[r,"{((h^\er)^{-1},\id)}"] 
\arrow[d,"g"] & 
\ZZ_{r^\pu_\bfG,r^\pu}\times\G^{\sigma_1,\sigma_2}_{r^\pu,r^\pu} \arrow[d,"g"]\\
\G_{r^\bfG,r^\pu} \arrow[r,"{(h^\er)^{-1}}"]& \G_{r^\pu_\bfG,r^\pu}
\end{tikzcd}
\]
where $g$ on both sides is Poisson by Proposition \ref{poissonmultiplication}.

To invert $h^\er$ on $\ZZ$ note that for $U\in\ZZ$ one has $\tilde U_-=\one$, and hence
\[
Z=h^\er(U)=HVU_{0,+}=(HV)_-(HV)_{0,+}U_{0,+}=Z_-Z_{0,+},
\]
(one more open condition), so that
\[
Z_-=(HV)_-=(H(\bgamma^*\bgamma(H^{-1})\bgamma^*(H)\wog^{-1})_-)_-.
\]
Clearly, $(AB_-)_-=(AB)_-$, since $AB=AB_-B_{0,+}=(AB_-)_-(AB_-)_{0,+}B_{0,+}$, so
\[
Z_-=(H\bgamma^*\bgamma(H^{-1})\bgamma^*(H)\wog^{-1})_-.
\]
Recall that $H\in\N_+^{\Gamma_2}$, hence $\bgamma^*(H)\wog^{-1}\in\G^{\Gamma_1}$. On the other hand, the projection of $H\bgamma^*\bgamma(H^{-1})$ to $\N_+^{\Gamma_1}$ is given by
\[
\bgamma^*\bgamma(H\bgamma^*\bgamma(H^{-1}))=\bgamma^*\bgamma(H)\bgamma^*\bgamma(H^{-1})=\one
\]
since $\bgamma^*\bgamma$ is an idempotent, and so $Z_-=(\bgamma^*(H)\wog^{-1})_-$. Using the same trick as in~\eqref{trick} in the opposite direction we get $\bgamma^*(H)=(Z_-\wog)_+=\bar Z_+$ for 
$\bar Z=Z_-\wog$. Note that $\bar Z_+\in\N_+^{\Gamma_1}$. Since 
$\bgamma\bgamma^*$ acts on $\N_+$ as the projection to $\N_+^{\Gamma_2}$, we get 
$\bgamma\bgamma^*(H)=H=\bgamma(\bar Z_+)$, and finally, 
$U=(h^\er)^{-1}(Z)=H^{-1}Z=\bgamma(\bar Z_+^{-1})Z$. Thus, we have inverted $h^\er$ on $\ZZ$.

 To proceed further we need to find the variation $\delta U$.
 Recall that $Z=Z_-Z_{0,+}$, hence $T_Z\G=(T_{Z_-}\N_-) Z_{0,+}\oplus Z_- (T_{Z_{0,+}}\B_+)$,
or, in other words, $\delta Z=\delta Z_-Z_{0,+}+Z_-\delta Z_{0,+}$. 
 Here and in what follows we admit a common abuse of notation and write $gv$ instead of $(\lambda_g)_*(v)$ for the left translation of a tangent vector $v$ by a group element $g$ and $vg$ instead of $(\rho_g)_*(v)$ for the right translation.
Note that $Z_-^{-1}\delta Z_-\in\n_-$ since the left translation by $Z_-^{-1}$ identifies $\n_-=T_1\N_-$ with $T_{Z_-}\N_-$. Similarly, 
$Z^{-1}\delta Z\in\g$ and $Z_{0,+}^{-1}\delta Z_{0,+}\in\b_+$. Therefore,
\[
\Ad_{Z_{0,+}}Z^{-1}\delta Z=Z_-^{-1}\delta Z_-+\Ad_{Z_{0,+}}Z_{0,+}^{-1}\delta Z_{0,+}.
\]
The first term on the right belongs to $\n_-$ and the second to $\b_+$, hence we get 
$(\Ad_{Z_{0,+}}Z^{-1}\delta Z)_<=Z_-^{-1}\delta Z_-$; here and in what follows we write
$A_<$ for $\pi_<(A)$, etc.

 Similarly, $\bar Z=\bar Z_+\bar Z_{0,-}$, 
and hence 
\[
\Ad_{\bar Z_{0,-}}\bar Z^{-1}\delta\bar Z=\bar Z_+^{-1}\delta\bar Z_+
+\Ad_{\bar Z_{0,-}}\bar Z_{0,-}^{-1}\delta\bar Z_{0,-}.
\]
Here the first term on the right belongs to $\n_+$ and the second to $\b_-$, hence
\begin{multline*}
\bar Z_+^{-1}\delta\bar Z_+=(\Ad_{\bar Z_{0,-}}\bar Z^{-1}\delta\bar Z)_>
=(\Ad_{\bar Z_{0,-}\wog^{-1}}Z_-^{-1}\delta Z_-)_>\\
=\left(\Ad_{\tilde Z}(\Ad_{Z_{0,+}}Z^{-1}\delta Z)_<\right)_>
\end{multline*}
with ${\tilde Z}=\bar Z_{0,-}\wog^{-1}$, since $\bar Z= Z_-\wog$ and $\bar Z^{-1}\delta\bar Z=
\Ad_{\wog^{-1}}Z_-^{-1}\delta Z_-$.

Finally, $U=\bgamma(\bar Z_+^{-1})Z$, so
\begin{multline*}
U^{-1}\delta U=U^{-1}\delta(\bgamma(\bar Z_+^{-1}))Z+U^{-1}\bgamma(\bar Z_+^{-1})\delta Z\\
=-U^{-1}\bgamma(\bar Z_+^{-1})\delta(\bgamma(\bar Z_+))\bgamma(\bar Z_+^{-1})Z+Z^{-1}\delta Z
=-\Ad_{U^{-1}}\left(\gamma(\bar Z_+^{-1}\delta\bar Z_+ )\right)+Z^{-1}\delta Z\\
= -\Ad_{U^{-1}}\left(\gamma\left(\left(\Ad_{{\tilde Z}}(\Ad_{Z_{0,+}}Z^{-1}\delta Z)_<\right)_>\right)\right)
+Z^{-1}\delta Z.
\end{multline*}

Let us compute the gradients of $\hat f(Z)=f\circ (h^\er)^{-1}(Z)$.
We start with
\[
\begin{aligned}
\langle&\nabla f(U)U,U^{-1}\delta U\rangle\\
&=\langle\nabla f(U)U,Z^{-1}\delta Z\rangle-
\left\langle \nabla f(U)U,\Ad_{U^{-1}}\left(\gamma\left(\left(\Ad_{{\tilde Z}}(\Ad_{Z_{0,+}}Z^{-1}\delta Z)_<
\right)_>\right)\right)\right\rangle\\
&=\langle\nabla f(U)U,Z^{-1}\delta Z\rangle-\left\langle\Ad_U (\nabla f(U)U),
\gamma\left(\left(\Ad_{{\tilde Z}}(\Ad_{Z_{0,+}}Z^{-1}\delta Z)_<
\right)_>\right)\right\rangle.
\end{aligned}
\]
Note that  $\langle \alpha, \gamma(\beta_>)\rangle=\langle \gamma^*(\alpha_{<}),\beta\rangle$ 
for any $\alpha,\beta \in\g$, since
\begin{multline*}
\langle \alpha, \gamma(\beta_>)\rangle=\langle \alpha_<,\gamma(\beta_>)\rangle+\langle\alpha_{\ge}, \gamma(\beta_>)\rangle=
\langle \alpha_<, \gamma(\beta_>)\rangle\\
=\langle \gamma^*(\alpha_<), \beta_>\rangle=\langle \gamma^*(\alpha_<), \beta_>\rangle+\langle \gamma^*(\alpha_<), \beta_{\le}\rangle
=\langle \gamma^*(\alpha_<), \beta\rangle,
\end{multline*}
so that
\begin{multline*}
\left\langle\Ad_U (\nabla f(U)U),\gamma\left(\left(\Ad_{{\tilde Z}}(\Ad_{Z_{0,+}}Z^{-1}\delta Z)_<
\right)_>\right)\right\rangle\\
=\left\langle \gamma^*\left(\left(\Ad_U(\nabla f(U)U)\right)_{<}\right),
\Ad_{{\tilde Z}}\left(\Ad_{Z_{0,+}}(Z^{-1}\delta Z)\right)_{<}\right\rangle\\
=\left\langle\Ad_{{\tilde Z}^{-1}}\gamma^*\left(\left(\Ad_U(\nabla f(U)U)\right)_{<}\right),
\left(\Ad_{Z_{0,+}}(Z^{-1}\delta Z)\right)_{<}\right\rangle.
\end{multline*}

Further, $\gamma^*\left(\left(\Ad_U(\nabla f(U)U)\right)_{<}\right)\in\n_-$, hence 
$\Ad_{\bar Z_{0,-}^{-1}}\left(\gamma^*\left(\left(\Ad_U(\nabla f(U)U)\right)_{<}\right)\right)\in\n_-$, so that
$\Ad_{\wog}\Ad_{\bar Z_{0,-}^{-1}}\left(\gamma^*\left(\left(\Ad_U(\nabla f(U)U)\right)_{<}\right)\right)\in\n_+$. Therefore
\begin{multline*}
\left\langle\Ad_{{\tilde Z}^{-1}}\gamma^*\left(\left(\Ad_U(\nabla f(U)U)\right)_{<}\right),
\left(\Ad_{Z_{0,+}}(Z^{-1}\delta Z)\right)_{<}\right\rangle\\
=\left\langle\Ad_{{\tilde Z}^{-1}}\gamma^*\left(\left(\Ad_U(\nabla f(U)U)\right)_{<}\right),
\Ad_{Z_{0,+}}(Z^{-1}\delta Z)\right\rangle\\
=\left\langle\Ad_{Z_{0,+}^{-1}}\Ad_{{\tilde Z}^{-1}}\gamma^*\left(\left(\Ad_U(\nabla f(U)U)\right)_{<}\right),
Z^{-1}\delta Z\right\rangle,
\end{multline*}
so finally
\[
\begin{aligned}
\langle \nabla f(U)U, U^{-1}\delta U\rangle &=
\left\langle \nabla f(U)U- \Ad_{Z_{0,+}^{-1}}\Ad_{{\tilde Z}^{-1}}
\gamma^*\left(\left(\Ad_U(\nabla f(U)U)\right)_{<}\right),Z^{-1}\delta Z \right\rangle\\
&=\left\langle \nabla f(U)U- \Ad_{Z^{-1}\bar Z_+}
\gamma^*\left(\left(\Ad_U(\nabla f(U)U)\right)_{<}\right),Z^{-1}\delta Z \right\rangle
\end{aligned}
\]
since $Z_{0,+}^{-1}\wog\bar Z_{0,-}^{-1}=Z^{-1}\bar Z_+$.

Recall that $\langle\nabla\hat f(Z)Z,Z^{-1}\delta Z\rangle=\langle\nabla f(U)U,U^{-1}\delta U\rangle$, hence
\[
\begin{aligned}
\nabla^L\hat f&=\nabla \hat f(Z)Z=\nabla f(U)U-\Ad_{Z^{-1}\bar Z_+}\gamma^*\left(\left(\Ad_U(\nabla f(U)U)\right)_{<}\right)\\
\nabla^R\hat f&=\Ad_Z(\nabla \hat f(Z)Z)=\Ad_Z(\nabla f(U)U)-\Ad_{\bar Z_+}\gamma^*\left(\left(\Ad_U(\nabla f(U)U)\right)_{<}\right).
\end{aligned}
\]

To prove Theorem~\ref{onesidedpoisson} we need to verify that 
$\{\hat f_1,\hat f_2\}_{r^\bfG,r^\pu}(Z)=
\{f_1,f_2\}_{r^\pu_\bfG,r^\pu}(U)$ for $U=h_\er^{-1}(Z)$. Recall that
\begin{equation}\label{brackets}
\begin{aligned}
\{\hat f_1,\hat f_2\}_{r^\bfG,r^\pu}&=\langle R_+^\pu\nabla^L\hat f_1,\nabla^L\hat f_2\rangle
-\langle R_+^\bfG\nabla^R\hat f_1,\nabla^R\hat f_2\rangle,\\
\{f_1,f_2\}_{r^\pu_\bfG,r^\pu}&=\langle R_+^\pu\nabla^L f_1,\nabla^L f_2\rangle
-\langle (R_+)^\pu_\bfG\nabla^R f_1,\nabla^R f_2\rangle,
\end{aligned}
\end{equation}
with
\begin{equation}\label{R+s}
\begin{aligned}
R_+^\bfG&=R_0^\bfG+\frac1{1-\gamma}\pi_>-\frac{\gamma^*}{1-\gamma^*}\pi_<,\\
R_+^\pu&=R_0^\pu+\pi_>,\qquad (R_+)^\pu_\bfG=R_0^\bfG+\pi_>.
\end{aligned}
\end{equation}
Further, $R_0^\bfG=(\frac12+S^\bfG)\pi_\ze$, where $\pi_\ze$ is the projection on $\h$ and
$S^\bfG$ is a skew-symmetric operator on $\h$ satisfying 
$S^\bfG(1-\gamma)=\frac12(1+\gamma)$ on $\h^{\Gamma_1}$.  
Consequently, on $\h^{\Gamma_2}$
\[
S^\bfG(\gamma^*-1)=S^\bfG(1-\gamma)\gamma^* =\frac12(1+\gamma)\gamma^*=\frac12(\gamma^*+1), 
\]
and hence 
\begin{equation}\label{form1}
R_0^\bfG(\gamma^*-1)=\gamma^*\quad \text{on $\h^{\Gamma_2}$.}
\end{equation}
 Finally, $(R_0^\bfG)^*=(\frac12-S^\bfG)\pi_\ze=\pi_\ze-R_0^\bfG$, 
so that 
\begin{equation}\label{form2}
(R_0^\bfG)^*(1-\gamma^*)=1\quad \text{on $\h^{\Gamma_2}$.}
\end{equation}

Introduce some notation:
\[
\alpha_f=\nabla^L f=\nabla f(U)U,\qquad \beta_f=\nabla^R f=\Ad_U\alpha_f.
\]
Further,
\[
\zeta_f=\Ad_{\bar Z_+}\gamma^*(\beta_f)_<,\qquad \xi_f=\Ad_{Z^{-1}}\zeta_f.
\]
Note that $\xi_f\in\n_+$ since $Z^{-1}\bar Z_+=Z_{0,+}^{-1}\wog\bar Z_{0,-}^{-1}$. Finally, denote
\[
\Ad_Z\alpha_f=\Ad_{\bgamma(\bar Z_+)}\beta_f=\eta_f,
\]
so that
\[
\nabla^L\hat f=\alpha_f-\xi_f,\qquad \nabla^R\hat f=\eta_f-\zeta_f.
\]

We will need the following technical result.

\begin{lemma}\label{adgamma}
For any $X\in \G^{\Gamma_1}$ and 
$\eta\in\g$ holds  $\gamma(\Ad_X\eta)=\Ad_{\bgamma(X)}\gamma(\eta)$.
\end{lemma}

\begin{proof} 
Write $\eta=\lambda+\theta$ with $\lambda\in\g^{\Gamma_1}$ and $\theta$ orthogonal to 
$\g^{\Gamma_1}$ with respect to $\langle\cdot,\cdot\rangle$, then $\gamma(\eta)=\gamma(\lambda)$. 
Next, $\gamma(\Ad_X\eta)=\gamma(\Ad_X\lambda)+\gamma(\Ad_X\theta)$. Note that 
$\gamma(\Ad_X\lambda)=\Ad_{\bgamma(X)}\gamma(\lambda)= \Ad_{\bgamma(X)}\gamma(\eta)$ since $\gamma$
is a Lie algebra isomorphism on $\g^{\Gamma_1}$. It remains to prove that $\Ad_X\theta$ is orthogonal to $\g^{\Gamma_1}$, and thus $\gamma(\Ad_X\theta)=0$. This is equivalent to 
$\langle \ad_\xi\theta,\zeta\rangle=0$ for any $\xi,\zeta\in\g^{\Gamma_1}$. Clearly, 
$\langle [\xi,\theta],\zeta\rangle=\langle\theta,[\xi,\zeta]\rangle=0$ since 
$ [\xi,\zeta]\in\g^{\Gamma_1}$.
\end{proof}

 From~\eqref{brackets} and~\eqref{R+s} follows that the non-diagonal part of $\{\hat f_1,\hat f_2\}_{\bfG,\varnothing}$ equals
\begin{equation}\label{brack}
\langle (\nabla^L\hat f_1)_>,\nabla^L\hat f_2\rangle
-\left\langle \frac1{1-\gamma}(\nabla^R\hat f_1)_>,\nabla^R\hat f_2\right\rangle
+\left\langle\frac{\gamma^*}{1-\gamma^*}(\nabla^R \hat f_1)_<,\nabla^R\hat f_2\right\rangle.
\end{equation}

The first expression in \eqref{brack} is equal to
\[
\langle (\alpha_{f_1}-\xi_{f_1})_>,\alpha_{f_2}-\xi_{f_2}\rangle=
\langle (\alpha_{f_1})_>,\alpha_{f_2}\rangle-\langle (\xi_{f_1})_>,\alpha_{f_2}\rangle-
\langle (\alpha_{f_1})_>,\xi_{f_2}\rangle+\langle (\xi_{f_1})_>,\xi_{f_2}\rangle.
\]
Recall that $\xi_f\in \n_+$, so the last two expressions above vanish. The second expression equals
\[
\langle (\xi_{f_1})_>,\alpha_{f_2}\rangle=\langle \xi_{f_1},\alpha_{f_2}\rangle=
\langle\Ad_{Z^{-1}}\zeta_{f_1},\alpha_{f_2}\rangle=\langle \zeta_{f_1},\Ad_Z\alpha_{f_2}\rangle=
\langle\zeta_{f_1},\eta_{f_2}\rangle,
\]
so that finally
\[
\langle (\nabla^L\hat f_1)_>,\nabla^L\hat f_2\rangle=
\langle (\alpha_{f_1})_>,\alpha_{f_2}\rangle-\langle\zeta_{f_1},\eta_{f_2}\rangle.
\]

The second expression in \eqref{brack} is equal to
\[
\begin{aligned}
&\left\langle \frac1{1-\gamma}(\eta_{f_1}-\zeta_{f_1})_>,\eta_{f_2}-\zeta_{f_2}\right\rangle\\
&=\langle(\eta_{f_1}-\zeta_{f_1})_> ,\eta_{f_2}-\zeta_{f_2}\rangle+
\left\langle \frac{\gamma}{1-\gamma}(\eta_{f_1}-\zeta_{f_1})_>, \eta_{f_2}-\zeta_{f_2}\right\rangle\\
&=\langle(\eta_{f_1}-\zeta_{f_1})_>,\eta_{f_2}-\zeta_{f_2}\rangle+
\left\langle (\eta_{f_1}-\zeta_{f_1})_>,\frac{\gamma^*}{1-\gamma^*}(\eta_{f_2}-\zeta_{f_2})\right\rangle.
\end{aligned}
\]
Note that
\[
\begin{aligned}
\frac{\gamma^*}{1-\gamma^*}(\eta_{f_2}-\zeta_{f_2})&=
\frac1{1-\gamma^*}\left({\gamma^*}(\eta_{f_2})-\zeta_{f_2}\right)+\zeta_{f_2}\\
&=\frac1{1-\gamma^*}\left({\gamma^*}(\Ad_{\bgamma(\bar Z_+)}\beta_{f_2})
-\Ad_{\bar Z_+}\gamma^*(\beta_{f_2})_<\right)+\zeta_{f_2}\\
&=\frac1{1-\gamma^*}\left(\Ad_{\bgamma^*\bgamma(\bar Z_+)}{\gamma^*}(\beta_{f_2})
-\Ad_{\bar Z_+}\gamma^*(\beta_{f_2})_<\right)+\zeta_{f_2}\\
&=\frac1{1-\gamma^*}\left(\Ad_{\bar Z_+}\gamma^*(\beta_{f_2})_{\geq}\right)+\zeta_{f_2},
\end{aligned}
\]
since $\bgamma^*\bgamma(\bar Z_+)=\bar Z_+$. Consequently,
\[
\begin{aligned}
&\left\langle (\eta_{f_1}-\zeta_{f_1})_>,\frac{\gamma^*}{1-\gamma^*}(\eta_{f_2}-\zeta_{f_2})\right\rangle\\
&=\left\langle (\eta_{f_1}-\zeta_{f_1})_>,
\frac1{1-\gamma^*}\left(\Ad_{\bar Z_+}\gamma^*(\beta_{f_2})_{\geq}\right)+\zeta_{f_2}\right\rangle\\
&=\left\langle (\eta_{f_1}-\zeta_{f_1})_>,\zeta_{f_2}\right\rangle,
\end{aligned}
\]
since $\bar Z_+\in \N_+$ and hence $\Ad_{\bar Z_+}\gamma^*(\beta_{f_2})_{\geq}\in\b_+$, so that finally
\[
\left\langle \frac1{1-\gamma}(\nabla^R\hat f_1)_>,\nabla^R\hat f_2\right\rangle=
\langle(\eta_{f_1})_>,\eta_{f_2}\rangle-\langle(\zeta_{f_1})_>,\eta_{f_2}\rangle.
\]

The third expression in \eqref{brack} is treated similarly to the second one:
\[
\begin{aligned}
&\left\langle \frac{\gamma^*}{1-\gamma^*}(\eta_{f_1}-\zeta_{f_1})_>,\eta_{f_2}-\zeta_{f_2}\right\rangle\\
&=\left\langle \frac1{1-\gamma^*}\left(\Ad_{\bar Z_+}\gamma^*(\beta_{f_1})_{\geq}\right)_<+(\zeta_{f_1})_<,
\eta_{f_2}-\zeta_{f_2}\right\rangle\\
&=\langle (\zeta_{f_1})_<,\eta_{f_2}\rangle-\langle (\zeta_{f_1})_<,\zeta_{f_2}\rangle.
\end{aligned}
\]
since $\Ad_{\bar Z_+}\gamma^*(\beta_{f_1})_{\geq}\in\b_+$.

Further,
$-\langle\zeta_{f_1},\eta_{f_2}\rangle+\langle(\zeta_{f_1})_>,\eta_{f_2}\rangle+\langle(\zeta_{f_1})_<,\eta_{f_2}\rangle=
-\langle(\zeta_{f_1})_\ze,(\eta_{f_2})_\ze\rangle$.
Note that for $A\in\N_+$ and $\beta\in\g$ one has $(\Ad_A\beta)_<=(\Ad_A\beta_<)_<$ and for 
$\xi\in\n_+$ one has 
$\langle\Ad_A\beta_<,\xi\rangle=\langle\Ad_A\beta,\xi\rangle$, hence
\[
\begin{aligned}
&-\langle(\eta_{f_1})_>,\eta_{f_2}\rangle-\langle(\zeta_{f_1})_<,\zeta_{f_2}\rangle=
-\langle\eta_{f_1},(\eta_{f_2})_<\rangle-\langle\zeta_{f_1},(\zeta_{f_2})_>\rangle\\
&=-\left\langle \eta_{f_1},\left(\Ad_{\bgamma(\bar Z_+)}\beta_{f_2}\right)_<\right\rangle-
\left\langle\Ad_{\bar Z_+}\gamma^*(\beta_{f_1})_<,(\zeta_{f_2})_>\right\rangle\\
&=-\left\langle \eta_{f_1},\left(\Ad_{\bgamma(\bar Z_+)}(\beta_{f_2})_<\right)_<\right\rangle-
\left\langle\Ad_{\bar Z_+}\gamma^*(\beta_{f_1}),(\zeta_{f_2})_>\right\rangle\\
&=-\left\langle \eta_{f_1},\Ad_{\bgamma(\bar Z_+)}(\beta_{f_2})_<\right\rangle+
\left\langle \eta_{f_1},\left(\Ad_{\bgamma(\bar Z_+)}(\beta_{f_2})_<\right)_{\geq}\right\rangle-
\left\langle\Ad_{\bar Z_+}\gamma^*(\beta_{f_1}),(\zeta_{f_2})_>\right\rangle.
\end{aligned}
\]
The first term above is equal to
\begin{multline*}
-\left\langle \eta_{f_1},\Ad_{\bgamma(\bar Z_+)}(\beta_{f_2})_<\right\rangle=
-\left\langle \Ad_{\bgamma(\bar Z_+)}\beta_{f_1},\Ad_{\bgamma(\bar Z_+)}(\beta_{f_2})_<\right\rangle\\
=-\langle \beta_{f_1},(\beta_{f_2})_<\rangle=-\langle (\beta_{f_1})_>,\beta_{f_2}\rangle.
\end{multline*}
The second term above is equal to
\[
\begin{aligned}
&\left\langle \eta_{f_1},\left(\Ad_{\bgamma(\bar Z_+)}(\beta_{f_2})_<\right)_{\geq}\right\rangle=
\left\langle \Ad_{\bgamma(\bar Z_+)}\beta_{f_1},\left(\Ad_{\bgamma(\bar Z_+)}(\beta_{f_2})_<\right)_{\geq}\right\rangle\\
&=\left\langle \gamma^*\Ad_{\bgamma(\bar Z_+)}\beta_{f_1},
d\gamma^*\left(\Ad_{\bgamma(\bar Z_+)}(\beta_{f_2})_<\right)_{\geq}\right\rangle\\
&=
\left\langle \Ad_{\bar Z_+}\gamma^*(\beta_{f_1}),\left(\Ad_{\bar Z_+}\gamma^*(\beta_{f_2})_<\right)_{\geq}\right\rangle
=\left\langle \Ad_{\bar Z_+}\gamma^*(\beta_{f_1}),(\zeta_{f_2})_{\geq}\right\rangle,
\end{aligned}
\]
which together with the third term gives
\[
\left\langle \left(\Ad_{\bar Z_+}\gamma^*(\beta_{f_1})\right)_\ze,(\zeta_{f_2})_\ze\right\rangle=
\langle \gamma^*(\eta_{f_1})_\ze,(\zeta_{f_2})_\ze\rangle.
\]

Therefore, the total contribution of the non-diagonal terms equals to
\[
\langle (\alpha_{f_1})_>,\alpha_{f_2}\rangle-\langle (\beta_{f_1})_>,\beta_{f_2}\rangle
-\langle(\zeta_{f_1})_\ze,(\eta_{f_2})_\ze\rangle
+\langle \gamma^*(\eta_{f_1}))_\ze,(\zeta_{f_2})_\ze\rangle.
\]
On the other hand, the total contribution of the non-diagonal terms to 
$\{f_1,f_2\}_{r_\bfGc^\pu,r_\bfGr^\pu}$ equals to
\[
\langle (\alpha_{f_1})_>,\alpha_{f_2}\rangle-\langle (\beta_{f_1})_>,\beta_{f_2}\rangle,
\]
so it remains to prove that
\begin{multline}\label{bradiag}
\langle R_0^\varnothing(\alpha_{f_1}-\xi_{f_1}),(\alpha_{f_2}-\xi_{f_2})_\ze\rangle-
\langle R_0^\bfG (\eta_{f_1}-\zeta_{f_1}),(\eta_{f_2}-\zeta_{f_2})_\ze\rangle\\
=\langle R_0^\varnothing(\alpha_{f_1}),(\alpha_{f_2})_\ze\rangle-
\langle R_0^\bfG (\beta_{f_1}),(\beta_{f_2})_\ze\rangle
+\langle(\zeta_{f_1})_\ze,(\eta_{f_2})_\ze\rangle
-\langle \gamma^*(\eta_{f_1})_\ze,(\zeta_{f_2})_\ze\rangle.
\end{multline}

 Recall that $\xi_f\in\n_+$, hence $(\xi_f)_\ze$ vanishes. Therefore, the first term on the left in
\eqref{bradiag} equals 
$\langle R_0^\varnothing(\alpha_{f_1}),(\alpha_{f_2})_\ze\rangle$, which coincides with the first term on the right.

Further, $\zeta_f=\Ad_{\bar Z_+}(\gamma^*(\beta_f))-\Ad_{\bar Z_+}(\gamma^*(\beta_f)_\geq)$, hence
$(\zeta_f)_\ze=\gamma^*(\eta_f)_\ze-\gamma^*(\beta_f)_\ze$ since $\bar Z_+\in\n_+$ and $\gamma^*(\beta_f)_\geq\in\b_+$, 
so that finally 
\[
(\eta_f-\zeta_f)_\ze=(1-\gamma^*)(\eta_f-\beta_f)_\ze+(\beta_f)_\ze.
\]

Note that $(\eta_f-\beta_f)_\ze\in\h^{\Gamma_2}$, 
consequently, in view of \eqref{form1}, the second term in \eqref{bradiag} can be rewritten as
\[
\begin{aligned}
&\langle R_0^\bfG (\eta_{f_1}-\zeta_{f_1}),(\eta_{f_2}-\zeta_{f_2})_\ze\rangle\\
&=\langle R_0^\bfG (1-\gamma^*)(\eta_{f_1}-\beta_{f_1}),(\eta_{f_2}-\zeta_{f_2})_\ze\rangle+
\langle R_0^\bfG (\beta_{f_1}),(\eta_{f_2}-\zeta_{f_2})_\ze\rangle\\
&=-\langle \gamma^*(\eta_{f_1}-\beta_{f_1})_\ze,(\eta_{f_2}-\zeta_{f_2})_\ze\rangle+
\langle R_0^\bfG (\beta_{f_1}),(\eta_{f_2}-\zeta_{f_2})_\ze\rangle\\
&=-\langle (\zeta_{f_1})_\ze,(\eta_{f_2})_\ze\rangle
+\langle \gamma^*(\eta_{f_1})_\ze,(\zeta_{f_2})_\ze\rangle
\!-\!\langle \gamma^*(\beta_{f_1})_\ze,(\zeta_{f_2})_\ze\rangle
+\langle R_0^\bfG (\beta_{f_1}),(\eta_{f_2}-\zeta_{f_2})_\ze\rangle.
\end{aligned}
\]
The first two terms in the last line above are cancelled by the last two terms in the right hand side of \eqref{bradiag}.
So, \eqref{bradiag} is reduced to
\[
\langle R_0^\bfG (\beta_{f_1}),(\eta_{f_2}-\zeta_{f_2})_\ze\rangle
-\langle \gamma^*(\beta_{f_1})_\ze,(\zeta_{f_2})_\ze\rangle=
 \langle R_0^\bfG (\beta_{f_1}),(\beta_{f_2})_\ze\rangle.
\]

The first term in the left hand side above can be rewritten using \eqref{form2} as
\begin{multline*}
\langle R_0^\bfG (\beta_{f_1}),(\eta_{f_2}-\zeta_{f_2})_\ze\rangle=
\langle R_0^\bfG (\beta_{f_1}),(1-\gamma^*)(\eta_{f_2}-\beta_{f_2})_\ze\rangle
+\langle R_0^\bfG (\beta_{f_1}),(\beta_{f_2})_\ze\rangle\\
=\langle  (\beta_{f_1})_\ze,(R_0^\bfG)^*(1-\gamma^*)(\eta_{f_2}-\beta_{f_2})\rangle
+\langle R_0^\bfG (\beta_{f_1}),(\beta_{f_2})_\ze\rangle\\
=\langle  (\beta_{f_1})_\ze,(\eta_{f_2}-\beta_{f_2})_\ze\rangle
+\langle R_0^\bfG (\beta_{f_1}),(\beta_{f_2})_\ze\rangle\\
=\langle  \gamma^*(\beta_{f_1})_+,(\zeta_{f_2})_\ze\rangle
+\langle R_0^\bfG (\beta_{f_1}),(\beta_{f_2})_\ze\rangle,
\end{multline*}
which proves \eqref{bradiag}.

\section{The $A_n$ case}
In this Section we assume that $\G=SL_n$, and hence 
$\Gamma_1$ and $\Gamma_2$ can be identified with subsets of $[1,n-1]$. Note that the isometry condition on $\gamma$ implies that 
if $i,i+1\in\Gamma_1$ then $\gamma(i+1)=\gamma(i)\pm1$. We say that $\bfG$ is {\it oriented\/}  
if $i,i+1\in\Gamma_1$ yields $\gamma(i+1)=\gamma(i)+1$. In other words, the orientation of every subset of $\Gamma_1$ that consists of consecutive roots is preserved by $\gamma$. In~\cite{GSVple} we
treated the case when both BD triples $\Gamma^\er$ and $\Gamma^\ec$ are oriented. In this paper we 
lift this restriction and consider the general case.

\subsection{Combinatorial data} 
Let us briefly remind combinatorial constructions introduced in~\cite{GSVple} together with their non-oriented analogs 
(see~\cite{GSVple} for more details and examples). 

For any $i\in [1,n]$ put
\[
i_+=\min\{j\in [1,n]\setminus\Gamma_1\: \ j\ge i\}, \qquad
i_-=\max\{j\in [0,n]\setminus\Gamma_1\: \ j<i\}.
\]
The interval $\Delta(i)=[i_-+1,i_+]$ is called the {\it $X$-run\/} of $i$. 
Clearly, all distinct $X$-runs form a 
partition of $[1,n]$. The $X$-runs are numbered consecutively from left to right. The dual partition of $[1,n]$ into {\it $X^\adj$-runs\/} is defined via $\Delta^\adj(i)=[n-i_++1,n-i_-]$; the $X^\adj$-runs are numbered consecutively from right to left.
In a similar way, $\Gamma_2$ defines another two partitions of $[1,n]$ into $Y$-runs 
$\bar\Delta(i)$ and $Y^\adj$-runs $\bar\Delta^\adj(i)$.

Runs of length one are called trivial. The map $\gamma$ induces a bijection on the sets of pairs of nontrivial $X$- and $X^\adj$-runs and 
$Y$- and $Y^\adj$-runs.  Abusing notation, we denote by the same $\gamma$ and say that 
$(\bar\Delta_i,\bar\Delta^\adj_i)=\gamma(\Delta_j,
\Delta^\adj_j)$ if there exists $k\in\Delta_j$ such that $\bar\Delta(\gamma(k))=\bar\Delta_i$. The inverse of the bijection $\gamma$ is naturally denoted $\gamma^*$.

The {\it BD graph\/} $\BD_\bfG$ is defined as follows. The vertices of $\BD_\bfG$ are two copies of the set of 
positive simple roots identified with $[1,n-1]$. One of the sets is called the {\it upper\/} part of the graph, 
and the other is called the
{\it lower\/} part. A vertex $i\in\Gamma_1$ is connected with an {\it inclined\/} edge to the vertex 
$\gamma(i)\in\Gamma_2$. Finally, vertices $i$ and $n-i$ in the same part are connected with a {\it horizontal\/} edge. 
If $n=2k$ and $i=n-i=k$, the corresponding horizontal edge is a loop. 

Given a pair of BD triples $(\bfGr, \bfGc)$, one can define a BD graph $\BD_{\bfGr, \bfGc}$ as follows. Take $\BD_{\bfGr}$ with all inclined edges directed downwards and 
$\BD_{\bfGc}$ in which all inclined edges are directed upwards. Superimpose these graphs by identifying the corresponding vertices.  In the resulting graph, for every pair of vertices 
$i, n -i$ in either top or bottom row there are two edges joining them. We give these edges opposite orientations. If $n$ is even, then we retain only one loop at each of the 
two vertices labeled $\frac{n}{2}$. The result is a directed graph $\BD_{\bfGr, \bfGc}$ on 
$2(n-1)$ vertices. For example, consider the case of $GL_7$ with 
$\bfGr=\left(\{1,2,5\}, \{1,3,4\}, 1\mapsto 4, 2\mapsto 3, 5\mapsto 1\right)$ and
$\bfGc=\left(\{3,4,6\}, \{2,3,5\}, 3\mapsto2, 4\mapsto3, 6\mapsto 5\right)$. The corresponding graph $\BD_{\bfGr, \bfGc}$ is shown 
on the left in Fig.~\ref{fig:bdgraph}. For horizontal edges, no direction is indicated, which means that they can be traversed in both directions. 

\begin{figure}[ht]
\begin{center}
\includegraphics[height=3.5cm]{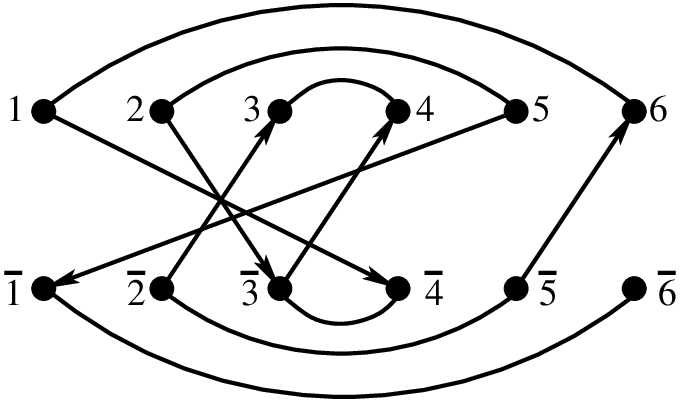}
\end{center}
\caption{BD graph $\BD_{\bfGr, \bfGc}$}
\label{fig:bdgraph}
\end{figure}

A directed path in $\BD_{\bfGr, \bfGc}$ is called {\em alternating\/} if horizontal and inclined edges in the path alternate. In particular, an edge is a (trivial) alternating path. 
An alternating path with coinciding endpoints and an even number of edges is called an {\em alternating cycle}. We can decompose the set of directed edges of $\BD_{\bfGr, \bfGc}$ into a disjoint union of maximal alternating paths and alternating cycles. 
If the resulting collection contains no alternating cycles, we call the pair $(\bfGr, \bfGc)$ {\em aperiodic\/}. 
For the graph in Fig.~\ref{fig:bdgraph}, 
the corresponding maximal alternating paths are $52\bar3\bar4$, $25\bar1\bar6$, $\bar5\bar2 34$,  $\bar2\bar5 61\bar4\bar3 43$, $16$, and $\bar6\bar1$ (here vertices in the lower part are marked with a dash for better visualization). None of them is an alternating cycle, so the corresponding pair is aperiodic.
 
Every horizontal directed edge in an upper part of the BD graph  defines a pair of {\it blocks\/} carved out from two $n\times n$ matrices: a matrix of indeterminates $X=(x_{ij})$ and the dual matrix $X^\adj$ obtained via conjugation of the cofactor matrix of $X$ by $w_0\J$
with $\J=\diag((-1)^i)_{i=1}^n$ and $w_0$ being the matrix of the longest permutation. The rows of $X$ are partitioned into $X$-runs with respect to $\bfGr$, and the columns of $X$, into $X$-runs with respect to $\bfGc$. The rows and columns of $X^\adj$ are partitioned into the corresponding dual $X^\adj$-runs: rows with respect to $\bfGr$ and columns with respect to $\bfGc$. A block in $X$ is a submatrix $X_{[\alpha,n]}^{[1,\beta]}$ whose row and column sets are unions of consecutive $X$-runs; a block in $X^\adj$ is defined similarly via $X^\adj$-runs. The $X$-block that corresponds to a horizontal directed edge $i\to (n-i)$ is the minimal block in $X$ that contains the subdiagonal through the entries $(n-i+1,1)$ and $(n,i)$. These entries are called the {\em exit point\/} and the {\em entrance point\/} of the $X$-block, respectively. Note that the exit point of an  $X$-block belongs to its uppermost $X$-run (with respect to $\bfGr$), and its entrance point belongs to the rightmost $X$-run (with respect to 
$\bfGc$). The $X^\adj$-block that corresponds to the same horizontal edge is the minimal block in $X^\adj$ that contains the subdiagonal through the entries $(i+1,1)$ and $(n,n-i)$ called the exit and the entrance points of the $X^\adj$-block; these points have similar extremal properties as the corresponding points of an $X$-block. It is easy to see that if $(i,1)$ and 
$(i^\adj,1)$ are the entry points of the $X$- and $X^\adj$-blocks corresponding to the same horizontal edge then $i+i^\adj=n+2$.

In a similar way, every horizontal directed edge in the lower part of the BD graph defines a pair 
of blocks carved out from an $n\times n$  matrix of indeterminates $Y=(y_{ij})$ and the dual matrix $Y^\adj$ obtained from cofactor matrix of $Y$ by the same procedure as above. 
The rows and columns of $Y$ are partitioned into $Y$-runs with respect to $\bfGr$ and $\bfGc$, respectively. The row and columns of $Y^\adj$ are partitioned into the corresponding $Y^\adj$-runs. A block in $Y$ is a submatrix $Y_{[1,\bar\alpha]}^{[\bar\beta,n]}$ whose row and column sets are unions of consecutive $Y$-runs; a block in $Y^\adj$ is defined similarly via 
$Y^\adj$-runs.
The $Y$-block that corresponds to a horizontal directed edge $i\to (n-i)$ is the minimal block in $Y$ that contains the superdiagonal through the entries $(1,n-i+1)$ and $(i,n)$. The $Y^\adj$-block that corresponds to the same 
horizontal edge is the minimal block in $Y^\adj$ that contains the superdiagonal 
through the entries $(1,i+1)$ and $(n-i,n)$. The exit and the entrance points retain their meaning and have similar extremal properties: namely, the exit point belons to the leftmost $Y$- or $Y^\adj$-run with respect to $\bfGc$, and the entrance point belongs to the lower $Y$- or 
$Y^\adj$-run with respect to $\bfGr$. If $(1,j)$ and 
$(1,j^\adj)$ are the entry points of the $Y$- and $Y^\adj$-blocks corresponding to the same horizontal edge then $j+j^\adj=n+2$.

For the BD graph shown in Fig.~\ref{fig:bdgraph}, the rows of $X$ are partitioned into the
$X$-runs $\Delta_1^\er=[1,3]$, $\Delta_2^\er=[4,4]$, $\Delta_3^\er=[5,6]$, and $\Delta_4^\er=[7,7]$; the first and the thrid are nontrivial. The columns of $X$ are partitioned into the $X$-runs
$\Delta_1^\ec=[1,1]$, $\Delta_2^\ec=[2,2]$, $\Delta_3^\ec=[3,5]$, $\Delta_4^\ec=[6,7]$; the last two are nontrivial. Consequently, the dual partition of rows and columns of $X^\adj$ is given by
$(\Delta_1^\er)^\adj=[5,7]$, $(\Delta_2^\er)^\adj=[4,4]$, $(\Delta_3^\er)^\adj=[2,3]$,  
$(\Delta_4^\er)^\adj=[1,1]$, and $(\Delta_1^\ec)^\adj=[7,7]$, $(\Delta_2^\ec)^\adj=[6,6]$, 
$(\Delta_3^\ec)^\adj=[3,5]$, $(\Delta_4^\ec)^\adj=[1,2]$. Thus, the $X$-block defined by the edge $5\to2$ in the upper part is the submatrix $X^{[1,5]}$, and the corresponding $X^\adj$-block is
the submatrix $(X^\adj)_{[5,7]}^{[1,2]}$.  Similarly, the rows of $Y$ are  partitioned into the $Y$-runs $\bar\Delta_1^\er=[1,2]$, $\bar\Delta_2^\er=[3,5]$, 
$\bar\Delta_3^\er=[6,6]$, and $\bar\Delta_4^\er=[7,7]$; the first two of them are nontrivial.
The columns of $Y$ are partitioned into the $Y$-runs
$\bar\Delta_1^\ec=[1,1]$, $\bar\Delta_2^\ec=[2,4]$, $\bar\Delta_3^\ec=[5,6]$, 
$\bar\Delta_4^\ec=[7,7]$; the second and the third are nontrivial. Consequently, the dual partition of rows and columns of $Y^\adj$ is given by
$(\bar\Delta_1^\er)^\adj=[6,7]$, $(\bar\Delta_2^\er)^\adj=[3,5]$, 
$(\bar\Delta_3^\er)^\adj=[2,2]$, $(\bar\Delta_4^\er)^\adj=[1,1]$, and 
$(\bar\Delta_1^\ec)^\adj=[7,7]$, $(\bar\Delta_2^\ec)^\adj=[4,6]$, 
$(\bar\Delta_3^\ec)^\adj=[2,3]$, $(\bar\Delta_4^\ec)^\adj=[1,1]$. Thus, the $Y$-block defined by the edge $\bar3\to\bar4$ in the lower part is the submatrix $Y_{[1,5]}^{[5,7]}$, and the corresponding $Y^\adj$-block is the submatrix $(Y^\adj)_{[1,5]}^{[4,7]}$. 

Every maximal alternating path defines a pair of matrices glued from blocks defined above that correspond to horizontal edges of the path. There are two types of gluing: row-to-row gluing governed by the BD triple $\bfGr$ and column-to-column gluing governed by the BD triple $\bfGc$. The first situation occurs when we consider three consecutive edges in an alternating path such that the first of them is a horizontal edge $i\to(n-i)$ in the upper part, the second one is an inclined edge $(n-i)\to k$ with $k=\gammar(n-i)$, and the third one is the horizontal edge 
$k\to(n-k)$ in the lower part. Assume that $n-i$ belongs to an $X$-run $\Delta_j^\er$ and 
$k$ belongs to a $Y$-run $\bar\Delta_m^\er$; as explained above, this means that 
$(\bar\Delta_m^\er,(\bar\Delta_m^\er)^\adj)=\gammar(\Delta_j^\er,(\Delta_j^\er)^\adj)$. Note that each $X$-run defined by $\bfGr$ contains a connected component of $\Gamma_1^\er$, while 
each $Y$-run defined by $\bfGr$ contains a connected component of $\Gamma_2^\er$. If the restriction of $\gammar$ to this connected component is oriented we glue $\Delta_j^\er$ to $\bar\Delta_m^\er$ 
and $(\Delta_j^\er)^\adj$ to $(\bar\Delta_m^\er)^\adj$. If the restriction of $\gammar$ reverses the orientation, we glue $\Delta_j^\er$ to $(\bar\Delta_m^\er)^\adj$ and $(\Delta_j^\er)^\adj$ to 
$\bar\Delta_m^\er$. 

For example, consider the path $52\bar3\bar4$ in the BD graph shown in Fig.~\ref{fig:bdgraph}. The corresponding blocks were described above. The exit point of the $X$-block $X^{[1,5]}$ belongs to  $\Delta_1^\er=[1,3]$, and the corresponding connected component of $\Gamma_1^\er$ is 
$[1,2]$. The entry point of the $Y$-block $Y_{[1,5]}^{[5,7]}$ belongs to 
$\bar\Delta_2^\er=[3,5]$, and the corresponding connected component of $\Gamma_2^\er$ is 
$[3,4]$. The map $\gammar$ on $[1,2]$ reverses the orientation, so $\Delta_1^\er$ is glued to
$(\bar\Delta_2^\er)^\adj=[3,5]$ and $(\Delta_1^\er)^\adj=[5,7]$ is glued to $\bar\Delta_2^\er$. The resulting matrices are shown in Fig.~\ref{fig:bigmatrices1}. All entries outside the blocks are equal to zero. The numbers of the rows that are glued are indicated in the figure.

\begin{figure}[ht]
\begin{center}
\includegraphics[height=3.5cm]{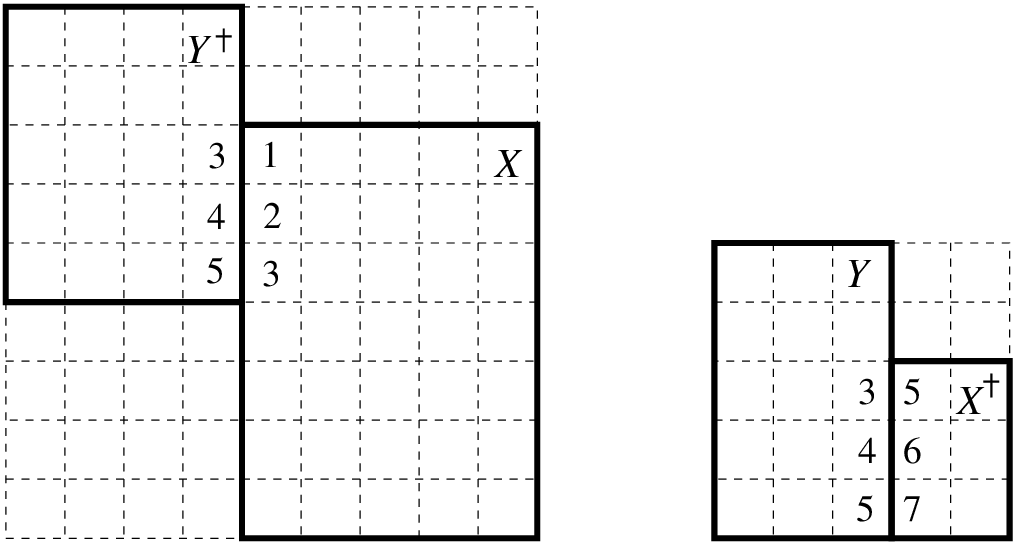}
\caption{The pair of matrices corresponding to the path $52\bar3\bar4$}
\label{fig:bigmatrices1}
\end{center}
\end{figure}

The column-to-column gluing occurs when we consider three consecutive edges in an alternating path such that the first of them is a horizontal edge $i\to(n-i)$ in the lower part, the second one is an inclined edge $(n-i)\to k$ with $n-i=\gammac(k)$, and the third one is the horizontal edge $k\to(n-k)$ in the upper part. Assume that $n-i$ belongs to a $Y$-run $\bar\Delta_j^\ec$ and 
$k$ belongs to an $X$-run $\Delta_m^\ec$; as explained above, this means that 
$(\bar\Delta_j^\ec,(\bar\Delta_j^\ec)^\adj)=\gammac(\Delta_m^\ec,(\Delta_m^\ec)^\adj)$.  If the restriction of $\gammac$ to the connected component  of $\Gamma_1^\ec$ contained in $\Delta_m^\ec$
is oriented we glue $\bar\Delta_j^\ec$ to $\Delta_m^\ec$ 
and $(\bar\Delta_j^\ec)^\adj$ to $(\Delta_m^\ec)^\adj$. If the restriction of $\gammac$ reverses the orientation, we glue $\bar\Delta_j^\ec$ to $(\Delta_m^\ec)^\adj$ and $(\bar\Delta_j^\ec)^\adj$ to $\Delta_m^\ec$. 

For example, consider the path $\bar5\bar2 34$ in the BD graph shown in Fig.~\ref{fig:bdgraph}. The exit point of the $Y$-block $Y_{[1,5]}^{[2,7]}$ belongs to  $\bar\Delta_2^\ec=[2,4]$, 
and the corresponding connected component of $\Gamma_2^\ec$ is 
$[2,3]$. The entry point of the $X$-block $X_{[5,7]}^{[1,5]}$ belongs to 
$\Delta_3^\ec=[3,5]$, and the corresponding connected component of $\Gamma_1^\ec$ is 
$[3,4]$. The map $\gammac$ on $[3,4]$ preserves the orientation, so $\bar\Delta_2^\ec$ is glued to
$\Delta_3^\ec$ and $(\bar\Delta_2^\ec)^\adj=[4,6]$ is glued to $(\bar\Delta_3^\ec)^\adj=[3,5]$. The resulting matrices are shown in Fig.~\ref{fig:bigmatrices2}. All entries outside the blocks are equal to zero. The numbers of the columns that are glued are indicated in the figure.

\begin{figure}[ht]
\begin{center}
\includegraphics[height=3.5cm]{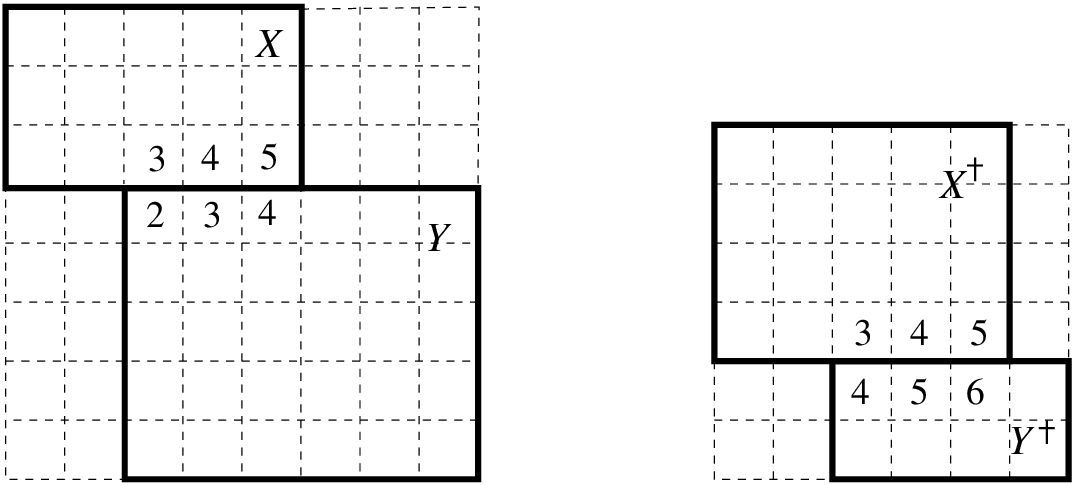}
\caption{The pair of matrices corresponding to the path $\bar5\bar234$}
\label{fig:bigmatrices2}
\end{center}
\end{figure}

In general, the number of blocks in the obtained pair of matrices is equal to the number of horizontal edges in the alternating path. Clearly, the entrance point of the first block is its lower right corner, while the exit point of the last block is its upper left corner, so each of the obtained matrices is square. The pair of matrices defined by the alternating path   
$\bar2\bar5 61\bar4\bar3 43$ is shown in Fig.~\ref{fig:bigmatrices3}.

\begin{figure}[ht]
\begin{center}
\includegraphics[height=5cm]{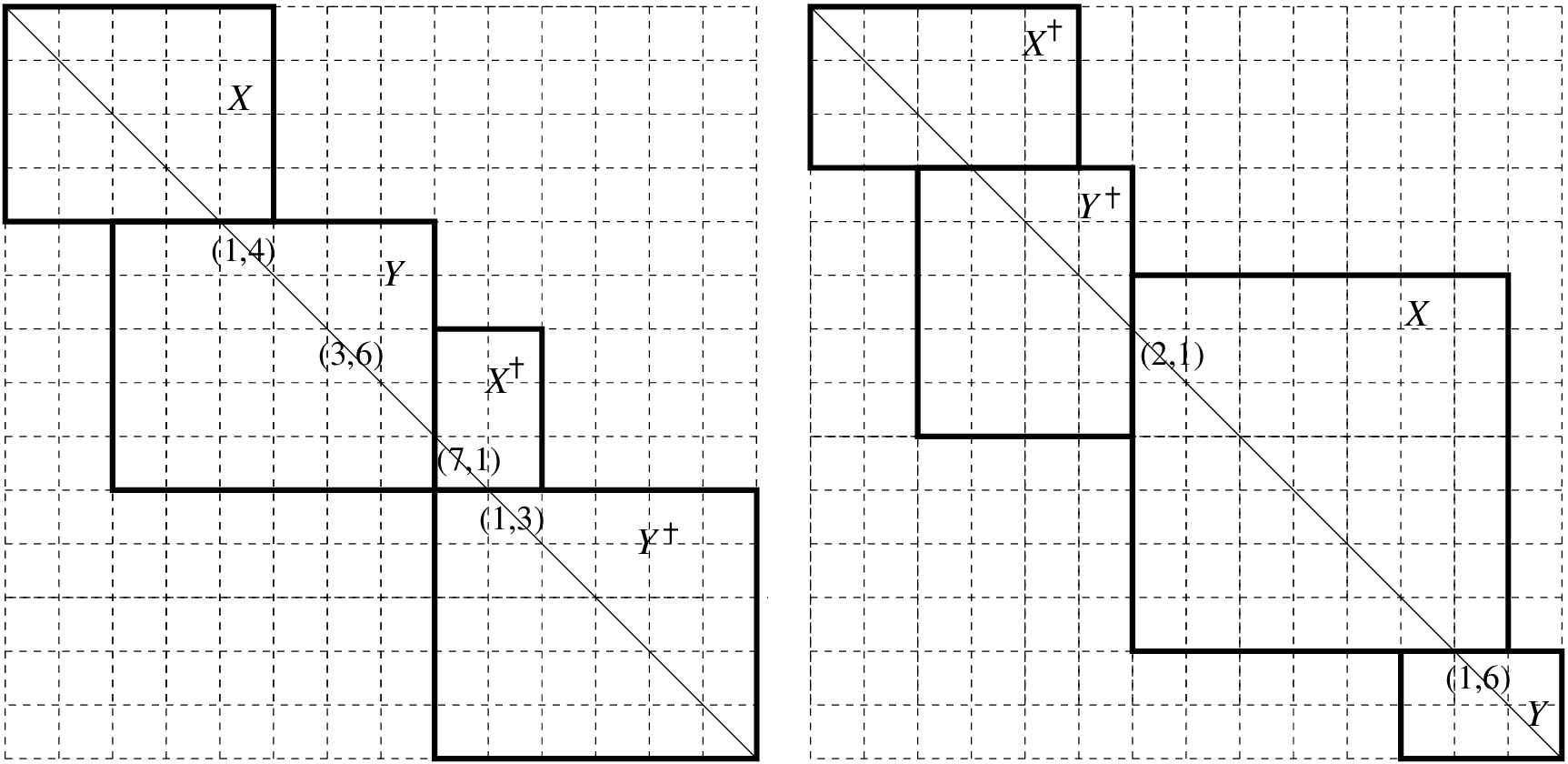}
\caption{The pair of matrices corresponding to the path $\bar2\bar5 61\bar4\bar3 43$}
\label{fig:bigmatrices3}
\end{center}
\end{figure}

Let $\L$ and $\L^\adj$ be two matrices corresponding to a maximal alternating path in $G_{\bfGr,\bfGc}$. Every initial segment 
of this path that ends with a horizontal edge defines a pair of distinguished trailing submatrices of $\L$ and $\L^\adj$ that are built of blocks that correspond to the horizontal edges in this segment. In the example shown in 
Fig.~\ref{fig:bigmatrices3}, the initial segment $\bar 2\bar 561$ defines a $7\times4$ submatrix in the lower right corner of the matrix on the left and a $8\times7$ submatrix in the lower right corner of the matrix on the right; both these submatrices consist of two blocks. 

For every pair of matrices as above we consider the set of principal trailing minors with the following property: the upper left corner of the minor belongs to an $X$-block or to a $Y$-block. 
For example, for the pair of matrices in Fig.~\ref{fig:bigmatrices1} these are the first three principal trailing minors of the second matrix and the last five of the first one. For the pair
of matrices in Fig.~\ref{fig:bigmatrices2} these are all principal trailing minors of the first matrix, while none of the second are used. 

\begin{remark} The situation described in the last example, when one of the matrices in the pair is built solely of $X$- and $Y$-blocks, and the other one is built solely of $X^\adj$- and $Y^\adj$-blocks occurs for all pairs if both $\bfGr$ and $\bfGc$ are oriented. For this reason we only needed one block matrix for each alternating path in our constructions in~\cite{GSVple}.
\end{remark}

Consider a trailing minor as above, and assume that its upper left corner contains an entry 
$x_{ij}$ (for $i>j$) or $y_{ij}$ (for $i<j$). We denote this minor by $\ttf_{ij}(X,Y)$ and its 
restriction to the diagonal $X=Y=Z$ by $f_{ij}(Z)$. Additionally, define $f_{ii}(Z)$ as the trailing minor of $Z$ in rows and columns $[i,n]$.

The following claim follows immediately from the construction.

\begin{proposition}
\label{onetoone}
For any pair $(i,j)\in[1,n]\times[1,n]$ there exists a unique function $\ttf_{ij}$. The upper left corner of the corresponding minor belongs to the $X$-block $X(i,j)$ defined by the horizontal edge $(n-i+j)\to(i-j)$ in the upper part of the graph (for $i>j$), or to the $Y$-block $Y(i,j)$ defined by the horizontal edge
$(n+i-j)\to (j-i)$ in the lower part of the graph (for $i<j$).
\end{proposition}. 

It follows from Proposition~\ref{onetoone} that we can unambiguously define $\L(i,j)$ for $i\ne j$ as the matrix corresponding to the maximal alternating path that goes through the horizontal edge $(n-i+j)\to(i-j)$ in the upper part of the graph (for $i>j$), or through the horizontal edge
$(n+i-j)\to (j-i)$ in the lower part of the graph (for $i<j$); let $\L^\adj(i,j)$ stand for the other matrix defined by the same path.  For example, the pair of matrices shown
in Fig.~\ref{fig:bigmatrices3} can be described as $\L(4,1)$ and $\L^\adj(4,1)$, or $\L(4,7)$ 
and $\L^\adj(4,7)$, or $\L^\adj(2,1)$ and $\L(2,1)$, etc.
It is clear from the definition that the matrices $\L(i,j)$ and $\L^\adj(i,j)$ depend only on the difference $i-j$.

Theorem~3.4 in~\cite{GSVple} claims that in the oriented case the family $\{f_{ij}\}$ forms a log-canonical coordinate system on $SL_n$ with respect to the Poisson bracket defined by the pair 
$\bfGr$, $\bfGc$. The proof is very technical and occupies 40 pages. Below we deduce a generalization of this result, which covers both the oriented and the non-oriented cases, from 
Theorem~\ref{twosidedpoisson}.

\subsection{The basis} The goal of this Section is the proof of the following generalization of Theorem~3.4 in~\cite{GSVple}.

Define $F_{\bfGr,\bfGc}=\{f_{ij}(Z):i,j\in[1,n], (i,j)\ne(1,1)\}$.

\begin{theorem}
\label{logcanbasis}
Let $(\bfGr,\bfGc)$ be an aperiodic pair of BD triples, then the family $F_{\bfGr,\bfGc}$ forms a log-canonical coordinate system on $SL_n$ with respect to the Poisson bracket 
$\Poi_{r^{\bfGc},r^{\bfGr}}$.
\end{theorem}

\begin{proof}
Let $F_{ij}(A)$ denote the trailing minor of $A$ whose upper left corner contains the entry 
$a_{ij}$. By Theorem~4.18 in~\cite{GSVb} (see also Theorem~5.2 in~\cite{GSV1}) functions $F_{ij}$ are log-canonical with respect to the standard Sklyanin bracket. The proof of Theorem~\ref{logcanbasis} is an immediate consequence of this fact, Theorem~\ref{twosidedpoisson}, and the following statement.

For an arbitrary pair $(i,j)$, $i\ne j$, consider 
the pair of matrices $\L(i,j)$ and $\L^\adj(i,j)$. We say that an exit point $(k,1)$ of an $X$-block  (or an exit point $(1,m)$ of a $Y$-block) is  {\em subordinate\/} to $(i,j)$
if either it or the exit point $(k^\adj,1)$ of the dual 
$X^\adj$-block (the exit point $(1,m^\adj)$ of the dual $Y^\adj$-block, respectively) belongs to the main diagonal of the matrix $\L(i,j)$ and lies below or to the right of the block $X(i,j)$ 
(or $Y(i,j)$). For example, consider the entry $(3,6)$ that belongs to the block $Y(3,6)$
in the left matrix in Fig.~\ref{fig:bigmatrices3}. The exit points subordinate to $(3,6)$ are $(2,1)$ and $(1,6)$ in the matrix on the right,  since the exit points $(7,1)$ and $(1,3)$ of the corresponding dual blocks lie to the right of $Y(3,6)$. The exit point $(1,4)$ is not subordinate to $(3,6)$.

Let $(k_1,1),\dots,(k_s,1)$ and $(1,m_1),\dots,(1,m_t)$ be all exit points subordinate to $(i,j)$ 
(note that $s-t$ is either $0$ or $\pm1$).

\begin{theorem}
\label{bigthroughsmall}
Let $h$ be the Poisson map defined in Theorem~\ref{twosidedpoisson}, 
and let $f_{ij}^h(U)=f_{ij}\circ h(U)$, then 
\begin{equation}
\label{bts}
 f_{ij}^h(U)=F_{ij}(U)\prod_{p=1}^s F_{k_p,1}(U)\prod_{q=1}^t F_{1,m_q}(U).
\end{equation}
\end{theorem}

\begin{remark} For $i=j$~\eqref{bts} holds trivially as $f_{ii}^h(U)=F_{ii}(U)$.
\end{remark}

\begin{proof}
We start from stating an invariance property of the functions $\ttf_{ij}$ that is a direct generalization of the invariance property (4.11) in~\cite{GSVple}.

\begin{proposition}
\label{invariance}
Let $f=\ttf_{ij}$ for some $(i,j)$, then for any $N_+\in\N_+$ and $N_-\in\N_-$
\[
f(N_+X(\bgammac)^*(N_-),\bgammar(N_+)YN_-)=f(X,Y).
\]
\end{proposition}

\begin{proof} It follows from the construction of the matrices described above that if
an $X$-block is multiplied on the left by $N_+$  then to keep the same value of $f$ the $Y$-block immediately to the left of this $X$-block should be multiplied on the left by $\bgammar(N_+)$.
Similarly, if a $Y$-block is multiplied on the right by $N_-$, the $X$-block immediately above it
should be multiplied on the right by $(\bgammac)^*(N_-)$. 
\end{proof}

\begin{remark}
\label{geninvariance}
In fact, it follows from the proof that one can choose different matrices $N_+$ for different $X$-blocks, and different matrices $N_-$ for different $Y$-blocks.
\end{remark}

Let us apply this Proposition for $N_+=(H^\er\bar V^\er)^{-1}$ and $N_-=(\bar V^\ec H^\ec)^{-1}$, 
then we get
\begin{equation}
\label{canform}
f_{ij}(Z)=\ttf_{ij}(Z,Z)=\ttf_{ij}(H^\er U H^\ec, H^\er U H^\ec)=\ttf_{ij}((\bar V^\er)^{-1}U,U(\bar V^\ec)^{-1}).
\end{equation}

The proof of Theorem~\ref{bigthroughsmall} proceeds by induction on the number of blocks in the submatrix of $\L(i,j)$ that defines the minor $f_{ij}$. If $b=1$, that is, the upper left corner of the above minor lies in the lower right block of 
$\L(i,j)$, then by \eqref{canform} it is either an $X$-block for $X=(\bar V^\er)^{-1}U$ or a $Y$-block for  
$Y=U(\bar V^\ec)^{-1}$. 
In both cases \eqref{bts} holds trivially, since left multiplication by a matrix from $\N_+$ and right multiplication by a matrix from $\N_-$ do not change minors in question. 

For $b>1$, consider the lower right block $B$ of $\L(i,j)$ and assume first that it is an $X$-block. The block Laplace expansion of $f_{ij}$ by the last block column involves minors of $X$ in the first $c$ columns, where $c$ is the width of 
$B$. Clearly, such minors are not affected by multiplication of $X$ on the right by a matrix in $\N_+$. Let
$[k,m]$ be the upper $X$-run of $B$ that contains its exit point $(k_s,1)$ (so that $c=n-k_s+1$). 
Consider once again the Gauss decomposition $U=U_-U_{0,+}$ and refactor $U_-$ as follows. Let $s_{[p,q]}$ stand for the product of reflections $s_ps_{p+1}\dots s_q$ for $p<q$ and $s_ps_{p-1}\dots s_q$ for $p>q$. Besides, let $\N_-(r)$ stand
for the subset of matrices in $\N_-$ such that the $r\times r$ submatrix in the lower left corner is upper triangular.
The factor $U_-$ corresponds to the reduced expression $s_{[1,n-1]}s_{[1,n-2]}\dots s_{[1,2]}s_1$ for the longest permutation 
$w_0$. First, we use 2-moves to rewrite it as 
\[
(s_{[1,m-1]}s_{[1,m-2]}\dots s_{[1,2]}s_1)(s_{[m,1]}s_{[m+1,1]}\dots s_{[n-1,1]}). 
\]
Next, we replace the left reduced expression above by its opposite 
\[
s_{[m-1,1]}s_{[m-1,2]}\dots s_{[m-1,m-2]}s_{m-1}. 
\]
Finally, we use 2-moves to rewrite it as
\[ 
(s_{[m-1,k]}s_{[m-1,k+1]}\dots s_{[m-1,m-2]}s_{m-1})(s_{[k-1,m-1]}s_{[k-2,m-1]}\dots s_{[1,m-1]}).
\]
The corresponding factorization of $U_-$ is $U_-=U_-^{[k,m]}U_LU_R$ with $U_-^{[k,m]}\in\N_-^{[k,m]}$, 
$U_L\in\N_-^{[1,m]}(m-k+1)$, and $U_R^{-1}\in\N_-(m)$ (the latter inclusion can be observed from the fact that the corresponding reduced word is $s_{[1,n-1]}s_{[1,n-2]}\dots s_{[1,m]}$). Consequently, we get $U=U_-^{[k,m]}U_LU_RU_{0,+}$, which can be further refactored as $U=U_-^{[k,m]}U_LU'_{0,+}U'_R$ with $U'_{0,+}\in\B_+$ and $(U'_R)^{-1}\in\N_-(n-m)$.

Recall that $(\bar V^\er)^{-1}$ has a block-diagonal structure, and its blocks correspond to $X$-runs defined by $\bfGr$. The block that corresponds to the $X$-run $[k,m]$ is $(U_-^{[k,m]}w_0^{[k,m]})_+^{-1}$. Note that 
\begin{multline*}
(U_-^{[k,m]}w_0^{[k,m]})_+^{-1}U_-^{[k,m]}=(U_-^{[k,m]}w_0^{[k,m]})_+^{-1}(U_-^{[k,m]}w_0^{[k,m]})w_0^{[k,m]}=\\
(U_-^{[k,m]}w_0^{[k,m]})_{0,-}w_0^{[k,m]},
\end{multline*}
and hence
\[
X=(\bar V^\er)^{-1}U=\hat V(U_-^{[k,m]}w_0^{[k,m]})_{0,-}w_0^{[k,m]}U_LU'_{0,+}U'_R.
\]

Note that $\hat V$ retains all blocks of $(\bar V^\er)^{-1}$ except for the one corresponding to $[k,m]$, so that
\[
\hat V(U_-^{[k,m]}w_0^{[k,m]})_{0,-}w_0^{[k,m]}=\begin{bmatrix} \hat V_1 & 0 & 0\\
                                                                   0 & \hat V_2 & 0\\
																                                    0 & 0 &\hat V_3
\end{bmatrix}
\]																
with $\hat V_1$ upper triangular of size $(k-1)\times(k-1)$, $\hat V_2$ is lower anti-triangular of size 
$(m-k+1)\times(m-k+1)$ and $\hat V_3$ is upper triangular of size $(n-m)\times(n-m)$. Further, 
\[
U_L=\begin{bmatrix} U_L^{11} & 0 & 0 \\
                    U_L^{21} & U_L^{22} & 0\\
											0 & 0 & \one
											\end{bmatrix}
\]
where $U_L^{11}$ is of size $(k-1)\times(m-k+1)$, $U_L^{22}$ is of size $(m-k+1)\times(m-k+1)$, and $U_L^{21}$ is upper triangular of size $(m-k+1)\times(m-k+1)$ since $U_L\in\N_-^{[1.m]}(m-k+1)$. Consequently
\begin{equation}
\label{firsthalf}
\hat V(U_-^{[k,m]}w_0^{[k,m]})_{0,-}w_0^{[k,m]}U_L=\begin{bmatrix} \star & \star & 0 \\
                                                                 \hat U_-^{21} & \star & 0\\
											                                              0 & 0 & \hat U_-^{33}
											\end{bmatrix}
\end{equation}
where $\hat U_-^{21}$ is lower anti-triangular of size $(m-k+1)\times(m-k+1)$, $\hat U_-^{33}$ is upper triangular of size $(n-m)\times(n-m)$, and shapes of all submatrices denoted by $\star$ are not relevant for this discussion. 

To proceed further, we need the following technical statement.

\begin{lemma}
\label{shapes}
Let $M$ be an $n\times n$ matrix such that $M^{-1}\in \N_-(r)$. Write $M$ as $M=\begin{bmatrix} M_{11} & M_{12}\\
M_{21}& M_{22}\end{bmatrix}$ where $M_{12}$ is of size $r\times r$, then $C(M)=M_{12}-M_{11}M_{21}^{-1}M_{22}$ is upper triangular.
\end{lemma} 

\begin{proof} Indeed, write $M^{-1}$ as $M^{-1}=\begin{bmatrix} \tilde M_{11} & 0\\
\tilde M_{21}& \tilde M_{22}\end{bmatrix}$ where $\tilde M_{21}$ is of size $r\times r$, and hence upper triangular. Then 
$M_{11}\tilde M_{11}+M_{12}\tilde M_{21}=\one$ and $M_{21}\tilde M_{11}+M_{22}\tilde M_{21}=0$, so that
$(M_{12}-M_{11}M_{21}^{-1}M_{22})\tilde M_{21}=\one$, and the claim follows.
\end{proof} 

We apply this Lemma to the matrix $M=U_R'$ with $r=m$, and get that $C(U_R')$ is an $m\times m$ upper triangular matrix. 
Recall that multiplying $X$ on the right by a matrix in $\N_+$ does not affect $f_{ij}$; we thus multiply $X$ by  
$K=\begin{bmatrix}\one & -M_{21}^{-1}M_{12}\\ 0&\one\end{bmatrix}$. Note that $U_R'K=\begin{bmatrix}\star & C(U_R')\\
\star& 0\end{bmatrix}$; multiplication by $U_{0,+}'$ on the left does not change the shape of the result, that is, the upper right $m\times m$ submatrix remains upper triangular.  Comparing this with \eqref{firsthalf} yields
\[
XK=\begin{bmatrix} \star&\star&\star\\
                   \star& \tilde X&\star\\
									\star & 0 & 0
\end{bmatrix}
\]
where $\tilde X$ is an $(m-k+1)\times (m-k+1)$ lower anti-triangular matrix and the submatrix in the lower left corner is of size $(n-m)\times(n-m)$.	Consequently, any minor of $XK$ in the first $c$ columns and rows $R\cup[m+1,n]$ for
$R\subset[k,m]$, $|R|=c-n+m=m-k_s+1$, vanishes unless $R=[k_s,m]$. The corresponding minor is exactly 
$F_{k_s,1}(XK)=F_{k_s,1}(X)=F_{k_s,1}(U)$.In the Laplace expansion for $f_{ij}$ by the last block column it us multiplied by the minor $f'_{ij}$ similar to $f_{ij}$. It has $b-1$ blocks: the block $B$ is deleted and the previous $Y$- or $Y^\adj$-block is truncated by deletion of the last $m-k_s+1$ rows. All the exit points subordinate to $(i,j)$ remain the same except for $(k_s,1)$ that disappears. So, by induction  \eqref{bts} holds for $f'_{ij}$ with $s-1$ factors in the first prodcut, hence it holds  for $f_{ij}\circ h(U)=f'_{ij}\circ h(U)\cdot F_{k_s,1}(U)$.								

Assume now that the lower right block $B$ of $\L(i,j)$ is an $X^\adj$-block. In this case the Laplace expansion by the last block column involves minors of $X^\adj$. By the Jacobi's complementary minor formula for the minors of the adjugate matrix, 
\begin{equation}
\label{jacobi}
|X_I^J|=|(X^\adj)_{\overline{w_0I}}^{\overline{w_0J}}|,
\end{equation}
 where bar stands for the complement and $w_0$ moves each index $p$ to $n-p+1$. The sign $(-1)^{\Sigma I+\Sigma J}$ in the Jacobi's formula is compensated by the conjugation by the signature matrix $\J$. Let $[m^\adj,k^\adj]$ be the upper $X^\adj$-run of $B$. The minors involved in the Laplace expansion lie in the first $c^\adj=n-k_s^\adj+1$ columns and in rows $R^\adj\cup[k^\adj+1,n]$ for $R^\adj\subset[m^\adj,k^\adj]$ and  $|R^\adj|=k^\adj-k_s^\adj+1$. By \eqref{jacobi} such minors correspond bijectively to the minors of the $X$-block studied above since $|R|+|R^\adj|=m-k+1=k^\adj-m^\adj+1$. Consequently, all of them vanish except for the one that corresponds to $R^\adj=[k_s^\adj,k^\adj]$, which is equal to $F_{k_s,1}(U)$.

The case when the  lower right block $B$ of $\L(i,j)$ is a $Y$-block is treated similarly to the case of and $X$-block. In this case $U_+$ is refactored and $Y$ is multiplied from the left by a lower triangular matrix $K'$ so that 
\[
K'Y=\begin{bmatrix} \star&\star&\star\\
                   \star& \tilde Y&0\\
									\star & \star & 0
\end{bmatrix}
\]
where $\tilde Y$ is a lower anti-triangular matrix whose size is equal to the size of the leftmost $Y$-run of $B$, so that the only non-vanihing minor of $B$ involved in the Laplace expansion by the last block row is $F_{1,m_t}(U)$. The case when $B$ is a $Y^\adj$-block is treated via the Jacobi's complementary minor formula exactly as above. 

\end{proof}
\end{proof}

Note that the double product in the right hand side of \eqref{bts} that defines the ratio 
$f_{ij}^h(U)/F_{ij}(U)$
depends only on the difference $i-j$; we denote it $t_{i-j}(U)$. Consequently, 
\begin{equation}\label{sameratio}
\frac{f_{i+1,j+1}^h(U)}{f_{ij}^h(U)}=\frac{F_{i+1,j+1}(U)}{F_{ij}(U)}
\end{equation}
for $1\le i,j<n$. Further,
\begin{equation}\label{tails}
\begin{aligned}t_{i-n}(U)&=
\begin{cases}
f^h_{(\gammar)^*(i)+1,1}(U)\quad\text{if $i\in\Gamma^\er_2$},\\
1\qquad\qquad\qquad\quad\text{otherwise},
\end{cases}\\
t_{n-j}(U)&=
\begin{cases}
f^h_{1,\gammac(j)+1}(U)\quad\text{if $j\in\Gamma^\ec_1$},\\
1\qquad\qquad\qquad\text{otherwise}.
\end{cases}	
\end{aligned}
\end{equation}		
If $(\gammar)^*$ keeps the orientation of the connected component of $\Gamma^\er_2$ that contains 
$i$, or, respectively, $\gammac$ keeps the orientation of the connected component of 
$\Gamma^\ec_1$ that contains $j$, the above formulas follow immediately from~\eqref{bts}. If the orientation is reversed, it is enough to note that
	by~\eqref{jacobi}, each minor in the product that defines $t_{i-j}(U)$ can be replaced by the corresponding minor of the dual block. 																             

It follows from the proof of Theorem~\ref{bigthroughsmall} that formulas similar to \eqref{bts} are valid for 
certain other minors of matrices  $\L$ and $\L^\adj$ restricted to the diagonal $X=Y=Z$.
 Slightly abusing notation, we will write $\L\circ h(U)$ instead of $\L(h(U),h(U))$, etc.
In particular, let $\L(i,1)$ be of size $N\times N$, 
and let $p$ be such that $\L(i,1)_{pp}=x_{i1}$. Recall that this entry of $\L(i,1)$ belongs to an $X$-block 
$X_{[\alpha,n]}^{[1,\beta]}$ with
$\alpha=(i-1)_-+1$. Similarly, let $\L^\adj(i,1)$ be of size $N^\adj\times N^\adj$, and let $p^\adj$ be such that 
$\L^\adj(i,1)_{p^\adj p^\adj}=x_{i^\adj 1}$, and let $(X^\adj)_{[\alpha^\adj,n]}^{[1,\beta^\adj]}$ be the $X^\adj$-block dual to $X_{[\alpha,n]}^{[1,\beta]}$. 

\begin{proposition}\label{perturminors}
{\rm (i)} Let $I\subset[\alpha,n]$ be an arbitrary subset of size $n-i+1$, then
\begin{equation}\label{primalpertur}
\det\L(i,1)^{[p,N]}_{(I-i+p)\cup[n-i+1+p,N]}\circ h(U)=\det h^\er(U)_{I}^{[1,n-i+1]}\cdot t_{i-1}(U),
\end{equation}
where $I+\gamma$ denotes the shift of $I$ by $\gamma$.

{\rm (ii)} Let $I^\adj\subset[\alpha^\adj,n]$ be an arbitrary subset of size $n-i^\adj+1$, then
\begin{equation}\label{dualpertur}
\det\L^\adj(i,1)^{[p^\adj,N^\adj]}_{(I^\adj-i^\adj+p^\adj)\cup[n-i^\adj+1+p^\adj,N^\adj]}\circ h(U)=
\det (h^\er(U)^\adj)_{I^\adj}^{[1,n-i^\adj+1]}\cdot t_{i-1}(U).
\end{equation}
\end{proposition}

Note that~\eqref{primalpertur} for $I=[i,n]$ coincides with~\eqref{bts} for $j=1$. There are similar formulas for 
the minors of $\L(1,j)$ and $\L^\adj(1,j)$, but we will not reproduce them here.

\begin{remark}
\label{mostgeneral}
Formulas~\eqref{primalpertur} and~\eqref{dualpertur} can be generalized even further. They remain valid if one replaces the top block in the left hand side with the corresponding block of an 
arbitrary $n\times n$ matrix $A$ and $h^\er(U)$ in the right hand side with $AH^\ec(U)^{-1}$.
\end{remark}

\subsection{The quiver} The goal of this Section is to describe the quiver $Q_{\bfGr,\bfGc}$ and to prove that the seed 
$\Sigma=(F_{\bfGr,\bfGc},Q_{\bfGr,\bfGc})$ defines a cluster structure 
$\CC_{\bfGr,\bfGc}$ compatible with a Poisson bracket $\Poi_{r^\bfGr,r^\bfGc}$. This provides a generalization of Theorem~
3.19 in~\cite{GSVple} that dealt only with oriented BD data. Moreover, the proof is much simpler that the one 
in~\cite{GSVple}. It is based on Theorem~\ref{bigthroughsmall} and avoids complicated calculations.

The quiver has $n^2-1$ vertices labeled $(i,j)$. The function attached to a vertex $(i,j)$ is $f_{ij}$. It is convenient to
describe the quiver with an additional  dummy frozen vertex $(1,1)$ that corresponds to $f_{11}=|X|=1$.  In fact,
the latter quiver corresponds to the cluster structure on $GL_n$ defined by $(\bfGr,\bfGc)$.

\begin{figure}[ht]
\begin{center}
\includegraphics[height=7cm]{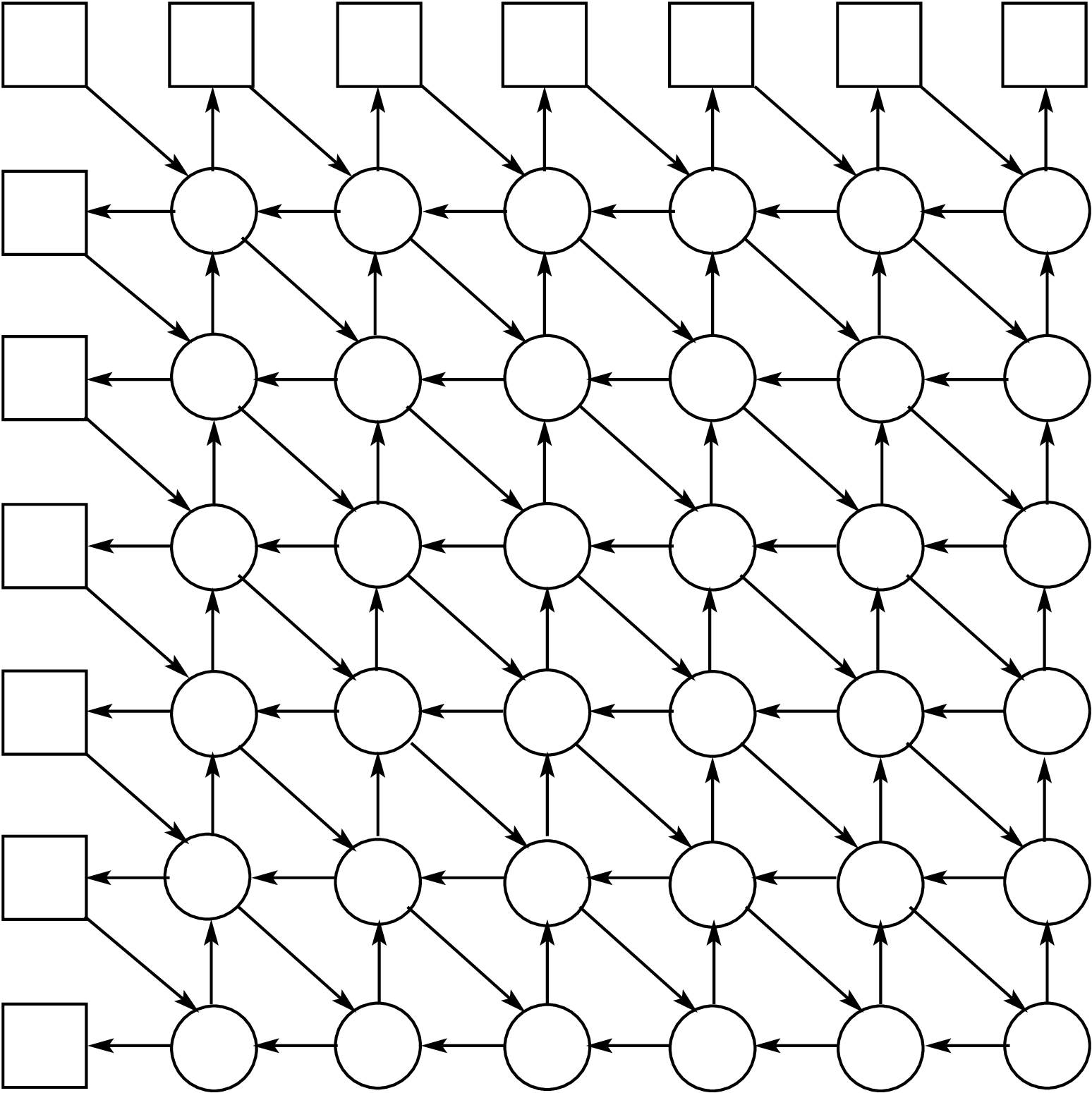}
\caption{Quiver $Q_{\varnothing,\varnothing}$ for $SL_7$}
\label{fig:quiver00}
\end{center}
\end{figure}

Recall first how looks the quiver $Q_{\varnothing,\varnothing}$. This quiver corresponds to the
standard cluster structure built for the open double Bruhat cell in~\cite{BFZ} and extended to the whole group in~\cite{GSVM}.
All vertices in the first row and column are frozen, all other vertices are mutable. The quiver $Q_{\varnothing,\varnothing}$ for $SL_7$ is presented in Fig.~\ref{fig:quiver00}.

The quiver $Q_{\bfGr,\bfGc}$ is obtained from $Q_{\varnothing,\varnothing}$ in the following way. 
For every row $X$-run $[k,m]$, the vertex $(k,1)$ remains frozen, and all other vertices $(k+1,1),\dots,(m,1)$ become mutable. If $\gammar$ preserves the orientation of the connected component $[k,m-1]\in\Gamma^\er_1$ then the following two paths are added:
$(m,1)\to(m-1,1)\to\dots\to(k,1)$ and $(\gammar(k),n)\to(k+1,1)\to(\gammar(k+1),n)\to(k+2,1)\to\dots\to(\gammar(m-1),n)\to(m,1)\to(\gammar(m-1)+1,n)$.
If $\gammar$ reverses the orientation of the connected component $[k,m-1]$ then the following two paths are added: $(k+1,1)\to(k,1)$ and $(\gammar(m-1),n)\to(m,1)\to(\gammar(m-2),n)\to(m-1,1)\to\dots\to(\gammar(k),n)\to(k+1,1)\to(\gammar(k)+1,n)$. 

Similarly, for every column $Y$-run $[p,q]$, the vertex $(1,p)$ remains frozen, and all other vertices
$(1,p+1),\dots,(1,q)$ become mutable. If $(\gammac)^*$ preserves the orientation of the connected component 
$[p, q-1]\in \Gamma^\ec_2$ that corresponds to the run $[p,q]$ then the following two paths are added:
$(1,q)\to(1,q-1)\to\dots\to(1,p)$ and
$(n,(\gammac)^*(p))\to(1,p+1)\to(n,(\gammac)^*(p+1))\to(1,p+2)\to\dots\to(n,(\gammac)^*(q-1))\to(1,q)\to
(n,(\gammac)^*(q-1)+1)$.  
If $(\gammac)^*$ reverses the orientation of the connected component $[p,q-1]$ then the following two paths are added: 
$(1,p+1)\to(1,p)$ and
$(n,(\gammac)^*(q-1))\to(1,q)\to(n,(\gammac)^*(q-2))\to(1,q-1)\to\dots\to(n,(\gammac)^*(p))\to(1,p+1)\to(n,(\gammac)^*(p)+1)$. 

\begin{figure}[ht]
\begin{center}
\includegraphics[height=9cm]{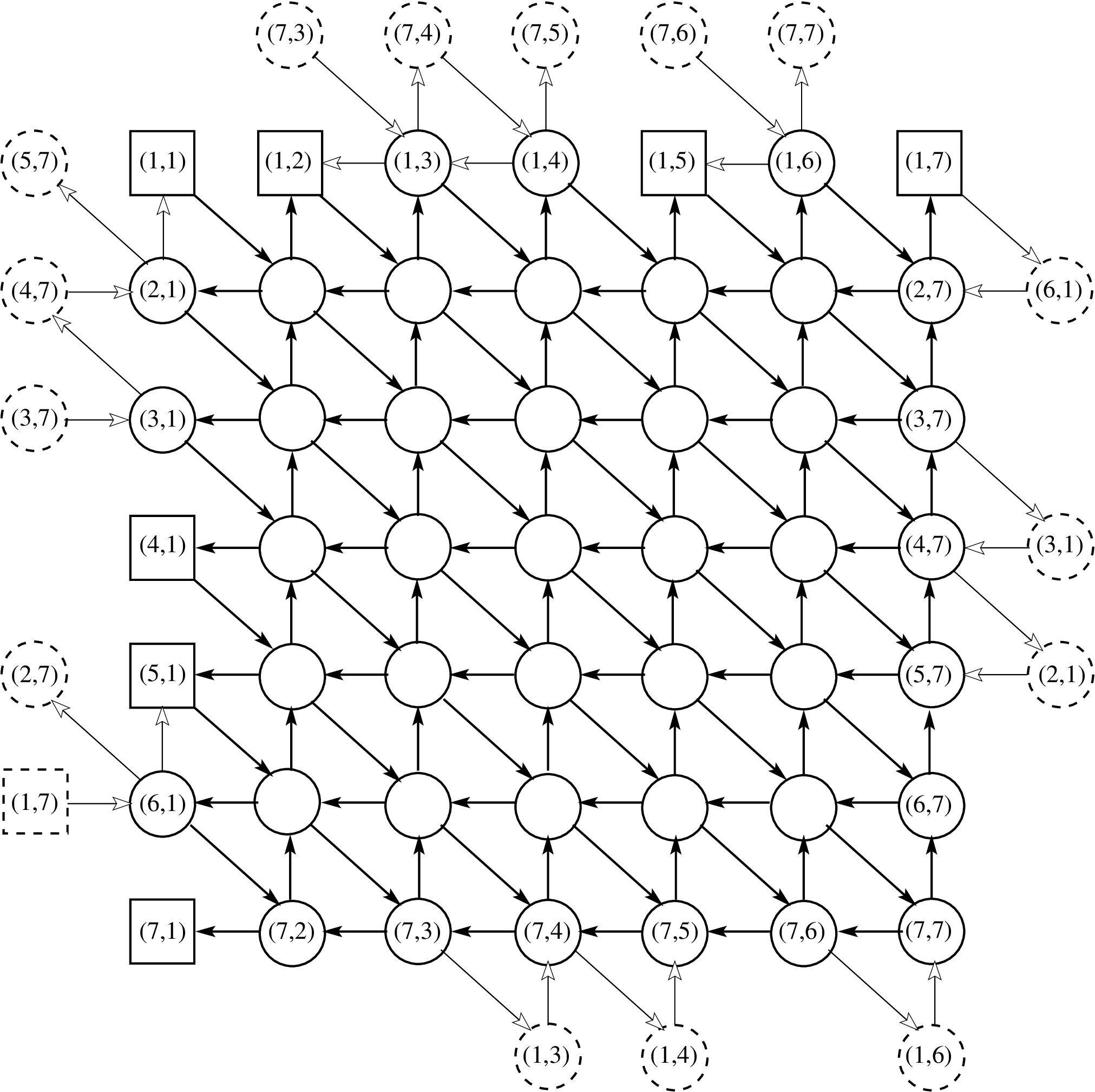}
\caption{Quiver $Q_{\bfGr,\bfGc}$ for the BD graph in Fig.~\ref{fig:bdgraph}}
\label{fig:quiver77}
\end{center}
\end{figure}

Consider our running example. As explained above, $Y$-runs defined by $\gammac$ are $[1,1]$, $[2,4]$, $[5,6]$, and $[7,7]$.
Consequently, vertices $(1,1)$, $(1,2)$, $(1,5)$, and $(1,7)$ remain frozen and vertices $(1,3)$, $(1,4)$, and $(1,6)$ 
become mutable. Further, $(\gammac)^*$ preserves the orientation of all connected components, hence the following paths are added: $(1,4)\to(1,3)\to(1,2)$ and $(7,3)\to(1,3)\to(7,4)\to(1,4)\to(7,5)$ for the component $[2,3]$ and $(1,6)\to(1,5)$ and
$(7,6)\to(1,6)\to(7,7)$ for the component $[5,5]$. Similarly, $X$-runs defined by $\gammar$ are $[1,3]$, $[4,4]$, $[5,6]$,
and $[7,7]$. Consequently, vertices $(4,1)$, $(5,1)$, and $(7,1)$ remain frozen and vertices $(2,1)$, $(3,1)$, and $(6,1)$
become mutable. Further, $\gammar$ reverses the orientation of the connected component $[1,2]$ and (trivially) preserves
the orientation of the connected component $[5,5]$, hence the following paths are added: $(2,1)\to(1,1)$ and 
$(3,7)\to(3,1)\to(4,7)\to(2,1)\to(5,7)$ for the component $[1,2]$ and $(6,1)\to(5,1)$ and $(1,7)\to(6,1)\to(2,7)$ for the componant $[5,5]$. The resulting graph is presented in Fig.~\ref{fig:quiver77}. Vertices shown as dotted circles are copies of the existing vertices and are placed to make the figure easier to comprehend. The edges of additional paths are shown by paler arrows.

\begin{theorem}\label{quiver}
Let $(\bfGr,\bfGc)$ be an aperiodic pair of BD triples, 
then the seed $(F_{\bfGr,\bfGc}, Q_{\bfGr,\bfGc})$ defines a cluster structure compatible with
the Poisson bracket $\Poi_{r^\bfGc,r^\bfGr}$ on 
$SL_n$ for any pair of R-matrices $r^\bfGc$, $r^\bfGr$ from the BD classes defined by 
$\bfGc$, $\bfGr$, respectively. 
\end{theorem}

\begin{proof}
The proof is based on the characterization of pairs of compatible Poisson and cluster structures given in~\cite{GSV1} and on Theorems~\ref{twosidedpoisson} and~\ref{bigthroughsmall} above.

Recall the definition of cluster $y$-variables associated with a seed $(\FF=(F_v)_{v\in Q},Q)$ (see \cite{GSV1, CA}): 
for any mutable $v\in Q$
\begin{equation}\label{yvariable}
y_v=\frac{\prod\limits_{v\to u}F_u}{\prod\limits_{w\to v}F_w},
\end{equation}
where $\to$ means an arrow in the quiver $Q$.

For $i,j\in[2,n]$, let $y_{ij}$ and $Y_{ij}$ be the $y$-variables that correspond to the vertex 
$(i,j)$ in the seeds $(F_{\bfGr,\bfGc},Q_{\bfGr,\bfGc})$ and 
$(F_{\varnothing,\varnothing},Q_{\varnothing,\varnothing})$, respectively; recall that 
$F_{\varnothing,\varnothing}=(F_{ij})_{i,j=1}^n$.

\begin{lemma}\label{bigythroughsmall}
For any $i,j\in[2,n]$,
\begin{equation}\label{byts}
y_{ij}^h(U)=Y_{ij}(U),
\end{equation}
where $y_{ij}^h(U)=y_{ij}\circ h(U)$.
\end{lemma}

\begin{proof} For $i,j\in[2,n-1]$ the neighborhoods of the vertex $(i,j)$ in $Q_{\bfGr,\bfGc}$
and $Q_{\varnothing,\varnothing}$ are identical, and
\[
y_{ij}^h(U)=\frac{f_{i+1,j+1}^h(U)}{f_{i-1,j-1}^h(U)}
\cdot\frac{f_{i-1,j}^h(U)}{f_{i,j+1}^h(U)}\cdot
\frac{f_{i,j-1}^h(U)}{f_{i+1,j}^h(U)}=Y_{ij}(U)
\]
by~\eqref{bts} and~\eqref{sameratio}.

For $j\in[2,n-1]$ the neighborhoods of the vertex $(n,j)$ in $Q_{\bfGr,\bfGc}$
and $Q_{\varnothing,\varnothing}$ are identical unless $j$ or $j-1$ belongs to 
$\Gamma_1^\ec$, in which case the former contains one or two additional vertices.
No matter which case occurs, we can use~\eqref{sameratio} and the second formula in~\eqref{tails} to write
\[
\begin{aligned}
y_{nj}^h(U)&=\frac{t_{n-j}(U)}{f_{n-1,j-1}^h(U)}
\cdot\frac{f_{n-1,j}^h(U)}{f_{n,j+1}^h(U)}\cdot
\frac{f_{n,j-1}^h(U)}{t_{n-j+1}(U)}\\
&=\frac1{F_{n-1,j-1}(U)}\cdot
\frac{F_{n-1,j}(U)}{F_{n,j+1}(U)}\cdot F_{n,j-1}(U)=Y_{nj}(U).
\end{aligned}
\]
A similar argument based on the first formula in~\eqref{tails} applies in the case of 
the vertex $(i,n)$, $i\in[2,n-1]$. To treat the vertex $(n,n)$ we use both arguments.
\end{proof}

 For brevity, in what follows we write $\Poi$ for 
$\Poi_{r^\pu_{\bfGc},r^\pu_{\bfGr}}$ and $\Poi_\bfG$ for $\Poi_{r^\bfGc,r^\bfGr}$. 
By Theorem 4.5 in~\cite{GSVb}, to prove 
Theorem \ref{quiver} it suffices to check relation
\begin{equation}
\label{compcondition}
\{\bar y_{ij},\bar f_{\ii\jj}\}=
\sum_{(u,v)\xrightarrow[\Gamma]{}(i,j)}\{\bar f_{uv},\bar f_{\ii\jj}\}_\bfG-
\sum_{(i,j)\xrightarrow[\Gamma]{}(u,v)}\{\bar f_{uv},\bar f_{\ii\jj}\}_\bfG=
\begin{cases} \lambda \enspace&\text{for $(\ii,\jj)=(i,j)$,}\\
0 \enspace & \text{otherwise}
\end{cases}							           
\end{equation}
for all pairs $(i,j), (\ii,\jj)$ such that $f_{ij}$ is not frozen, where $\lambda\ne0$ is fixed, $\xrightarrow[\Gamma]{}$ is an arrow in $Q_{\bfGr,\bfGc}$, and the bar over a function
stands for the logarithm of this function. 

 For $i,j\in[2,n]$, Theorems~\ref{twosidedpoisson} and~\ref{bigthroughsmall} together
with Lemma~\ref{bigythroughsmall} imply
\[
\{\bar y_{ij},\bar f_{\ii\jj}\}_\bfG=\{\bar Y_{ij},\bar F_{\ii\jj}+\bar t_{\ii-\jj}\}
=\{\bar Y_{ij},\bar F_{\ii\jj}\}=
\begin{cases} 1& \quad\text{for $(i,j)=(\ii,\jj)$},\cr
0& \quad\text{otherwise}.
\end{cases}
\]
Here the second equality uses the fact that $t_{\ii-\jj}(U)$ is a product of frozen variables for the standard cluster structure defined by the seed $(F_{\varnothing,\varnothing},Q_{\varnothing,\varnothing})$, and therefore haz a zero Poisson bracket with $Y_{ij}(U)$. The third equality follows from the compatibility of the log-canonical basis $(F_{ij})$ with the standard cluster structure,
see the proof of Theorem 4.18 in~\cite[p.98]{GSVb}. Note that the bracket 
in this theorem has the opposite sign, which is compensated by the opposite direction of the quiver, see~\cite[p.32]{GSVb}.

Consider now the case $1<i<n$, $j=1$. Assume first that $i-1\in\Gamma_1^\er$, $i\notin\Gamma_1^\er$, and $\gammar$ preserves the orientation of the connected component of $\Gamma_1^\er$ that contains $i-1$. In this subcase the index set in the first sum in~\eqref{compcondition} consists fo the vertices $(\gammar(i-1),n)$ and $(i,2)$, and the index set of the second sum consists of the vertices $(i+1,2)$, 
$(i-1,1)$, and $(\gammar(i-1)+1,n)$. Further, $\bar t_{\gammar(i-1)-n}=\bar F_{i1}+\bar t_{i-1}$, and 
$\bar t_{\gammar(i-1)+1-n}=0$. Consequently, the left hand side of~\eqref{compcondition} boils down to
\[
\left\{\bar F_{i1}-\bar F_{i-1,1}+\bar F_{i2}-\bar F_{i+1,2}+\bar F_{\gammar(i-1),n}-\bar F_{\gammar(i-1)+1,n},
\bar F_{\ii\jj}+\bar t_{\ii-\jj}\right\}\!=\!\{\Phi,\bar F_{\ii\jj}+\bar t_{\ii-\jj}\}.
\] 

Assume now that $i-1,i\in\Gamma_1^\er$, and $\gammar$ preserves the orientation of the connected component of $\Gamma_1^\er$ that contains $i-1$, so that $\gammar(i-1)+1=\gammar(i)$. In this subcase the vertex $(i+1,1)$ is added to the 
index set in the first sum in~\eqref{compcondition}.  Further, condition
$\bar t_{\gammar(i-1)+1-n}=0$ is replaced by $\bar t_{\gammar(i)-n}=\bar F_{i+1,1}+\bar t_{i}$. 
Consequently, the left hand side of~\eqref{compcondition} is given by the same expression 
$\{\Phi,\bar F_{\ii\jj}+\bar t_{\ii-\jj}\}$ as before.

Assume next that $i-1\in\Gamma_1^\er$, $i-2\notin\Gamma_1^\er$ and $\gammar$ reverses the orientation of the connected component of $\Gamma_1^\er$ that contains $i-1$. In this subcase the index sets for both sums in~\eqref{compcondition} are the same as in the first subcase, and the tails $t, \bar t$ satisfy the same conditions. 
Consequently, the left hand side of~\eqref{compcondition} is given by the same expression 
$\{\Phi,\bar F_{\ii\jj}+\bar t_{\ii-\jj}\}$ as before.

Finally, assume that $i-2,i-1\in\Gamma_1^\er$ and $\gammar$ reverses the orientation of the connected component of 
$\Gamma_1^\er$ that contains $i-1$, so that $\gammar(i-1)+1=\gammar(i-2)$. In this subcase the vertex $(i-1,1)$ is deleted from the index set in the second sum in~\eqref{compcondition}. Further, condition
$\bar t_{\gammar(i-1)+1-n}=0$ is replaced by $\bar t_{\gammar(i-2)-n}=\bar F_{i-1,1}+\bar t_{i-2}$. 
Consequently, the left hand side of~\eqref{compcondition} is given by the same expression 
$\{\Phi,\bar F_{\ii\jj}+\bar t_{\ii-\jj}\}$ as before.

To evaluate $\{\Phi,\bar F_{\ii\jj}+\bar t_{\ii-\jj}\}$, we start with studying the bracket $\Poi_0=\Poi_{r^\pu_\pu,r^\pu_\pu}$ where  
$r_\pu^\pu$ corresponds to $R^\pu_0=\frac12\pi_\ze$. For an arbitrary $\ii\in[1,n]$ and a subset $I\subset[1,n]$ define
\[
\sign(\ii-I)=\begin{cases} -1 \quad&\text{if $\ii$ is less than the minimal element in $I$,}\\
                          0\quad&\text{if $\ii\in I$,}\\
													1\quad&\text{if the maximal element in $I$ is less than $\ii$;}
				\end{cases}									
\]
otherwise $\sign(\ii-I)$ is not defined.

\begin{lemma}
\label{zerobracket}
If $\sign(\ii-I)$ and $\sign(\jj-J)$ are defined and satisfy the inequality  $|\sign(\ii-I)+\sign(\jj-J)|\leq1$ then
\begin{equation*}
\{\bar u_{\ii\jj},|\bar U_I^J|\}_0=\frac12\left(\sign(\ii-I)+\sign(\jj-J)\right).
\end{equation*}
\end{lemma}

\begin{proof} Follows immediately from \cite[equation (8.21)]{GSVb} and \cite[Lemma 4.7]{GSVb}.
\end{proof}

Further, for an arbitrary pair of functions $f_1, f_2$ we have 
 $\Delta(f_1,f_2)=\{f_1,f_2\}-\{f_1,f_2\}_0=\langle S^{\bfGc}\pi_\ze\nabla^Lf_1,\nabla^Lf_2\rangle-
\langle S^{\bfGr}\pi_\ze\nabla^Rf_1,\nabla^Rf_2\rangle$.
A straightforward computation gives $\Delta(\bar u_{ij},\bar u_{kl})=s^\bfGc_{lj}-s^\bfGr_{ki}$, hence
\begin{equation}
\label{detdif}
\Delta(\overline{\det U_I^J},\overline{\det U_{I'}^{J'}})=
\sum_{i\in J, j\in J'}s^\bfGc_{ij}-\sum_{i\in I,j\in I'}s^\bfGr_{ij}.
\end{equation}

It follows from Lemma~\ref{zerobracket} that
\[
\{\bar F_{i1}-\bar F_{i-1,1},\bar F_{\ii\jj}\}_0=
\begin{cases}
              \ \ \frac12\quad&\text{for $2\le\ii\le i-1$, $n-i+3\le\jj\le n-i+1+\ii$,}\\
							-\frac12 \quad&\text{for $i\le\ii\le n$, $\ii-i+2\le\jj\le n-i+2$,}\\
								\ \ 0\quad&\text{otherwise.}
				\end{cases}									
\]
Further,
\[
\{\bar F_{i2}-\bar F_{i+1,2},\bar F_{\ii\jj}\}_0=
\begin{cases}
              \ \ \frac12\quad&\text{for $i+1\le\ii\le n$, $\ii-i+2\le\jj\le n-i+2$, }\\ 
							\quad&\text{or $(\ii,\jj)=(i,1)$, or $(\ii,\jj)=(1,n-i+2)$,}\\
							-\frac12 \quad&\text{for $3\le\ii\le i$, $n-i+3\le\jj\le n-i+\ii$,}\\
								\ \ 0\quad&\text{otherwise,}
				\end{cases}									
\]
and
\[
\{\bar F_{\gammar(i-1),n}-\bar F_{\gammar(i-1)+1,n}\}_0=
\begin{cases}-\frac12 \quad&\text{for $\ii=\gammar(i-1)+1$}\\
                          \quad&\text{or $1\le\ii\le \gammar(i-1)$, $\jj= n-\gammar(i-1)+\ii$,}\\
							 \ \ 0\quad&\text{otherwise.}
				\end{cases}									
\]
Consequently, $\left\{\Phi,
\bar F_{\ii\jj}\right\}_0$ vanishes if $(\ii,\jj)$ does not belong to the rows $\ii=i$ and $\ii=\gammar(i-1)+1$ or to the diagonals $\jj-\ii=n-\gammar(i-1)$ and $\jj-\ii=n-i+1$. Both rows and the first of the diagonals contribute $-\frac12$, the second diagonal contributes $\frac12$. Therefore, for $i-1>\gammar(i-1)$ the second diagonal intersects the row 
$\ii=\gammar(i-1)+1$ and the contributions cancel at $(\gammar(i-1)+1,\gammar(i-1)+i-n)$, while for $i-1<\gammar(i-1)$ 
the first diagonal intersects the row $\ii=i$ and the contributions at $(i,n+i-\gammar(i-1)$ add to $-1$. Finally, the value of the bracket at $(i,1)$ equals $1/2$.

To compute $\Delta(\Phi,\bar F_{\ii\jj})$ note that the column sets for the minors involved with the positive sign are 
$[1,n-i+1]$, $[2,n-i+2]$, and 
$[n,n]$, while for the minors involved with the negative sign they are $[1,n-i+2]$, $[2,n-i+1]$, and $[n,n]$. Consequently, the contribution of the elements of $S^\bfGc$ in~\eqref{detdif} vanishes. The row sets for the minors involved with the positive sign are
$[i,n]$, $[i,n]$, and $[\gammar(i-1),\gammar(i-1)]$, while for the minors involved with the negative sign they are 
$[i-1,n]$, $[i+1,n]$, and  $[\gammar(i-1)+1,\gammar(i-1)+1]$. It follows from ~\eqref{detdif} that
\[
\Delta(\Phi,\bar F_{\ii\jj})=\sum_l\left(s^\bfGr_{i-1,l}-s^\bfGr_{il}-s^\bfGr_{\gammar(i-1),l}+s^\bfGr_{\gammar(i-1)+1,l}\right)
\]
where $l$ belongs to the row set of the minor that defines $F_{\ii\jj}$. Recall that $S^\bfGr$ is skew symmetric and 
$S^\bfGr(1-\gammar)h_\alpha=\frac12(1+\gammar)h_\alpha$, hence 
\[
 s^\bfGr_{i-1,l}-s^\bfGr_{il}-s^\bfGr_{\gammar(i-1),l}+s^\bfGr_{\gammar(i-1)+1,l}=
\begin{cases}
\ \ \frac12 \quad&\text{for $l=i$ or $l=\gammar(i-1)+1$,}\\
              -\frac12\quad&\text{for $l=i-1$ or $l=\gammar(i-1)$, }\\ 
								\ \ 0\quad&\text{otherwise.}
				\end{cases}
\]
Consequently, $\Delta(\Phi,\bar F_{\ii\jj})$ vanishes if $(\ii,\jj)$ does not belong to the same rows $\ii=i$ and $\ii=\gammar(i-1)+1$ or to the same diagonals $\jj-\ii=n-\gammar(i-1)$ and $\jj-\ii=n-i+1$. This time both rows contribute $\frac12$, and both diagonals contribute $-\frac12$. For $i-1>\gammar(i-1)$ the second diagonal intersects the row 
$\ii=\gammar(i-1)+1$ and the contributions cancel at $(\gammar(i-1)+1,\gammar(i-1)+i-n)$, while for $i-1<\gammar(i-1)$ 
the first diagonal intersects the row $\ii=i$ and the contributions cancel at $(i,n+i-\gammar(i-1)$ add to~$1$. 

Combining this result with the previous computations for the bracket $\{\Phi,\bar F_{\ii\jj}\}_0$ we see that
 $\left\{\Phi,\bar F_{\ii\jj}\right\}$ equals $1$ for 
$(\ii,\jj)=(i,1)$, $-1$ on the diagonal $\jj-\ii=n-\gammar(i-1)$ and vanishes otherwise. Consequently, 
$\left\{\Phi,\bar t_{\ii-\jj}\right\}$ equals $1$ on the diagonal $\jj-\ii=n-\gammar(i-1)$ (those $(\ii,\jj)$ for which $(i,1)$ is subordinate and $(1,n-\gammar(i-1)+1)$ is not subordinate) and vanishes otherwise. Tus, $\{\Phi,\bar F_{\ii\jj}+
\bar t_{\ii-\jj}\}$ equals $1$ for $(\ii,\jj)=(i,1)$ and vanishes otherwise, which completes the verification 
of~\eqref{compcondition} in this case.

The case $i=1$, $1<j<n$ is treated along the same lines. In this case the left hand side in~\eqref{compcondition} boils down to
\[
\left\{\bar F_{1j}-\bar F_{1,j-1}+\bar F_{2j}-\bar F_{2,j+1}+\bar F_{n,(\gammac)^*(j-1)}-\bar F_{n,(\gammac)^*(j-1)+1},
\bar F_{\ii\jj}+\bar t_{\ii-\jj}\right\}.
\] 
The latter is treated in a similar way as above. In this case the contribution of the elements of $S^\bfGr$ vanishes, and the required result follows from $S^\bfGc((\gammac)^*-1)h_\alpha=\frac12(1+(\gammac)^*)h_\alpha$ and the skew symmetry of $S^\bfGc$.

Cases $i=n$, $j=1$ and $i=1$, $j=n$ are treated similarly taking into account $\bar t_{n-1}=\bar F_{1,\gammac(1)+1}+
\bar t_{-\gammac(1)}$ for $1\in\Gamma_1^\ec$ and $\bar t_{1-n}=\bar F_{(\gammar)^*(1)+1,1}+\bar t_{(\gammar)^*(1)}$ for 
$1\in\Gamma_2^\er$.
\end{proof}

\subsection{Regularity} Recall that a cluster structure in the field of rational functions on a quasi-affine variety is called {\it regular\/} if every variable in every cluster is a regular function. By~\cite[Proposition~3.11]{GSVple}, to prove regularity it is enough to exhibit a regular cluster such that all adjacent clusters are regular as well.
The goal of this section is to extend the regularity result of Theorem~6.1 in~\cite{GSVple} to the general case of an aperiodic pair $(\bfGr,\bfGc)$.

\begin{theorem}\label{regneighbour}
For any mutable cluster variable $f_{ij}\in F_{\bfGr,\bfGc}$, the adjacent variable $f_{ij}'$ is a regular function on $SL_n$.
\end{theorem}

\begin{proof}
We start with the following auxiliary statement. Assume that $i-1\in\Gamma_1^\er$ and that $\gammar$ reverses the orientation of the connected component of $\Gamma_1^\er$ that contains $i-1$. Let this component be $[i-1-s,i-1+t]$, $s+t>0$. Consider the pair of dual matrices $\L(i,1)$ and $\L^\adj(i,1)$
restricted to the diagonal $X=Y$. Abusing notation, we denote them by the same symbols $\L(i,1)$ and $\L^\adj(i,1)$. This should not lead to confusion since from now on we will only deal with matrices subject to this restriction.

Denote by $M$ and $M^\adj$ the pair of square trailing submatrices of $\L(i,1)$ and $\L^\adj(i,1)$, respectively, such that the entry in the upper left corner of $M$ is $x_{i1}$, and the entry in the upper left corner of $M^\adj$ is $x^\adj_{i^\adj 1}$; recall that by definition, $f_{i1}=\det M$. Let $r$ denote the size of 
$M$ and $r^\adj$ denote the size of $M^\adj$. Note that the first row of $M$ is an initial segment of the row $X_i$. 
For $1\le j\le s$ and $0\le k\le t-1$ define an $r\times r$ matrix $M(j,k)$ via deleting row $k+1$ from $M$ and adding the corresponding segment of row $X_{i-j}$ on top of the obtained matrix. Similarly, the first row of $M^\adj$ is an initial segment of the row $X^\adj_{i^\adj}$; define an $r^\adj\times r^\adj$ matrix $M^\adj(j,k)$ via adding  the corresponding segment of row $X^\adj_{i^\adj-k-1}$ on top of $M^\adj$ and deleting row $j+1$ of the obtained matrix. 

\begin{lemma} \label{minorcorresp}
For any $1\le j\le s$ and $0\le k\le t-1$
\[
\det M(j,k)=\det M^\adj(j,k).
\]
\end{lemma}

\begin{proof} Define $I(j,k)=(i-j)\cup\left([i,n]\setminus(i+k)\right)$ and 
$I^\adj(j,k)=\left([i^\adj,n]\cup(i^\adj-k-1)\right)\setminus(i^\adj+j-1)$, then 
\[
\begin{aligned}
M(j,k)&=\L(i,1)^{[p,N]}_{(I(j,k)-i+p)\cup[n-i+1+p,N]},\\
M^\adj(j,k)&=\L^\adj(i,1)^{[p^\adj,N^\adj]}_{(I^\adj(j,k)-i^\adj+p^\adj)\cup[n-i^\adj+1+p^\adj,N^\adj]},
\end{aligned}
\]
and hence $\det M(j,k)$ and $\det M^\adj(j,k)$ are 
particular cases of minors studied in Lemma~\ref{perturminors}. Note that $I^\adj(j,k)=\overline{w_0I(j,k)}$, 
so by~\eqref{jacobi} and Lemma~\ref{perturminors} it follows that $\det M(j,k)\circ h(U)=\det M^\adj(j,k)\circ h(U)$. It remains to note that $h$ is invertible, as explained in Section~\ref{sectiononesided}.
\end{proof}

\begin{remark}\label{dualrepr}
(i) The statement of the lemma remains true for $(j,k)=(0,0)$, in which case $I(0,0)=[i,n]$ and $I^\adj(0,0)=[i^\adj,n]$,
so that $M(0,0)=M$ and $M^\adj(0,0)=M^\adj$; the proof goes without any changes.The obtained equality gives an alternative representation $f_{i1}=\det M^\adj$.

 (ii) In fact, the statement of the lemma holds also for the corresponding minors of $\L(X,Y)$ and $\L^\adj(X,Y)$
and can be proved directly by using block-Laplace expansions.
\end{remark}

We can now proceed with the proof of Theorem~\ref{regneighbour}. Assume first that we want to prove the regularity of $f_{ij}'$ for $1<i<n$, $1<j<n$. Recall that the approach suggested in~\cite{GSVple} consists of the following steps. If 
$p=\deg f_{ij}<\deg f_{i-1,j}=m$, we define an $m\times(m+1)$ submatrix $A$ of $\L(i-1,j)$ such that 
$A_{12}=x_{i-1,j}$. Note that 
\begin{equation}\label{typical1}
\begin{aligned} 
&f_{i-1,j}=\det A^{\hat 1}, \qquad f_{i,j+1}=\det A^{\hat 1\hat 2}_{\hat 1}, \\ 
&f_{i-1,j-1}\cdot\det B=\det A^{\widehat {m+1}},\ \ f_{ij}\cdot\det B=\det A^{\hat 1\widehat {m+1}}_{\hat 1}
\end{aligned}
\end{equation}
with $B=A_{[p+2,m]}^{[p+2,m]}$; here and in what follows ``hatted'' subscripts and superscripts indicate deleted rows and columns, respectively.

Applying the Desnanot--Jacobi identity for matrices of size $d\times(d+1)$ we get
\begin{equation}\label{firstdj}
f_{i-1,j}\cdot\det\bar A_{\hat 1}^{\hat 2}+f_{i-1,j-1}f_{i,j+1}=f_{ij}\cdot\det A^{\hat 2}
\end{equation}
where $\bar A=A_{[1,p+1]}^{[1,p+1]}$ has the property $f_{i-1,j-1}=\det \bar A$.

If $p=\deg f_{ij}\ge\deg f_{i-1,j}=m$, we define a $(p+1)\times(p+2)$ matrix $A$ by taking the submatrix of $\L(i-1,j-1)$ 
whose upper left entry equals $x_{i-1,j-1}$ and adding on the right the column $[0,\dots,0,1]^T$. Similarly 
to~\eqref{typical1}, we have
\begin{equation}\label{typical2} 
\begin{aligned}
&f_{i-1,j}\cdot\det B=\det A^{\hat 1}, \ \ f_{i,j+1}\cdot\det B=\det A^{\hat 1\hat 2}_{\hat 1}, \\ 
&f_{i-1,j-1}=\det A^{\widehat {p+2}},\qquad f_{ij}=\det A^{\hat 1\widehat {p+2}}_{\hat 1}
\end{aligned}
\end{equation}
with $B=A_{[m+1,p+1]}^{[m+2,p+2]}$. Applying the same Desnanot--Jacobi identity we arrive at the same 
equation~\eqref{firstdj}, see Section~6.1 in~\cite{GSVple} for more details.

Next, we compare $\deg f_{ij}$ with $\deg f_{i,j-1}$ and consider in a similar way two cases
$\deg f_{ij}<\deg f_{i,j-1}$ and $\deg f_{ij}\ge\deg f_{i,j-1}$, both producing equation
\begin{equation}\label{seconddj}
f_{ij}\cdot\det C_{\hat 1}^{\hat 2}+f_{i+1,j+1}f_{i,j-1}=f_{i+1,j}\cdot\det A^{\hat 2}
\end{equation}
where $C$ is the square submatrix of $\L(i,j-1)$ with the property $f_{i,j-1}=\det C$ and $\bar A$ is the same as 
in~\eqref{firstdj}. The linear combination of~\eqref{firstdj} and~\eqref{seconddj} with coefficients $f_{i+1,j}$ and 
$f_{i-1,j}$, respectively, yields
\[
f_{ij}(f_{i+1,j}\det A^{\hat 2}-f_{i-1,j}\det C^{\hat 2}_{\hat 1})=f_{i-1,j-1}f_{i,j+1}f_{i+1,j}+
f_{i-1,j}f_{i,j-1}f_{i+1,j+1}.
\] 
Combining this with the description of the quiver $Q_{\bfGr,\bfGc}$ given in the previous section we see that
$f_{ij}'=f_{i+1,j}\det A^{\hat 2}-f_{i-1,j}\det C^{\hat 2}_{\hat 1}$ is a regular function. Note that the above 
reasoning does not depend on whether $\gammar$ and $\gammac$ reverse orientation or preserve it.

Consider now $f_{in}'$ for $1\le i\le n$. Assume first that both $i-1$ and $i$ belong to $\Gamma_2^\er$ and that $(\gammar)^*$ preserves the orientation of the corresponding connected component of $\Gamma_2^\er$, that is, 
$(\gammar)^*(i)=(\gammar)^*(i-1)+1$. Then the above reasoning remains valid with $f_{i,j+1}$ 
in~\eqref{typical1}-\eqref{typical2} replaced by $f_{(\gammar)^*(i),1}$ and $f_{i+1,j+1}$ in~\eqref{seconddj} replaced by
$f_{(\gammar)^*(i)+1,1}$. The resulting equation reads
\begin{multline*}
f_{in}(f_{i+1,n}\det A^{\hat 2}-f_{i-1,n}\det C^{\hat 2}_{\hat 1})\\
=f_{i-1,n-1}f_{(\gammar)^*(i),1}f_{i+1,n}+
f_{i-1,n}f_{i,n-1}f_{(\gammar)^*(i)+1,1},
\end{multline*}
and hence $f_{in}'=f_{i+1,n}\det A^{\hat 2}-f_{i-1,n}\det C^{\hat 2}_{\hat 1}$ is a regular function. If $(\gammar)^*$ reverses
the orientation of the connected component of $\Gamma_2^\er$ that contains $i-1$ and $i$ then 
$(\gammar)^*(i-1)=(\gammar)^*(i)+1$. Using the alternative representation of $f_{(\gammar)^*(i-1),1}$ and 
$f_{(\gammar)^*(i-1)+1,1}$ provided by Remark~\ref{dualrepr}, we apply the same reasoning as above with $f_{i,j+1}$ 
in~\eqref{typical1}-\eqref{typical2} replaced by $f_{(\gammar)^*(i-1),1}$ and $f_{i+1,j+1}$ in~\eqref{seconddj} replaced by
$f_{(\gammar)^*(i-1)+1,1}$. The resulting equation reads
\begin{multline*}
f_{in}(f_{i+1,n}\det A^{\hat 2}-f_{i-1,n}\det C^{\hat 2}_{\hat 1})\\
=f_{i-1,n-1}f_{(\gammar)^*(i-1),1}f_{i+1,n}+
f_{i-1,n}f_{i,n-1}f_{(\gammar)^*(i-1)+1,1},
\end{multline*}
which yields the same regular expression for $f_{in}'$. If $i-1\notin\Gamma_2^\er$ then $f_{i,j+1}$ in all formulas above is replaced by~1, which corresponds to a vertex of degree~5. Similarly, if $i\notin\Gamma_2^\er$ then $f_{i+1,j+1}$ in all formulas above is replaced by~1, which again corresponds to a vertex of degree~5. If both conditions hold simultaneously then both functions are replaced by~1, which corresponds to a vertex of degree~4.

Consider now $f_{i1}'$ for $1\le i\le n$. Assume first that both $i-1$ and $i$ belong to $\Gamma_1^\er$ and that $\gammar$ preserves the orientation of the corresponding connected component of $\Gamma_1^\er$, that is, 
$\gammar(i)=\gammar(i-1)+1$. Then the above reasoning remains valid with $f_{i-1,j-1}$ 
in~\eqref{typical1}-\eqref{typical2} replaced by $f_{\gammar(i-1),n}$ and $f_{i,j-1}$ in~\eqref{seconddj} replaced by
$f_{\gammar(i),n}$. The resulting equation reads
\[
f_{i1}(f_{i+1,1}\det A^{\hat 2}-f_{i-1,1}\det C^{\hat 2}_{\hat 1})
=f_{\gammar(i-1),n}f_{i+1,1}f_{i2}+f_{i-1,1}f_{i+1,2}f_{\gammar(i),1},
\]
and hence $f_{i1}'=f_{i+1,1}\det A^{\hat 2}-f_{i-1,1}\det C^{\hat 2}_{\hat 1}$ is a regular function. If $i\notin\Gamma_1^r$ then $f_{\gammar(i),1}$ above is replaced by $f_{\gammar(i-1)+1,1}$ while $f_{i+1,1}$ is replaced 
by~1, which corresponds to a vertex of degree~5.

Consider now the case when both $i-1$ and $i-2$ belong to $\Gamma_1^\er$ and $\gammar$ reverses the orientation of the corresponding connected component of $\Gamma_1^\er$, that is, $\gammar(i-2)=\gammar(i-1)+1$. Assume first that 
$\deg f_{i1}\ge\deg f_{i-1,1}$. Consider the $(p+1)\times p$ trailing submatrix $A$ of $\L(i,1)$ defined by the property $A_{21}=x_{i1}$.
Note that $A_{[1,m]}^{[1,m]}$ for some $m\le p$ is the submatrix of $\L(i-1,1)$ whose determinant equals $f_{i-1,1}$; we denote it $\bar M$ to distinguish it from $M$ that plays the same role for $\L(i,1)$. 
Consequently,
\begin{equation}\label{special1}
\begin{aligned}
&f_{i1}=\det A_{\hat1},\qquad f_{i+1,2}=\det A_{\hat 1\hat 2}^{\hat 1},\\
&f_{i-1,1}\cdot\det B=\det A_{\widehat{p+1}},\ \ f_{i2}\cdot\det B=\det A_{\hat 1\widehat{p+1}}^{\hat 1}
\end{aligned}
\end{equation} 
with $B=A_{[m+1,p]}^{[m+1,p]}$. Applying the Desnanot--Jacobi identity for matrices of size $(d+1)\times d$ we get
\begin{equation}\label{primaldj}
f_{i1}\cdot\det\bar M_{\hat 2}^{\hat 1}+f_{i-1,1}f_{i+1,2}=f_{i2}\cdot\det A_{\hat 2}.
\end{equation}

Next, consider the $(p^\adj+1)\times(p^\adj+1)$ trailing submatrix $A^\adj$ of $\L(\gammar(i-1),n)$ defined by the property
$A^\adj_{11}=x_{\gammar(i-1),n}$. Note that 
$\deg f_{\gammar(i-2),n}=1+\deg f_{i-1,1}\le1+\deg f_{i1}=\deg f_{\gammar(i-1),n}$, and hence 
$(A^\adj)_{[2,m^\adj+2]}^{[1,m^\adj+1]}$ for some $m^\adj$ is the submatrix of $\L(\gammar(i-2),n)$ whose determinant 
equals $f_{\gammar(i-2),n}$. Additionally, $(A^\adj)_{\hat 1}^{\hat 1}$ is exactly the submatrix $M^\adj$ of 
$\L^\adj(i,1)$ defined above, and $\bar M^\adj=(A^\adj)_{[3,m^\adj+2]}^{[2,m^\adj+1]}$ plays the same role for $\L^\adj(i-1,1)$.
Consequently, using the alternative description of $f_{i1}$ and $f_{i-1,1}$ provided by Remark~\ref{dualrepr}, we get
\begin{equation}\label{special2}
\begin{aligned}
&f_{\gammar(i-1),n}=\det A^\adj,\qquad\qquad f_{i1}=\det (A^\adj)_{\hat 1}^{\hat 1},\\
&f_{\gammar(i-2),n}\cdot\det B^\adj=\det (A^\adj)^{\widehat{p^\adj+1}}_{\hat 1},\ \ 
f_{i-1,1}\cdot\det B^\adj=\det (A^\adj)^{\hat 1\widehat{p^\adj+1}}_{\hat 1\hat 2}
\end{aligned}
\end{equation} 
with $B^\adj=(A^\adj)_{[m^\adj+3,p^\adj+1]}^{[m^\adj+2,p^\adj]}$. Applying the Desnanot--Jacobi identity for square matrices we get
\begin{equation}\label{dualdj}
f_{i1}\cdot\det (\bar A^\adj)_{\hat 2}=f_{\gammar(i-1),n}f_{i-1,1}+f_{\gammar(i-2),n}\cdot\det (A^\adj)_{\hat 2}^{\hat 1}
\end{equation}
with $\bar A^\adj=(A^\adj)_{[1,m^\adj+2]}^{[1,m^\adj+1]}$.
The linear combination of~\eqref{primaldj} and~\eqref{dualdj} with coefficients $f_{\gammar(i-1),n}$ and $f_{i+1,2}$, respectively, yields
\begin{multline}\label{longeq}
f_{i1}(f_{\gammar(i-1),n}\cdot\det\bar M_{\hat 2}^{\hat 1}+f_{i+1,2}\cdot\det (\bar A^\adj)_{\hat 2})\\=
f_{i2}f_{\gammar(i-1),n}\cdot\det A_{\hat 2}+f_{i+1,2}f_{\gammar(i-2),n}\cdot\det (A^\adj)_{\hat 2}^{\hat 1}.
\end{multline}

Note that $A_{\hat 2}$ is $M(1,0)$ and $(A^\adj)_{\hat 2}^{\hat 1}$ is $M^\adj(1,0)$ for the dual pair 
$\L(i,1)$, $\L^\adj(i,1)$, 
hence by Lemma~\ref{minorcorresp}  we get $\det A_{\hat 2}=\det (A^\adj)_{\hat 2}^{\hat 1}$, and so the right hand side of~\eqref{longeq} factors.  

Further, expand $f_{\gammar(i-1),n}$ in the left hand side of~\eqref{longeq} by the first column as 
\[
f_{\gammar(i-1),n}=\sum_{j=0}^s (-1)^j x_{\gammar(i-1)+j,n}\det M^\adj(j,0)=
\sum_{j=0}^s (-1)^j x_{\gammar(i-1)+j,n} \det M(j,0)
\]
via Lemma~\ref{minorcorresp}. Similarly, expand $\det (\bar A^\adj)_{\hat 2}$ in the left hand side of~\eqref{longeq} by the first column as 
\begin{multline*}
\det (\bar A^\adj)_{\hat 2}=x_{\gammar(i-1),n}\det\bar M^\adj+
\sum_{j=2}^s (-1)^{j-1} x_{\gammar(i-1)+j,n}\det\bar M^\adj(j-1,1)\\=
x_{\gammar(i-1),n}\det\bar M+\sum_{j=2}^s (-1)^{j-1} x_{\gammar(i-1)+j,n}\det\bar M(j-1,1)
\end{multline*}
via Lemma~\ref{minorcorresp}; note that the exit point for the block  that defines $\bar M$ is $(i-1,1)$, so $\bar M(j-1,1)$ has on top the same segment of row $X_{j-i}$ that $M(j,1)$ does. Substituting into the left hand side of~\eqref{longeq} gives
\begin{multline}\label{expandlong}
f_{\gammar(i-1),n}\cdot\det\bar M_{\hat 2}^{\hat 1}+f_{i+1,2}\cdot\det (\bar A^\adj)_{\hat 2}\\=
x_{\gammar(i-1),n}\left(\det M\det \bar M_{\hat 2}^{\hat 1}+f_{i+1,2}\det\bar M\right)
-x_{\gammar(i-1)+1,n}\det M(1,0)\det\bar M_{\hat 2}^{\hat 1}\\
+\sum_{j=2}^s (-1)^{j} x_{\gammar(i-1)+j,n}\left(\det M(j,0)\det \bar M_{\hat 2}^{\hat 1}-
f_{i+1,2}\det\bar M(j-1,1)\right).
\end{multline}

Consider the coefficient at $x_{\gammar(i-1),n}$ in~\eqref{expandlong}. Recall that by~\eqref{special1},
$M=A_{\hat 1}$, $f_{i+1,2}=\det A_{\hat 1 \hat 2}^{\hat 1}$, $\det\bar M_{\hat 2}^{\hat 1}
\det B=\det A_{\hat 2\widehat{p+1}}^{\hat 1}$, $\det\bar M\det B=\det A_{\widehat{p+1}}$ and 
$\det\bar M_{\hat 1}^{\hat 1}\det B=\det A_{\hat 1\widehat{p+1}}^{\hat 1}$, so that the Desnanot-Jacobi identity 
for the $(p+1)\times p$ matrix $A$ yields
\[
\det M\det \bar M_{\hat 2}^{\hat 1}+f_{i+1,2}\det\bar M=\det M(1,0)\det\bar M_{\hat 1}^{\hat 1}.
\]

To treat the coefficient at $x_{\gammar(i-1)+j,n}$ in~\eqref{expandlong}, consider the $(p+1)\times p$ matrix 
$A(j)$ obtained by adding the initial segment of $X_{i-j}$ on top of $M(1,0)$. Then $M(j,0)=A(j)_{\hat 2}$, 
$f_{i+1,2}=\det A(j)_{\hat 1 \hat 2}^{\hat 1}$, $\det\bar M_{\hat 2}^{\hat 1}\det B=
\det A(j)_{\hat 1\widehat{p+1}}^{\hat 1}$, $\det\bar M(j-1,1)\det B=\det A(j)_{\widehat{p+1}}$ and 
 $\det\bar M(j-1,1)_{\hat 2}^{\hat 1}\det B=\det A(j)_{\hat 2\widehat{p+1}}^{\hat 1}$, so that the Desnanot-Jacobi identity 
for the $(p+1)\times p$ matrix $A(j)$ yields
\[
\det M(j,0)\det \bar M_{\hat 2}^{\hat 1}-f_{i+1,2}\det\bar M(j-1,1)=\det M(1,0)\det\bar M(j-1,1)_{\hat 2}^{\hat 1}.
\]
Substitution of the obtained formulas into~\eqref{longeq} and cancellation of $\det M(1,0)$ in both sides yields
\begin{multline*}
f_{i1}\left(x_{\gammar(i-1),n}\det\bar M_{\hat 1}^{\hat 1}+\sum_{j=1}^s (-1)^{j}x_{\gammar(i-1)+j,n}\det\bar M(j-1,1)_{\hat 2}^{\hat 1}\right)\\
=f_{i2}f_{\gammar(i-1),n}+f_{i+1,2}f_{\gammar(i-2),n},
\end{multline*}
which means that $f_{i1}'$ is a regular function, and the degree of the vertex $(i,1)$ is~4.

If $i-2\notin\Gamma_1^\er$ the above reasoning remains valid with $f_{\gammar(i-2),n}$ replaced by 
$f_{\gammar(i-1)+1,n}f_{i-1,1}$, which yields a vertex of degree~5.

The case when $\deg f_{i1}<\deg f_{i-1,1}$ is treated in a similar way. We consider an $m\times m$ submatrix $A$ of $\L(i-1,1)$ characterized by $A_{11}=x_{i-1,1}$ and an $(m^\adj+1)\times m^\adj$ submatrix $A^\adj$ of $\L^\adj(i-1,1)$ characterized 
by $A^\adj_{21}=x_{\gammar(i-2),n}$. Reasoning along the same lines we arrive at
\begin{multline*}
f_{i1}(f_{\gammar(i-1),n}\cdot\det M_{\hat 2}^{\hat 1}+f_{i+1,2}\cdot \det A^\adj_{\hat 2})\\=
f_{i2}f_{\gammar(i-1),n}\cdot\det\bar A_{\hat 2}+f_{i+1,2}f_{\gammar(i-2),n}\cdot\det (\bar A^\adj)_{\hat 2}^{\hat 1},
\end{multline*}
which coincides with~\eqref{longeq} up to switching $M$ with $\bar M$, etc.

Functions $f_{nj}$ and $f_{1j}$ are treated in a similar way with an analog of Lemma~\ref{minorcorresp}. 
\end{proof}

\subsection{Completeness} Recall that a cluster structure in the ring of regular functions of an algebraic variety is called {\it complete\/} if the corresponding upper cluster 
algebra is naturally isomorphic to this ring.
The goal of this section is to extend the completeness result of Theorem~3.3(ii) in~\cite{GSVple} to the general case of an aperiodic pair $(\bfGr,\bfGc)$. As explained in Section~3.4 of~\cite{GSVple}, this amounts to extending  Theorem 7.1 in~\cite{GSVple} and to claim the following Laurent property.

\begin{theorem}
\label{laurent}
Let  $(\bfGr, \bfGc)$ be an aperiodic pair of Belavin--Drinfeld triples  and $\CC=\CC_{\bfGr, \bfGc}$ be the cluster structure on $SL_n$ defined by the seed $(F_{\bfGr,\bfGc}, Q_{\bfGr,\bfGc})$. 
Then every matrix entry can be written as a Laurent polynomial in the initial cluster 
$F_{\bfGr,\bfGc}$ and in any cluster adjacent to it.
\end{theorem}

\begin{proof}  We will adjust the inductive argument of the corresponding proof in~\cite{GSVple} to allow for non-oriented BD data. In the process, we will 
use Theorem~\ref{bigthroughsmall} and formulas~\eqref{sameratio},~\eqref{tails} to streamline the necessary technical results of~\cite[Section 7.1]{GSVple} even in the oriented case. 

Recall that the induction is on the total size $|\Gamma_1^\er|+|\Gamma_1^\ec|$ of
the pair  $(\bfGr, \bfGc)$. Since each step of induction involves either $\bfGr$ or $\bfGc$, but not both, we will only consider the case of reducing the size of $\bfGr$; the other case can be treated similarly.

The induction step involves removing the first or the last root $\alpha$ of a connected component of $\Gamma_1^\er$, removing its image in $\Gamma_2^\er$, and modifying $\gammar$ accordingly. We denote the BD triple resulting from the operation above by
$\tbfGr=(\tGamma_1^\er,\tGamma_2^\er,\tgammar)$. Below, for any object associated
with the pair $(\bfGr,\bfGc)$, we decorate with $\ \tilde{}\ $ 
the notation for its counterpart associated with $(\tbfGr,\bfGc)$. Since the total size of this pair is smaller, we assume that $\tilde\CC=\CC_{\tbfGr,\bfGc}$ possesses the above mentioned Laurent property. 

Let $F =\{ f_{ij}(Z) \:  i,j\in[1,n]\}$ and $\tilde F=\{ \tilde f_{ij}(Z) \:  i,j\in[1,n]\}$ 
be initial clusters for $\CC$ and $\tilde \CC$, respectively, and $Q$ and $\tilde Q$ be the corresponding quivers. It is easy to see that all maximal alternating paths in 
$\BD_{\bfGr,\bfGc}$ are preserved in $\BD_{\tbfGr,\bfGc}$ except for the path that goes through the directed inclined edge $\alpha\to \gammar(\alpha)$. The latter one is split into two: the initial segment up to the vertex $\alpha$ and the closing segment
starting with the vertex $\gammar(\alpha)$. Consequently, the only difference between
$Q$ and $\tilde Q$ is that the vertex $v=(\alpha+1,1)$ that corresponds to the chosen endpoint of
$[k,m-1]$ is mutable in $Q$ and frozen in $\tilde Q$, and that the neighborhoods of this vertex in $Q$
and $\tilde Q$ are different. This allows to invoke Proposition~7.4 in ~\cite{GSVple}. Namely, 
define 
\[
\lambda_{ij} = \frac{\deg f_{ij}(Z) - \deg\tilde f_{ij}(\tilde Z)}{\deg\tilde f_{\alpha+1,1}(\tilde Z)}
\]
 and choose $\Phi=\{\tilde f_{\alpha+1,1}^{\lambda_{ij}}\tilde f_{ij}\}$  as an  initial cluster associated with $Q$; note that we do not go beyond polynomials since it will be shown below that 
$\lambda_{ij}$ defined as above are integers. Then
 if  $\tilde\varphi$ is obtained via a sequence of mutations avoiding $v$ applied to the seed 
$(\tilde F, \tilde Q)$, then the same sequence of mutations applied to the seed $(\Phi, Q)$ 
yields $\varphi= \tilde f_{\alpha+1,1}^\lambda\tilde\varphi$ for some integer 
$\lambda$.

To implement the induction step, we need the following statement which
is a simultaneous extension of Theorems~7.2
and~7.3 in~\cite{GSVple} to the case of arbitrary aperiodic pairs of Belavin--Drinfeld triples.

\begin{theorem}
\label{inductionstep}
There exists a unipotent upper triangular matrix $C=C(\tilde Z)$ whose entries are rational functions in $\tilde x_{ij}$ with 
denominators equal to powers of $\tilde f_{\alpha+1,1}(\tilde Z)$ such that $Z=C\tilde Z$ and
\[
f_{ij}(Z)=
\begin{cases} \tilde f_{ij}(\tilde Z)\tilde f_{\alpha+1,1}(\tilde Z) & \quad\text{if $(\alpha+1,1)$ is
subordinate to $(i,j)$ for $(\bfGr,\bfGc)$,}\\
\tilde f_{ij}(\tilde Z) & \quad\text{otherwise.}
\end{cases}
\]
\end{theorem}

 It follows from Theorem~\ref{inductionstep} that $\lambda_{ij}$ defined above is equal to $1$ if $(\alpha+1,1)$ is subordinate to $(i,j)$ for $(\bfGr,\bfGc)$, and equal to $0$ otherwise. Since, additionally, $\tilde f_{\alpha+1,1}(\tilde Z) =  f_{\alpha+1,1}(Z)$, we conclude that any Laurent polynomial in $\tilde F$ is also a Laurent polynomial in $F$, and any  Laurent polynomial in variables of the cluster in $\tilde\CC$ obtained by mutation of $\tilde F$ in a direction {\em other than $v=(\alpha+1,1)$} is also a Laurent polynomial in the cluster in $\CC$ obtained by mutation of $F$ in the same direction. By inductive assumption, every matrix entry 
$\tilde z_{ij}$ can be expressed as a Laurent polynomial in  $\tilde F$ or any cluster adjacent to it.
The first claim of Theorem~\ref{inductionstep} then implies that for any of these clusters except the one obtained by mutation in the direction $v$, the entries of $C=C(\tilde Z)$, and therefore of 
$Z=C\tilde Z$, are Laurent polynomials in the corresponding cluster in $\CC$. To verify the claim of Theorem~\ref{laurent} for the cluster in $\CC$ obtained by mutation of the initial one in the direction $v$, we apply the same induction step to a different root in $\bfGr$ or, if $\bfGr=\{\alpha\}$, apply a similar procedure to a root in $\bfGc$. In the latter case, we use an analogue of Theorem \ref{inductionstep} that can be easily obtained by transposition. The case $|\Gamma_1^\er|+|\Gamma_1^\ec|=1$ serves as the base of induction; it was handled in Section~7.3 of~\cite{GSVple}. Thus, to complete the proof of Theorem~\ref{laurent} we only need to finish 

\begin{proof}[Proof of Theorem~\ref{inductionstep}]
First, we compare functions $f_{ij}(Z)$ and $\tilde f_{ij}(\tilde Z)$ in the initial seeds of the two cluster structures using formula~\eqref{bts}. The pair of indices $(i,j)$, $i\ne j$, defines uniquely a directed horizontal edge $e(i,j)=(n-i+j)\to(i-j)$ in the upper part of the BD graph for $i>j$ and a directed horizontal edge $e(i,j)=(n+i-j)\to(j-i)$ in the lower part of the BD graph for $i<j$.  
Note that despite the functions themselves depend on the whole BD graph, the right hand side of this formula can be read off directly from the maximal alternating path through $e(i,j)$ and does not depend on the rest of the graph. Indeed, each factor in the right hand side of~\eqref{bts} corresponds to a minor of $U$ defined by a directed horizontal edge preceding the edge $e(i,j)$ in the alternating path. Further, the exit points of all such blocks are subordinate to $(i,j)$  
for both $(\bfGr,\bfGc)$ and $(\tbfGr,\bfGc)$. As an immediate consequence, we conclude that if maximal alternating paths that correspond to $(i,j)$ in $G$ and $\tilde G$ coincide up to and including $e(i,j)$ then $f_{ij}^h(U)=
\tilde f_{ij}^{\tilde h}(U)$, which immediately yields
\begin{equation}
\label{fafter}
f_{ij}(Z)=\tilde f_{ij}(\tilde Z),
\end{equation}
since $\tilde Z=\tilde h\circ h^{-1}(Z)$.

It is easy to see that all maximal alternating paths in $G$ are preserved in $\tilde G$ except for the path $P$ that goes through the directed inclined edge $\alpha\to \gammar(\alpha)$. The latter one is split into two: the initial segment up to the vertex $\alpha$ and the closing segment
starting with the vertex $\gammar(\alpha)$. Using~\eqref{bts} and the reasoning above, we conclude that if  the inclined edge $\alpha\to\gammar(\alpha)$ precedes $e(i,j)$ in $P$ then
\begin{equation}
\label{fbefore}
f_{ij}(Z)=\tilde f_{ij}(\tilde Z)\tilde f_{\alpha+1, 1}(\tilde Z),
\end{equation}
since the horizontal edge in $P$ that immediately preceeds $\alpha\to \gammar(\alpha)$ is
$(n-\alpha)\to\alpha$, which corresponds to $f_{\alpha+1,1}$. Note that in this case the exit point
$(\alpha+1,1)$ is subordinate to $(i,j)$ for $(\bfGr,\bfGc)$.

Recall that by~\eqref{XthruU}, $Z=H^\er(U)UH^\ec(U)$ and $\tilde Z=\tilde H^\er(U)U\tilde H^\ec(U)$.
Additionally, $\tilde H^\ec(U)=H^\ec(U)$ since $\bfGc$ is the same in both cases. Consequently,
\begin{equation}
\label{XthrutX}
Z=H^\er(U)\tilde H^\er(U)^{-1}\tilde Z=C\tilde Z.
\end{equation}
To complete the proof of Theorem~\ref{inductionstep} we have to check that the entries of the matrix 
$C$ as functions in $\tilde Z$ obtained via $U=\tilde h^{-1}(\tilde Z)$ are rational with denominators equal to powers of $\tilde f_{\alpha+1,1}(\tilde Z)$.  

Assume that $[k,m-1]$ is the connected component of $\Gamma^\er_1$ to which the induction step is being applied; denote $p=m-k+1$. Define the subgroup $\KK\subset\N_+^{\Gamma_1^\er}$ via
\begin{equation}
\label{KK}
\KK=
\begin{dcases}
\left\{\diag\left(\one_{k-1},\begin{bmatrix}\one_{p-1} & \xi^T \cr 0 & 1\end{bmatrix},\one_{n-m}\right)\right\} 
\quad&\text{for $\alpha=m-1$,}\\
\left\{\diag\left(\one_{k-1},\begin{bmatrix} 1 & \xi \cr 0 & \one_{p-1}\end{bmatrix},\one_{n-m}\right)\right\} 
\quad&\text{for $\alpha=k$,}
\end{dcases}
\end{equation}
where $\xi=(\xi_1,\dots,\xi_{p-1})$. Every $p\times p$ unipotent upper triangular matrix $A$ can be uniquely factored in every one of the following four ways:
\begin{equation}
\label{fourfactor}
\begin{aligned}
A &= \begin{bmatrix}A_1 & 0 \cr 0 & 1\end{bmatrix} \begin{bmatrix}\one_{p-1} & \xi_1^T \cr 0 & 1\end{bmatrix} =
\begin{bmatrix}\one_{p-1} & \xi_2^T \cr 0 & 1\end{bmatrix} \begin{bmatrix}A_1 & 0 \cr 0 & 1\end{bmatrix} \\
&=
\begin{bmatrix}1 & 0 \cr 0 & A_2\end{bmatrix} \begin{bmatrix} 1 & \xi_3 \cr 0 & \one_{p-1}
\end{bmatrix} =
\begin{bmatrix} 1 & \xi_4 \cr 0 & \one_{p-1}\end{bmatrix} \begin{bmatrix}1 & 0 \cr 0 & A_2\end{bmatrix},
\end{aligned}
\end{equation}
where $A_1, A_2$ are $(p-1)\times (p-1)$ unipotent upper triangular matrices and 
$\xi_1,\dots,\xi_4$ are $(p-1)$-vectors. Consequently, every element  
$V\in\N_+^{\Gamma^\er_1}$ can be uniquely factored as $V=T(V) K_1(V)= 
K_2(V) T(V)$ with $T(V)\in\N_+^{\tilde\Gamma^\er_1}$ and $K_1(V),K_2(V)\in \KK$. 
Recall that $(\bgamma^\er)^*\bgamma^\er$ acts on $\N_+$ as the
projection to $\N_+^{\Gamma_1^\er}$, which allows to define $T(V)$ and $K(V)$
for any $V\in\N_+$ as $T((\bgamma^\er)^*\bgamma^\er(V))$ and $K((\bgamma^\er)^*\bgamma^\er(V))$, respectively.

We start with the following relation between $\bar V^\er$ and $\bar{\tilde V}^\er$, 
as defined in Section~\ref{mainconstruction}. 

\begin{lemma}
\label{VtV}
$T(\bar V^\er)=\bar{\tilde V}^\er$.
\end{lemma}

\begin{proof}
Let us start with comparing $V^\er$ and $\tilde V^\er$. By construction, they are 
block-diagonal matrices with lower unipotent blocks that differ only in the block with the row and column set $\Delta=[k,m]$; let us call it the $\Delta$-block. For $V^\er$, this block coincides with the $\Delta$-block of the factor $U_-^{[k,m]}\in\N_-^{[k,m]}$ in the factorization 
$U_-=U_-^{[k,m]}U_LU_R$, see the proof of Theorem~\ref{bigthroughsmall} above. We denote this 
block $B$; it is a lower unipotent $p\times p$ matrix.

Let $s_{\{k,m-1\}}=s_{[m-2,k]}s_{[m-2,k+1]}\dots s_{m-2}$ be the reduced expression for 
$w_0^{[k,m-1]}$ analogous to one used in the proof of Theorem~\ref{bigthroughsmall}, 
and $s_{\{k+1,m\}}=s_{[m-1,k+1]}\allowbreak s_{[m-1,k+2]}\dots s_{m-1}$ be a similar reduced expression
for $w_0^{[k+1,m]}$. Note that both products $s_{\{k,m-1\}}s_{[m-1,k]}$ and
 $s_{\{k+1,m\}}s_{[k,m-1]}$ constitute reduced expressions for $w_0^{[k,m]}$. The two corresponding factorizations of $B$ are
\begin{equation}\label{Bblock}
B=\begin{dcases}\begin{bmatrix} 
\tilde B & 0\cr
0 & 1 \end{bmatrix} S^{-1}\quad&\text{for $\alpha=m-1$},\\
\begin{bmatrix} 
1 & 0\cr
0 &\tilde B \end{bmatrix} S\quad&\text{for $\alpha=k$},
\end{dcases}
\end{equation}
where $S$ is a lower unipotent bidiagonal $p\times p$ matrix with generically nonzero subdiagonal entries. In both cases $\tilde B$ is the $\tilde\Delta$-block for $\tilde V^\er$.

We can now compare $\bar V^\er$ and $\bar{\tilde V}^\er$. Let $\bar B$ and $\bar{\tilde B}$ be 
the corresponding $\Delta$- and $\tilde\Delta$-blocks. Recall that 
$\bar V^\er=(V^\er w_0^{[1,n]})_+$, so that $\bar B=(Bw_0^{[1,p]})_+$. A straightforward check 
shows that $S^{-1}w_0^{[1,p]}$ can be refactored as
\[
S^{-1}w_0^{[1,p]}=\begin{bmatrix} w_0^{[1,p-1]} & \star\cr 0 & 1\end{bmatrix} T_1
\]
with $T_1\in\B_-$. Consequently,~\eqref{Bblock} for $\alpha=m-1$ yields
\[
\bar B=\left(\begin{bmatrix} \tilde B & 0\cr 0 & 1\end{bmatrix}
\begin{bmatrix} w_0^{[1,p-1]}& \star\cr 0 & 1\end{bmatrix}T_1\right)_+=
\begin{bmatrix} \tilde Bw_0^{[1,p-1]} &\star \cr 0 & 1\end{bmatrix}_+=
\begin{bmatrix} \bar{\tilde B} &\star\cr 0 &1\end{bmatrix}.
\]
Taking into account that all other blocks for $\bar V^\er$ and $\bar{\tilde V}^\er$ are identical, we get the statement of Lemma for $\alpha=m-1$.

Further, a straightforward check shows that $Sw_0^{[1,p]}$ can be refactored as
\[
Sw_0^{[1,p]}=\begin{bmatrix} 1 & \star\cr 0& w_0^{[1,p-1]} \end{bmatrix} T_2
\]
with $T_2\in\B_-$. Consequently,~\eqref{Bblock} for $\alpha=k$ yields
\[
\bar B=\left(\begin{bmatrix} 1 & 0\cr 0& \tilde B \end{bmatrix}
\begin{bmatrix} 1&\star \cr 0 & w_0^{[1,p-1]}\end{bmatrix}T_2\right)_+=
\begin{bmatrix} 1&\star\cr 0 &\tilde Bw_0^{[1,p-1]} \end{bmatrix}_+=
\begin{bmatrix} 1&\star\cr 0 &\bar{\tilde B} \end{bmatrix}.
\]
Once again, all other blocks for $\bar V^\er$ and $\bar{\tilde V}^\er$ are identical, hence 
we get the statement of Lemma for $\alpha=k$.
\end{proof}

Our next step is to make relation~\eqref{XthrutX} 
between $Z$ and $\tilde Z$ more explicit.

\begin{lemma}
\label{HtHinv}
{\rm (i)} There exists $K\in\KK$ such that $\bgammar(C)=C\bgammar(K^{-1})$.

{\rm (ii)} Consequently, $C=\dots (\bgammar)^2(K)\bgammar(K)$.
\end{lemma}

\begin{proof}
(i) As explained above, any $V\in\N_+$ can be factored as $V=T(V)K(V)$. It follows from~\eqref{fourfactor} that 
\begin{equation}
\label{gatga}
\tilde\bgamma^\er(V)=\tilde\bgamma^\er(T(V))=\bgammar(T(V))=\bgammar(VK(V)^{-1}).
\end{equation}

Define $K_0,K_1,\dots\in\KK$ via $K_0=K(V)$, $K_j=K((\tilde\bgamma^\er)^j(V))$ for $j=1,\dots$. 
Then~\eqref{gatga} implies
\begin{multline*}
(\tilde\bgamma^\er)^2\left(T(V)\right)=\tilde\bgamma^\er\left(\bgammar(T(V))\right)=
\bgammar\left(T(\tilde\bgamma^\er(V))\right)=\bgammar\left(\tilde\bgamma^\er(V)K_1^{-1}\right)\\
=\bgammar\left(\bgammar(T(V))K_1^{-1}\right)=(\bgammar)^2(T(V))\bgammar(K_1^{-1})=
(\bgammar)^2(V)(\bgammar)^2(K_0^{-1})\bgammar(K_1^{-1}),
\end{multline*}
and more generally,
\[
(\tilde\bgamma^\er)^j\left(T(V)\right)=(\bgammar)^j(V)(\bgammar)^j(K_0^{-1})
(\bgammar)^{j-1}(K_1^{-1})\dots\bgammar(K_{j-1}^{-1}).
\]
Consequently,
\begin{equation}
\label{gatgaj}
\bgammar\left((\tilde\bgamma^\er)^j(T(V))\right)=(\tilde\bgamma^\er)^{j+1}(T(V))\bgammar(K_j).
\end{equation}

Recall that $\tilde H^\er(U)=\dots (\tilde\bgamma^\er)^2(\bar{\tilde V}^\er)\tilde\bgamma^\er
(\bar{\tilde V}^\er)$ and $\bar{\tilde V}^\er=T(\bar V^\er)$, hence $\tilde H^\er(U)^{-1}=
\tilde\bgamma^\er(T(V))(\tilde\bgamma^\er)^2(T(V))\dots$ for $V=(\bar V^\er)^{-1}$. 
Therefore,~\eqref{gatgaj} yields
\begin{multline*}
\bgammar(\tilde H^\er(U)^{-1})=(\tilde\bgamma^\er)^2(T(V))\bgammar(K_1)(\tilde\bgamma^\er)^3(T(V))
\bgammar(K_2)\dots\\=\left((\tilde\bgamma^\er)^2(T(V))(\tilde\bgamma^\er)^3(T(V))\dots\right)
\bgammar(K')=\tilde\bgamma^\er(\bar{\tilde V}^\er)\tilde H^\er(U)^{-1}\bgammar(K')
\end{multline*}
for some $K'\in\KK$ due to commutation rules~\eqref{fourfactor}. Further, the definition of $H^\er(U)$ in Section~\ref{mainconstruction} immediately yields 
$\bgammar(H^\er(U))=H^\er(U)\bgammar((\bar V^\er )^{-1})$. So, finally,
\[
\begin{aligned}
\bgammar(C)&=H^\er(U)\bgammar((\bar V^\er )^{-1})\tilde\bgamma^\er(\bar{\tilde V}^\er)\tilde H^\er(U)^{-1})\bgammar(K')\\
&=H^\er(U)\bgammar((\bar V^\er )^{-1}(\bar{\tilde V}^\er))\tilde H^\er(U)^{-1})\bgammar(K')\\
&=H^\er(U)\bgammar(K'')\tilde H^\er(U)^{-1})\bgammar(K')=H^\er(U)\tilde H^\er(U)^{-1})\bgammar(K^{-1})
=C\bgammar(K^{-1})
\end{aligned}
\]
for some $K\in\KK$; 
here the equality in the second line follows from~\eqref{gatga}, the first equality in the third line 
follows from Lemma~\ref{VtV} for $K''= (\bar V^\er )^{-1}(\bar{\tilde V}^\er)\in\KK$, and the second equality follows from commutation rules~\eqref{fourfactor}.

(ii) Indeed, by (i),  
$\bgammar(K)=\bgammar(C)^{-1}C$ and hence 
\[
(\bgammar)^j(K)=(\bgammar)^j(C)^{-1}(\bgammar)^{j-1}(C),
\]
so that $\dots(\bgammar)^2(K)\bgammar(K)=C$.
\end{proof}

To complete the proof of Theorem~\ref{inductionstep} we have to find an explicit expression for
the matrix $K$ in Lemma~\ref{HtHinv}, that is, to compute parameters $\xi_1,\dots,\xi_{p-1}$ in~\eqref{KK}. These parameters are determined uniquely via 
equations~\eqref{fafter},~\eqref{fbefore} and
the determinantal description of functions $f_{ij}(Z)$ and $f_{ij}(\tilde Z)$ for a particular
collection of $p-1$ pairs $(i,j)$. There are four cases to consider depending on whether the deleted 
root $\alpha$ is $k$ or $m-1$ and on whether or not $\gammar$ reverses the orientation of $[k,m-1]$. 

{\it Case 1\/:} $\alpha=k$, $\gammar$ preserves the orientation of $[k,m-1]$. 
Let $K=\linebreak\diag(\one_{k-1},\Xi,\one_{n-m})$ with  
$\Xi=\begin{bmatrix} 1 &\xi\cr 0 &\one_{p-1}\end{bmatrix}$. Denote $q=\gammar(k)$ and consider 
$f_{qj}(Z)$ for $j\in[n-p+2,n]$. Recall that $f_{qj}(Z)$ is the determinant of the principal trailing
submatrix $M$ of $\L(q,j)$ such that the entry in the upper left corner of $M$ is $z_{qj}$. It is easy to see that for $j$ as above, the top left block of $M$ is a $Y$-block, and since the orientation is preserved, the block immediately to the right of it is an $X$-block with the exit point $(k+n-j+1,1)$. By Remark~\ref{geninvariance}, this determinant does not change if the first of the above blocks is replaced by
the corresponding block of $\bgammar(N_+)Z$, and the second one by the corresponding blok of 
$N_+Z$. We choose $N_+=C^{-1}$, hence, as mentioned in the proof of Lemma~\ref{HtHinv}(ii), 
$\bgammar(N_+)=\bgammar(K)C^{-1}$. Consequently, by~\eqref{XthrutX}, the matrix in the first block is $\bgammar(K)\tilde Z$, and the matrix in the second block iz $\tilde Z$.

Consider the Laplace expansion of the matrix $M$ amended as explained above with respect to the first block column. In the obtained sum, apply Remark~\ref{mostgeneral} with $A=\tilde Z$ to all second factors and collect all the obtained expressions back to the unexpanded form. We thus obtain
\begin{equation}
\label{longf}
f_{qj}(Z)=\det\left[
\begin{array}{c:c}
\begin{matrix} \Xi\tilde Z_{[q,q+p-1]}^{[j,n]}\cr \zero \end{matrix} & 
(\tilde Z H^\ec(U)^{-1})_{[k,n]}^{[1,j-k]}\end{array}
\right]t_{k+n-j}(U),
\end{equation}  
where $\zero$ is a zero $(n-k-p+1)\times(n-j+1)$ matrix.
In a similar way,
\begin{equation}
\label{longtf}
\tilde f_{qj}(\tilde Z)=\det\left[
\begin{array}{c:c}
\begin{matrix} \tilde Z_{[q,q+p-1]}^{[j,n]}\cr \zero \end{matrix} & 
\begin{matrix} 0\ \dots\ 0\cr (\tilde Z H^\ec(U)^{-1})_{[k+1,n]}^{[1,j-k]}\end{matrix}
\end{array}
\right]\tilde t_{k+n-j}(U);
\end{equation}
note that in this case there is no need to amend the first two blocks of $\tilde M$. 
Further,~\eqref{bts} and~\eqref{fafter},~\eqref{fbefore} yield
\[
t_{k+n-j}(U)=
\begin{cases}
\tilde t_{k+n-j}(U)f_{k+1,1}(\tilde Z)  & \!\!\!\text{if $(k+1,1)$ is subordinate to $(k+n-j+1,1)$}\\
\tilde t_{k+n-j}(U)  &\!\!\!\text{otherwise}.
\end{cases}
\]
Consequently,~\eqref{longf},~\eqref{longtf} and~\eqref{fafter},~\eqref{fbefore}  yield
\[
\det\left[
\begin{array}{c:c}
\begin{matrix} \Xi\tilde Z_{[q,q+p-1]}^{[j,n]}\cr \zero \end{matrix} & 
(\tilde Z H^\ec(U)^{-1})_{[k,n]}^{[1,j-k]}\end{array}
\right]=
\det\left[
\begin{array}{c:c}
\begin{matrix} \tilde Z_{[q,q+p-1]}^{[j,n]}\cr \zero \end{matrix} & 
\begin{matrix} 0\ \dots\ 0\cr (\tilde Z H^\ec(U)^{-1})_{[k+1,n]}^{[1,j-k]}\end{matrix}
\end{array}
\right]
\]
for $j\in[n-p+2,n]$. In other words, let $P=(\tilde Z H^\ec(U)^{-1})_{[k+1,n]}^{[1,n-k]}$ be 
a $(n-k)\times(n-k)$ matrix and $v=(\tilde Z H^\ec(U)^{-1})_k^{[1,n-k]}$ be an $(n-k)$-vector, then all dense $(n-k+1)\times(n-k+1)$ minors of the $(n-k+1)\times(n-k+p-1)$ matrices
\[
\left[\begin{array}{c:c}
\begin{matrix} \Xi\tilde Z_{[q,q+p-1]}^{[n-p+2,n]}\cr \zero \end{matrix} & 
\begin{matrix} v \cr P\end{matrix}
\end{array}\right] 
\qquad\text{and}\qquad
\left[\begin{array}{c:c}
\begin{matrix} \tilde Z_{[q,q+p-1]}^{[n-p+2,n]}\cr \zero \end{matrix} & 
\begin{matrix} 0 \cr P\end{matrix}
\end{array}\right] 
\]
are equal, which gives $p-1$ linear equations for $\xi_1,\dots,\xi_{p-1}$. Define a unipotent upper
triangular $(n-k+1)\times(n-k+1)$ matrix $\Theta$ via $\Theta =\begin{bmatrix} 1 & vP^{-1}\cr
0 & \one_{n-k}\end{bmatrix}$, then multiplying the second matrix above on the left by $\Theta$ we preserve all dense $(n-k+1)\times(n-k+1)$ minors and obtain
\[
\left[\begin{array}{c:c}
\begin{matrix} \Theta_{[1,p]}^{[1,p]}\tilde Z_{[q,q+p-1]}^{[n-p+2,n]}\cr \zero \end{matrix} & 
\begin{matrix} v \cr P \end{matrix}
\end{array}\right].
\]
Consequently, all the minors in question are equal for $\Xi=\Theta_{[1,p]}^{[1,p]}$, which yields
\[
\xi_i=\frac{(-1)^{i-1}\det P^{(i)}}{\det P}, \qquad i\in[1,p-1],
\]
where $P^{(i)}$ is obtained from $P$ via replacing the $i$th row by $v$. Note that this solution remains valid if $\tilde Z_{[q,q+p-1]}^{[n-p+2,n]}$ above is replaced by an arbitrary $p\times(p-1)$ matrix 
$A$. It follows from the lower semicontinuity of the rank function that $\xi_i$ are defined uniquely, since for $A=\begin{bmatrix} \star \cr \one_{p-1}\end{bmatrix}$ the system of equations for $\xi_i$ is triangular and diagonal elements are minors of $P$. Finally, we invoke once again Theorem~\ref{bigthroughsmall} and Remark~\ref{mostgeneral} to conclude that the ratio above is equal to
\[
\frac{(-1)^{i-1}\det\tilde \L^{(i)}(k+1,1)(\tilde Z)}{\det\tilde \L(k+1,1)(\tilde Z)},
\]
where $\tilde \L^{(i)}(k+1,1)(\tilde Z)$ is obtained from $\tilde \L(k+1,1)(\tilde Z)$ via replacing the $i$th row of its upper leftmost block, which is a submatrix of $\tilde Z_{[k+1,n]}$, by the corresponding segment of the row $\tilde Z_k$ (cf.~the construction preceding 
Lemma~\ref{minorcorresp}). Consequently, $\xi_i$ are polynomials in $\tilde Z$ divided by 
$\tilde f_{k+1,1}(\tilde Z)$, as required.

{\it Case 2\/:} $\alpha=m-1$, $\gammar$ preserves the orientation of $[k,m-1]$. Let 
$K=\diag(\one_{k-1},\Xi,\one_{n-m})$ with  
$\Xi=\begin{bmatrix} \one_{p-1} &\xi^T\cr 0 & 1\end{bmatrix}$. Put $q=\gammar(k)$ as before and consider $f_{q+j,n-p+j+1}(Z)$ for $j\in[1,p-1]$. 
Note that $f_{q+j,n-p+j+1}(Z)$ is the determinant of the principal trailing
submatrix $M$ of $\L(m,1)$ such that the entry in the upper left corner of $M$ is $z_{q+j,n-p+j+1}$. It is easy to see that for $j$ as above, the top left block of $M$ is a $Y$-block, same as in the previous case, and since the orientation is preserved, the block immediately to the right of it is an $X$-block with the exit point $(m,1)$. Arguing as in Case 1, we arrive at the equality of the 
first $p-1$ leading minors of the $(n-k+1)\times(n-k+1)$ matrices
\[
\left[\!
\begin{array}{c:c}
\begin{matrix} \Xi\tilde Z_{[q,q+p-1]}^{[n-p+2,n]}\cr \zero \end{matrix} & 
(\tilde Z H^\ec(U)^{-1})_{[k,n]}^{[1,n-m+1]}\end{array}
\!\right]\text{and}
\begin{bmatrix} \tilde Z_{[q,q+p-2]}^{[n-p+2,n]}& \zero \cr
 \zero & (\tilde Z H^\ec(U)^{-1})_{[m,n]}^{[1,n-m+1]}\end{bmatrix},
\]
which yields a triangular system of linear equations on $\xi_1,\dots,\xi_{p-1}$. Write
\[
(\tilde Z H^\ec(U)^{-1})_{[k,n]}^{[1,n-m+1]}=\begin{bmatrix} P'\cr P\end{bmatrix} 
\]
where $P'$ consists of the upper $p-1$ rows, and $P$ is the remaining square submatrix, and 
define a unipotent upper triangular $(n-k+1)\times(n-k+1)$ matrix $\Theta$ via 
$\Theta=\begin{bmatrix} \one_{p-1} & P'P^{-1}\cr \zero & \one_{n-m+2}\end{bmatrix}$.
Multiplication of the second matrix above on the left by $\Theta$ preserves the leading minors and produces
\[
\left[
\begin{array}{c:c}
\begin{matrix} \Xi\tilde Z_{[q,q+p-1]}^{[n-p+2,n]}\cr \zero \end{matrix} & 
(\tilde Z H^\ec(U)^{-1})_{[k,n]}^{[1,n-m+1]}\end{array}
\right].
\]
Consequently, all the minors in question are equal for $\Xi=\Theta_{[1,p]}^{[1,p]}$, which yields
\[
\xi_i=\frac{\det P^{((i))}}{\det P}, \qquad i\in[1,p-1],
\]
where $P^{((i))}$ is obtained from $P$ via replacing its first row by the $i$th row of $P'$. 
The same reasoning as in Case~1 shows that $\xi_i$ are polynomials in $\tilde Z$ divided by
$\tilde f_{m1}(\tilde Z)$, as required.

{\it Cases 3 and 4}. These are the two cases when $\alpha=k$ or $\alpha=m-1$ and $\gammar$ reverses the orientation of $[k,m-1]$. The treatment of these cases is very similar to the treatment described above. In both cases the second block of the matrix $M$ is an $X^\adj$-block, so $\tilde Z H^\ec(U)^{-1}$ in all formulas should be replaced by $(\tilde Z H^\ec(U)^{-1})^\adj$. Further, 
finding $\bgammar(K)$ now involves the conjugation of the cofactor matrix by $w_0\J$. Note that 
conjugation by $w_0\J$ of the cofactor matrix of $\Xi$ used in Case~1  gives $\Xi$ used in 
Case~2. Consequently, the argument of Case~1 should be now used for $\alpha=m-1$, and the argument used in Case~2 should be used for $\alpha=k$.  
\end{proof}

Therefore, the proof of Theorem~\ref{laurent} is completed.
\end{proof}

\end{document}